\documentclass[12pt]{article}

\usepackage[utf8]{inputenc}
\usepackage[T1]{fontenc}
\usepackage[english]{babel}

\usepackage{amsmath,amssymb,amsthm}
\usepackage{authblk}
\usepackage{geometry}
\usepackage{graphicx}
\usepackage{microtype}
\usepackage{xcolor}
\usepackage{bbm}
\usepackage{mathrsfs}
\usepackage{enumitem}
\usepackage{hyperref}
\usepackage{mathtools}
\usepackage{fontawesome5}
\usepackage{tikz}
\usepackage{mathrsfs}
\usepackage{url}
\usepackage[numeric,initials,nobysame,msc-links,abbrev]{amsrefs}

\renewcommand{\MR}[1]{}

\renewcommand{\le}{\leqslant}
\renewcommand{\leq}{\leqslant}
\renewcommand{\ge}{\geqslant}
\renewcommand{\geq}{\geqslant}

\newcommand\renorm{\cR}

\usetikzlibrary{calc,decorations.markings,arrows.meta}
\tikzset{xarrow/.style={postaction={decorate},decoration={
        markings,
        mark=at position .6 with {\arrow[#1]{Stealth[]}}}},
    yarrow/.style={postaction={decorate},decoration={
        markings,
        mark=at position .7 with {\arrow[#1]{Stealth[sep]Stealth[]}}}},
    aarrow/.style={double distance=0.5pt}
}
\newcommand\closerel[3]{\draw[xarrow] ($#3 - #2 + #1$) -- #1;
  \draw[yarrow] ($#3 - #2 + #1$) -- #3;
}
\newcommand\straightrel[2]{\coordinate (straightrel) at ($0.5*#1 + 0.5*#2 + (-135:0.3)$);
  \draw[xarrow] (straightrel) -- #1;
  \draw[yarrow] (straightrel) -- #2;
}

\def\N{\mathbb N}
\def\Z{\mathbb Z}

\DeclareMathOperator{\supp}{supp}
\DeclareMathOperator{\argmin}{argmin}
\DeclareMathOperator{\var}{Var}

\newcommand{\cC}{\ensuremath{\mathcal C}}
\newcommand{\cD}{\ensuremath{\mathcal D}}
\newcommand{\cE}{\ensuremath{\mathcal E}}

\newcommand{\cG}{\ensuremath{\mathcal G}}

\newcommand{\cI}{\ensuremath{\mathcal I}}

\newcommand{\cL}{\ensuremath{\mathcal L}}

\newcommand{\cR}{\ensuremath{\mathcal R}}

\newcommand{\cX}{\ensuremath{\mathcal X}}

\newcommand{\cZ}{\ensuremath{\mathcal Z}}

\newcommand{\bbE}{{\ensuremath{\mathbb E}} }
\newcommand{\bbF}{{\ensuremath{\mathbb F}} }

\newcommand{\bbN}{{\ensuremath{\mathbb N}} }

\newcommand{\bbP}{{\ensuremath{\mathbb P}} }

\newcommand{\bbR}{{\ensuremath{\mathbb R}} }

\newcommand{\bbT}{{\ensuremath{\mathbb T}} }

\newcommand{\bbZ}{{\ensuremath{\mathbb Z}} }

\newtheorem{theorem}{Theorem}
\newtheorem{prop}[theorem]{Proposition}
\newtheorem{lemma}[theorem]{Lemma}
\newtheorem{claim}[theorem]{Claim}
\newtheorem{corollary}[theorem]{Corollary}
\theoremstyle{definition}
\newtheorem{definition}[theorem]{Definition}

\numberwithin{theorem}{section}

\newcommand{\1}{\mathbb{I}}
\newcommand{\trel}{T_{\mathrm{rel}}}
\newcommand{\dtv}{d_{\mathrm{TV}}}
\newcommand{\bone}{\mathbf{1}}
\newcommand{\bzero}{\mathbf{0}}

\providecommand{\symdiff}{\mathbin{\mathpalette\triangle\relax}}

\newcommand{\lon}{{\color{black}\faLightbulb}}
\newcommand{\loff}{{\faLightbulb[regular]}}
\newcommand{\soff}{{\faToggleOff}}

\newcommand{\son}{{\faToggleOn}}

\title{Super-Arrhenius relaxation of the triangular plaquette model in any dimension}
\author[1,2]{Laurent Bartholdi}
\author[1]{Ivailo Hartarsky}
\author[2]{Ivan Mitrofanov}
    \affil[1]{Universit\'e Lyon 1, Centrale Lyon, INSA Lyon, Universit\'e Jean Monnet, CNRS, ICJ UMR5208, 69622 Villeurbanne, France, \texttt{\{laurent,hartarsky\}@math.univ-lyon1.fr}}
\affil[2]{Université de Genève, \texttt{\{Laurent.Bartholdi,Ivan.Mitrofanov\}@unige.ch}}

\begin{document}
\maketitle

\begin{abstract} 
Consider the following plaquette model from statistical physics: a lamp lies at every vertex of the triangular lattice and a switch lies at every even vertex of the (bipartite) dual hexagonal lattice. Each switch toggles the three lamps on its face. The energy of a configuration is the number of ON lamps.

For the Glauber dynamics associated with the Gibbs measure defined by this Hamiltonian at any inverse temperature $\beta>0$, we show that, in any dimension $d\ge 2$, the infinite volume relaxation time satisfies
\[e^{\beta^2/C}/C \le \trel\le Ce^{e^{C\beta}}\]
for some $C>0$. Our result entails that the Gibbs measure is unique. The $e^{\beta^2}$ scaling was conjectured by Garrahan and Newman in 2000 and matches the behaviour of supercritical rooted kinetically constrained models such as the East model, thus recovering fragile glass phenomenology in the absence of kinetic constraints. More precisely, we show that, on a torus of side length $2^k$, when $\beta\to\infty$ and $k/\beta\to0$, we have $\trel=e^{2\beta k(1+o(1))}$, confirming a prediction of Newman and Moore. Quite surprisingly, however, we also prove that, on non-periodic finite domains of size $n\le e^{\beta/C}$ for large $C>0$, we have the much larger asymptotics $\ln\trel=\beta n^{\Theta(1)}$.

The main ingredients of the proofs are new results in extremal and enumerative combinatorics and rely on renormalisation ideas for the dynamics and its groundstates also known as the Ledrappier subshift. Our combinatorial results also have consequences in geometric group theory (more precisely to the complexity of the word problem for the Baumslag finitely presented group), ergodic theory and linear coding theory.
\end{abstract}


\tableofcontents
\section{Introduction}
\label{sec:intro}
One of the great achievements of statistical physics is the development of discrete, finite models whose behaviour closely follows, and ideally explains, physical phenomena. Foremost among these, the Ising model~\cites{Lenz20,Ising25} in dimension $d\ge 2$ exhibits phase transitions as manifested in magnetic bodies: there is a $\pm1$ \emph{spin} at every lattice vertex, and the system's energy is computed from spin values and nearest-neighbour interactions. A \emph{temperature} parameter dictates the likelihood of frustrated (${+}{-}$) bonds. From Glauber's dynamical perspective \cite{Glauber63}, the spins evolve at random, favouring lower energy states. The phase transition then manifests in the exponential slowdown of relaxation as temperature decreases past its critical point, displaying metastable phenomenology.

Several paradigms have been proposed to explain the more complex behaviour of glassy materials in condensed matter physics (see e.g.~\cite{Arceri21} for an overview of this vast domain). One hallmark of such materials is a \emph{super-Arrhenius} divergence of relaxation time scales at high inverse temperature $\beta>0$. These time scales may be viewed as quantifying the exponential decay rate of spin correlations with time or the viscosity of the system. For instance, a quadratic scaling $\trel\approx e^{\beta^2}$ can supply a good quantitative fit for the behaviour experimentally measured in so-called fragile glass materials (see \cite{Elmatad09}). In any case, the emergence of energy barriers whose size diverges when approaching the possibly degenerate phase transition point is to be expected.

One such paradigm is provided by a model introduced in 1989 in dimensions two and three by Heringa, Bl\"ote and Hoogland \cite{Heringa89} (see also \cite{Newman99}). 
It can be viewed as a natural generalisation of the one-dimensional Ising model to $d$ dimensions with $(d+1)$-body interactions. We concentrate on their \emph{triangular plaquette model}, and present it informally (see Section~\ref{subsec:model} for a more formal but strictly equivalent definition). Consider $\pm1$ spins at the vertices of the triangular lattice $\Lambda=\bbZ(1,0)+\bbZ(1/2,\sqrt3/2)$ whose energy is determined by twice the number of upwards-pointing triangles with odd number of $-1$ spins. In other words, the energy of the configuration $\sigma\in\{\pm1\}^\Lambda$ is written as a Hamiltonian
\[H(\sigma)=-\sum_{x\in\Lambda}\sigma_x\cdot\sigma_{x+(1,0)}\cdot\sigma_{x+(1/2,\sqrt3/2)},
\]
where the infinite sum is formal, but suffices to give sense to energy differences. The probability of a configuration $\sigma$ is given by the Gibbs measure
$\mu(\sigma)\propto\exp(-\beta H(\sigma))$. Spins evolve under \emph{Glauber dynamics}: each site of $\Lambda$ is equipped with a standard Poisson clock; when the clock at site $x$ rings, $\sigma$ is updated to the configuration $\sigma'$ which differs from $\sigma$ only at site $x$ with probability $1/(1+\exp(\beta(H(\sigma')-H(\sigma))))$. Numerical simulations of this dynamics can be used to estimate the relaxation time (see \eqref{eq:def:trel} for its formal definition) providing support for their glassy (super-Arrhenius) behaviour, for which we refer to~\cites{Newman99,Jack05,Garrahan02,Jack05a}, and \cite{Inack22} for a more recent study.\footnote{In the physics literature, the triangular plaquette model is sometimes referred to as `exactly solvable'. This refers to the fact \cite{Newman99}*{(10),(12)} that the partition function in \eqref{eq:def:H} can be computed exactly for tori of side length $2^k$ for integer $k$ and approximately otherwise, giving the same thermodynamic limit. As for kinetically constrained models (see Section~\ref{subsec:kcm}), such `trivial' static behavior does not imply that the associated dynamics is uninteresting or even tractable.}

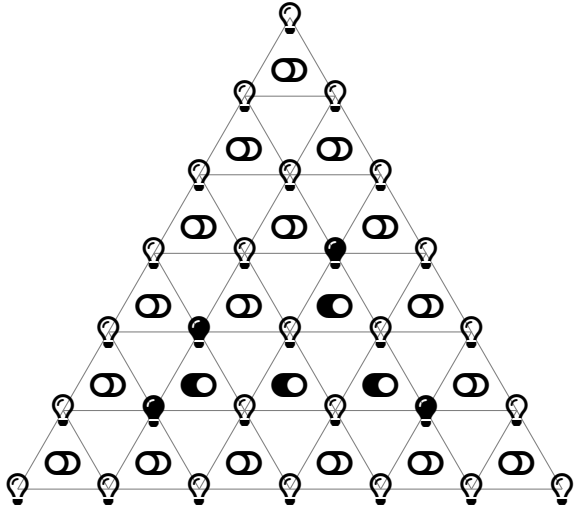
\begin{figure}
\centering\begin{tikzpicture}[x={(1cm,0cm)},y={(0.5cm,0.866cm)},scale=1.2,every circle/.style={scale=1/1.2,radius=3pt}]
  \foreach \x in {0,...,6} {
    \draw[gray] (\x, 0)--(0,\x);
    \draw[gray] (\x,0)--(\x,6-\x)--(0,6-\x);}
  \foreach \x in {0,...,6} {
  	\draw (\x,0) node{\loff};}
  \foreach \x in {0,...,5} {
  	\draw (0.33+\x,0.33) node{\soff};}
  \foreach \y in {1,...,6} {
  	\draw (0,\y) node{\loff};}
  \foreach \y in {1,...,5} {
  	\draw (0.33,0.33+\y) node{\soff};}
  \foreach \x in {1,...,5} {
  	\draw (\x,6-\x) node{\loff};}
  \foreach \x in {1,...,4} {
  	\draw (0.33+\x,5.33-\x) node{\soff};}

  	\draw (1.33,2.33) node{\soff};
  	\draw (1.33,3.33) node{\soff};
  	\draw (2,1) node{\loff};
  	\draw (3,1) node{\loff};
  	\draw (2,2) node{\loff};
  	\draw (3,2) node{\loff};
  	\draw (1,3) node{\loff};
  	\draw (1,4) node{\loff};

  	\draw (1.33,1.33) node{\son};
  	\draw (2.33,2.33) node{\son};
  	\draw (2.33,1.33) node{\son};
  	\draw (1,1) node{\lon};
  	\draw (1,2) node{\lon};
  	\draw (2,3) node{\lon};
  	\draw (3.33,1.33) node{\son};
  	\draw (4,1) node{\lon};
    \end{tikzpicture}
    \caption{A configuration of lamps and switches on the triangular lattice. This is the only figure in which the triangular lattice is actually triangular; later on, we will use the standard orthogonal simplex.}
    \label{fig:lamps}
\end{figure}

An equivalent description of the model, going back to \cite{Newman99}, is as follows, see Figure~\ref{fig:lamps}. The spins are switches (ON/OFF) at all vertices of the triangular lattice $\Lambda$; there are lamps (ON/OFF) at another triangular lattice, consisting of even positions of the dual lattice to $\Lambda$. Each lamp is ON when an odd number of the three switches adjacent to it are ON, so toggling a switch toggles three lamps simultaneously. Each ON lamp consumes an amount $\exp(2\beta)$ of energy. There are numerous popular games in which lamps are controlled by complicated configurations of switches, see e.g.~\cite{Marcarini26} for a recent account. The triangular plaquette model studied in this paper corresponds to a directed (North-East) version of the `Lights Out' game or its $3\times 3$ predecessor, the `Magic square' electronic game, from 1978. In general, the interest and difficulty of these games is related to the difficulty of finding canonical paths between given configurations.

Informally, our results are as follows (see Section~\ref{subsec:results} for formal statements). Firstly, if we work on a torus of side length $2^k$ with $k/\beta\to0$ and $\beta\to\infty$, then $\ln\trel\sim 2 k\beta$, proving the prediction of \cite{Newman99}, where precisely this setting was considered (namely, $k\to\infty$ after $\beta\to\infty$). Secondly, if we work in infinite volume, for any dimension $d\ge 2$, suitably large $C>0$ and any $\beta>0$ we have
\[e^{\beta^2/C}/C\le \trel\le Ce^{e^{C\beta}},\]
the lower bound confirming a prediction of \cite{Garrahan00}. While these bounds may seem distant, we show that, strikingly, in finite volume of size $n$ without periodic boundary condition, the upper bound is close to the truth for $\ln n/\beta$ small, while we expect the lower bound to be essentially sharp when $\ln n/\beta$ is large. The latter would imply $\trel=e^{\Theta(\beta^2)}$ in infinite volume, despite the much slower relaxation in finite volume. We are not aware of other models known to display such uncommon speed-up from constant to extremely slow and then to just slow relaxation, as the system size increases.

\subsection{Model and notation}\label{subsec:model}
Let $d$ be a positive integer. We will work on $\bbZ^d$ rather than the triangular lattice (and higher dimensional generalisations) for convenience of notation, but this is equivalent up to a linear transformation of the lattice. We denote the canonical basis of $\bbR^d$ by $e_1,\dots,e_d$. We introduce the (simplex) \emph{plaquette}
\begin{equation}
\label{eq:def:Td}
T_d=\{0,e_1,e_2,\dots,e_d\}.
\end{equation}
For any finite set $\Lambda$, let $\Omega_\Lambda=\{1,-1\}^\Lambda$ and $\Omega=\Omega_{\bbZ^d}$. We denote by $\bone_\Lambda\in\Omega_\Lambda$ the constant configuration and omit the subscript $\Lambda$ when it is clear from the context. For $\sigma\in\Omega_\Lambda$ and $x\in\Lambda$, we denote by $\sigma_x$ the value of $\sigma$ at site $x$. For $V\subseteq\Lambda$ and $\sigma\in\Omega_\Lambda$, we write $\sigma_V\in\Omega_V$ for the restriction of $\sigma$ to $V$. 
For $\sigma\in\Omega_\Lambda$ and finite $V\subseteq\Lambda$, we write $[\sigma]_V=\prod_{x\in V}\sigma_x$. 
For disjoint $\Lambda,\Lambda'\subseteq \bbZ^d$ and $\sigma\in\Omega_\Lambda,\sigma'\in\Omega_{\Lambda'}$, we write $\sigma\cdot\sigma'\in\Omega_{\Lambda\cup\Lambda'}$ for the configuration equal to $\sigma$ on $\Lambda$ and to $\sigma'$ on $\Lambda'$. For $x\in\Lambda$ and $\sigma\in\Omega_\Lambda$, we denote by $\sigma^x\in\Omega_\Lambda$ the configuration such that $\sigma^x_x=-\sigma_x$ and $\sigma_{\Lambda\setminus\{x\}}=\sigma_{\Lambda\setminus\{x\}}^x$, that is, the result of flipping the state of $x$. A function $f\colon\Omega\to\bbR$ is \emph{local}, if there exists a finite $V\subset\bbZ^d$ such that for any $\sigma,\sigma'\in\Omega$ such that $\sigma_V=\sigma'_V$, we have $f(\sigma)=f(\sigma')$. We call the smallest set $V$ as above the support of $f$ and denote it by $\supp f$. We denote the integral of a function $f$ with respect to a measure $\mu$ by $\mu(f)$. For $z\in\bbZ^d$, $\lambda\in\bbZ$ and $X,Y\subseteq\bbZ^d$, we use the standard notation $z+ X=\{z+x\colon x\in X\}$, $\lambda X=\{\lambda x\colon x\in X\}$, $X+Y=\{x+y\colon x\in X,y\in Y\}$, etc.

For finite $\Lambda$ and $\tau\in\Omega_{\bbZ^d\setminus \Lambda}$ (viewed as a \emph{boundary condition}) and $\sigma\in\Omega_\Lambda$, we define the \emph{Hamiltonian}, \emph{partition function} and \emph{Boltzmann measure} by
\begin{align}
\label{eq:def:H}
    H_\Lambda^\tau(\sigma)&{}=-\sum_{x\in\Lambda-T_d}[\sigma\cdot\tau]_{x+T_d},&Z_\Lambda^\tau&{}=\sum_{\sigma\in\Omega_\Lambda}e^{-\beta H_\Lambda^\tau(\sigma)},&\mu_{\Lambda}^{\tau}(\sigma)&{}=\frac{e^{-\beta H_\Lambda^\tau(\sigma)}}{Z_\Lambda^\tau},
\end{align}
keeping the \emph{inverse temperature} parameter $\beta>0$ implicit. When $\Lambda=\{x\}$ is a singleton, we write $x$ instead of $\Lambda$ in the above notation. The degenerate cases $\beta=0$ and $\beta=\infty$ (which we do not consider unless otherwise stated) correspond to $\mu_\Lambda^\tau$ being the uniform measure on $\Omega_\Lambda$ and on the \emph{groundstates} $\argmin H_\Lambda^\tau$, respectively. Notice that, for any local $f$ and finite $\Lambda$, $\mu_\Lambda(f)$ is a function from $\Omega_{\bbZ^d\setminus\Lambda}$ to $\bbR$, which we identify with a function from $\Omega$ to $\bbR$ by projection. We denote the variance associated to $\mu_\Lambda^\tau$ by $\var_\Lambda^\tau:f\mapsto\mu_\Lambda^\tau(f^2)-(\mu_\Lambda^\tau(f))^2$.

We denote by $\mu$ any \emph{infinite volume Gibbs measure}, that is, any measure on $\Omega$ such that, for all finite $\Lambda$ and local function $f$ with $\supp f\subseteq\Lambda$, we have
\begin{equation}
\label{eq:Gibbs:DLR}
\int \mu_\Lambda^{\sigma_{\bbZ^d\setminus\Lambda}}(f)\mathrm d\mu(\sigma)=\mu(f).\end{equation}
It follows from our results that $\mu$ is unique for any $\beta>0$, but this is not clear a priori. In order to lighten notation, we set 
\[\mu_\Lambda(f)\coloneqq\mu_\Lambda^{\sigma_{\bbZ^d\setminus\Lambda}}(f),\]
whenever the configuration $\sigma$ is clear from the context, e.g.\ \eqref{eq:Gibbs:DLR} reads $\mu(\mu_\Lambda(f))=\mu(f)$.

The \emph{Glauber dynamics} in a domain $\Lambda\subseteq\bbZ^d$ with boundary condition $\tau\in\Omega_{\bbZ^d\setminus\Lambda}$ is the (continuous time) Markov process whose \emph{generator} acts on local functions $f\colon\Omega_\Lambda\to\bbR$ via 
\[\cL_\Lambda^\tau f(\sigma)=\sum_{x\in\Lambda}\left(\mu_x^{\tau\cdot\sigma_{\Lambda\setminus\{x\}}}(f)-f\right)(\sigma).\]
Note that, this generator is uniquely defined by the Hamiltonian and does not depend on the Gibbs measure $\mu$ in case there are multiple. Equivalently, one may view this process via the following more intuitive \emph{graphical construction}. Each site of $\Lambda$ is equipped with a Poisson clock, which rings at independent intervals of time with exponential distribution of mean 1. When the clock at site $x$ rings, the current configuration $\sigma$ is updated to either $\sigma$ or $\sigma^x$ with probability such that $\sigma_x$ has distribution $\mu_x^{\tau\cdot\sigma_{\Lambda\setminus \{x\}}}$ after the update. It is classical that such finite range dynamics are well-defined in infinite volume (see e.g.\ \cite{Liggett05}) and the generator $\cL=\cL_{\bbZ^d}$ extends to a self-adjoint operator on $L^2(\Omega,\mu)$.

The \emph{Dirichlet form} associated to the generator is given by
\begin{equation}
\label{eq:def:D}
\cD_\Lambda^\tau(f)=-\mu_\Lambda^\tau(f\cL_\Lambda^\tau f)=\sum_{x\in\Lambda}\mu_\Lambda^\tau(\var_x(f)),\end{equation}
where, as before, $\var_x$ denotes the variance with respect to $\mu_{\{x\}}^{\tau\cdot\sigma_{\bbZ^d\setminus\Lambda}}$ with $\sigma$ the configuration over which $\mu_\Lambda^\tau$ runs. We may now define the \emph{relaxation time} as
\begin{equation}
\label{eq:def:trel}
\trel^{\Lambda,\tau}=\left(\inf\left\{\frac{\cD_\Lambda^\tau(f)}{\var_\Lambda^\tau(f)}:f\text{ local},\supp f\subseteq\Lambda,\var_\Lambda^\tau(f)\neq 0\right\}\right)^{-1}\in[0,\infty]\end{equation}
and $\trel=\trel^{\bbZ^d}$. This fundamental quantity has many equivalent definitions, in particular its inverse is known as the spectral gap of $\cL_\Lambda^\tau$. It governs the exponential rate at which correlations decay in the system. We direct the reader to \cite{Martinelli99} for background on Glauber dynamics though we will not assume familiarity with this classical reference.

All the above quantities also apply with periodic boundary condition as follows. For positive integer $n$, let 
\[\bbT_n=(\bbZ/n\bbZ)^d\]
denote the torus of dimension $d$ and (linear) size $n$. In this case, no boundary condition needs to be specified, so we write $\mu_{\bbT_n}$, $\cL_{\bbT_n}$ $\cD_{\bbT_n}$, $\trel^{\bbT_n}$ for the corresponding Boltzmann measure, generator, Dirichlet form and relaxation time defined as above.

We denote by $\Lambda_{d,n}$ the simplex of side length $n\in\bbN$:
\begin{equation}\label{eq:def:Lambda:dn}
  \Lambda_{d,n}\coloneqq\left\{(x_1,\dots,x_d)\in \N^d \mid x_1+\dots+x_d\leqslant n\right\}\subset \Z^d.
\end{equation}
(throughout this text $\N$ denotes the natural numbers and contains $0$) and we call \emph{simplex} a translate $x+\Lambda_{d,n}$ for some $x\in\bbZ^d$. Thus, $\Lambda_{d,0}=\{0\}$ and $\Lambda_{d,1}=T_d$ from \eqref{eq:def:Td}.

\subsection{Results}
\label{subsec:results}

We are now ready to state our main results. We begin with those pertaining to infinite-volume configurations.
\begin{theorem}[Infinite volume relaxation time]
\label{th:main}
For any $d\ge 2$ there exists $C=C(d)\ge1$ such that, for any inverse temperature $\beta>0$, we have 
\[e^{\beta^2/C}/C\le \trel\le Ce^{e^{C\beta}}
\]
and the Gibbs state $\mu$ is unique.
\end{theorem}
From the physics perspective, the lower bound establishing slowness of the dynamics is the most important and matches the prediction of Garrahan and Newman~\cite{Garrahan00}. The fact that $\trel<\infty$ when $d=2$ was established in~\cite{Chleboun17}, while a quantitative upper bound of order $e^{e^{C\beta}}$ follows immediately from that work by~\cite{Martinelli99}*{Theorem~3.8 and Proposition~4.10}, since strong spatial mixing is known to imply finite relaxation time~\cite{Martinelli94a} (see \cite{Ott25} for more recent work on mixing conditions in two dimensions). The approach of \cite{Chleboun17} does not generalise to higher dimension and differs from ours even in two dimensions. The uniqueness of Gibbs state implies that the model does not undergo a phase transition.

Turning to finite volume, we start with (our formal interpretation of) the conjecture of \cite{Newman99}*{Figure~3 and (13)} concerning tori of subcritical size $2^k$.
\begin{theorem}[Subcritical asymptotics for dyadic tori]
\label{th:2k:asymptotics}
    In dimension $d=2$, we have
  \[\lim_{\substack{\beta\to \infty\\k/\beta\to0}}\frac{\ln\trel^{\bbT_{2^k}}}{k\beta}=2.\]
\end{theorem}
Note that the size being constraint to be a power of $2$ is not a technical restriction. Indeed, as noted already in \cite{Martin84}, even the groundstates of the Hamiltonian behave very irregularly with the system size (also see \cite{Devakul19} for a more recent interpretation in terms of fractal symmetry topological phases).

We complement this result by a much larger lower bound essentially matching the upper bound in Theorem~\ref{th:main}, for a domain $\Lambda_{d,n}$ of subcritical size. It does not seem to have been predicted in the physics studies and somewhat undermines physical intuition.
\begin{theorem}[Stronger bottleneck]
\label{th:finite}
  For any $d\ge 2$ there exists $C=C(d)\ge 1$ such that the following holds. If $n\le e^{\beta/C}$, then 
  \[e^{\beta n^{1/C}}/C\le \trel^{-\Lambda_{d,n},\bone}\le Ce^{Cn^{d}(\beta+1)}.\]
\end{theorem}
We note however that, because of the subcriticality assumption $n\le e^{\beta/C}$, despite \eqref{eq:trel:small:scales} below, Theorem~\ref{th:finite} does not imply that the same lower bound holds in infinite volume. Indeed, we rather expect a highly unusual speed-up to take place in larger volumes and the lower bound of Theorem~\ref{th:main} to be closer to the truth. The result of Theorem~\ref{th:finite} is in fact not restricted to this particular shape of domain (e.g.\ $\{1,\dots,n\}^d$ would work) or boundary condition, but is specific to the size specified in the statement.

Further consequences of our study in geometric group theory, ergodic theory and error-correcting linear codes will be presented directly in Sections~\ref{ss:applications}, \ref{sec:Ledrappier} and \ref{sec:coding} respectively, once the necessary background is recalled.

\subsection{Further motivation and context}
\subsubsection{From kinetically constrained models to plaquette models}
\label{subsec:kcm}
We mention another class of models that have been proposed to explain the behaviour of glassy materials, the \emph{kinetically constrained models}. An example in dimension one is the \emph{East model} introduced by J\"ackle and Eisinger \cite{Jackle91}: a site is occupied ($-1$) or empty ($+1$); and a site may switch its occupation status as in Glauber dynamics, but only when its East (right) neighbour is occupied. Aldous and Diaconis~\cite{Aldous02} showed, in their pioneering work (see \cite{Faggionato13} for an overview), that this model also follows a quadratic super-Arrhenius law $\trel=e^{(C+o(1))\beta^2}$ as $\beta\to\infty$ with $C=1/(2\ln2)$, see \cite{Cancrini08}.

A major drawback of these kinetically constrained models is that their constraints are imposed ad hoc to reflect the intuition of dynamical facilitation and caging effects, without any microscopic justification for the emergence of such constraints. This was an important motivation to introduce a more physically plausible model in the form of plaquettes. While widely studied in physics (nowadays, quantum versions of these models are actively investigated in physics, see \cites{Sfairopoulos23,Sfairopoulos25}) plaquette models remained largely intractable mathematically until the work of Chleboun, Faggionato, Martinelli and Toninelli \cite{Chleboun17}. One of the aims of our work is to provide the first rigorous justification that glassy phenomenology similar to the one of kinetically constrained models is indeed present in plaquette models.

\subsubsection{Universality}
A broad class of kinetically constrained models is obtained by replacing the East constraint above by an arbitrary local translation-invariant monotone constraint. Recent years have seen the emergence of a rigorous universality theory for these models in two dimensions. While referring the reader to \cite{HartarskybookKCM} for a detailed exposition of this theory (also see \cites{Ritort03,Garrahan11} for physics reviews predating it), let us highlight three important universality classes. `Supercritical unrooted' kinetically constrained models satisfy $\trel=e^{\Theta(\beta)}$; `supercritical rooted' ones (such as the East model) satisfy $\trel=e^{\Theta(\beta^2)}$, while `critical' ones satisfy $\trel=e^{e^{\Theta(\beta)}}$, as $\beta\to 0$.

Until present, only the square plaquette model (replacing $T_d$ by $\{(0,0),(1,0),(1,1),(0,1)\}$) has been shown to be a `supercritical unrooted' plaquette model \cites{Chleboun21,Chleboun20,Espriu04,Mueller17}. We establish the triangular plaquette model as a first `supercritical rooted' or `critical' plaquette model. Some of our proofs already adapt naturally to other models from \cite{Heringa89}. For instance, the crucial renormalisation of Lemma~\ref{lem:renorm} below applies to \emph{any} finite set instead of $T_d$, and one can establish a quadratic super-Arrhenius lower bound for any odd-size subset of the cube $\{0,1\}^d$. However, we prefer to restrict the exposition to a single model at present and leave its generalisation to future investigation.

\subsubsection{The Ledrappier subshift}
The triangular plaquette model is also closely related to important notions in quite different fields. In ergodic theory, the \emph{Ledrappier subshift}, or \emph{$3$-dot subshift}~\cite{Ledrappier78} consists in spin configurations in which every triangle contains an odd number of $+1$; equivalently, in minimal-energy configurations also known as \emph{groundstates}. It is a prominent example of $\Z^2$-subshift which is mixing but not $2$-mixing. We refer to Section~\ref{sec:Ledrappier} for more background and consequences of our results to the description of configurations with small support.

\subsubsection{Linear codes}

The restriction of the Ledrappier subshift to the simplex $\Lambda_{d,n}$ is a
$\binom{n+d-1}{d-1}$-dimensional linear subspace of $\mathbb{F}_2^{\binom{n+d}{d}}$. This subspace can be
viewed as a linear code with parameters 
$$\left[\binom{n+d}{d},\binom{n+d-1}{d-1},n+1\right].$$ 
Here, the third
parameter is the minimum Hamming weight (the number of ones) of a
nonzero codeword.
In two dimensions, the codewords are known
as Steinhaus triangles or Pascal Boolean triangles (see Figure~\ref{fig:Steinhaus}).

\begin{figure}\centering
\begin{tikzpicture}[x=6mm,y=6mm,line width=0.5pt]
	\foreach \row [count=\i from 0] in {
		{1},
		{0,1},
		{0,0,1},
		{0,0,0,1},
		{1,1,1,1,0},
		{0,1,0,1,0,0},
		{1,1,0,0,1,1,1},
		{1,0,1,1,1,0,1,0},
	} {
		\foreach \bit [count=\j from 0] in \row {
			\pgfmathtruncatemacro{\fillpercent}{100*\bit}
			\draw[fill=black!\fillpercent]
			({\j-\i/2},{-\i*sqrt(3)/2}) circle[radius=1.7mm];
		}
	}
\end{tikzpicture}
\caption{A Steinhaus triangle\label{fig:Steinhaus}}
\end{figure}

Suppose that a codeword is transmitted over a binary symmetric channel, with each bit independently flipped with probability $p$.
Theorem~\ref{th:upper} implies that, for every $d$ and every fixed $p<1/2$, there
exists a decoding procedure that recovers the original codeword with
an error probability tending to zero as $n\to\infty$.

We refer the reader to Section~\ref{sec:coding} for a more detailed explanation.

\subsubsection{Baumslag's metabelian group}
Within group theory, Baumslag \cite{Baumslag72} gave an example of finitely-presented metabelian group $\Gamma$ containing the wreath product $\Z\wr\Z=\bigoplus_{\Z}\Z\rtimes\Z$. A variant contains the ``lamplighter group'', namely the wreath product $(\Z/2\Z)\wr\Z$. It may be given by the presentation
\begin{equation}\label{eq:Baumslag}
\Gamma=\left\langle a,x,y\mid xy=yx,a^2=a\cdot x^{-1}a x\cdot y^{-1}a y=1\right\rangle\end{equation}
or as $\Gamma=S\rtimes\Z^2$ where $S$ consists in equivalence classes of finitely supported lamp configurations (in other words, those that differ by an element of the Ledrappier subshift). We refer to Section~\ref{ss:applications} for more background and consequences or our results to the word problem of $\Gamma$.

\subsection{Outline of the paper}
In Section~\ref{ss:combinatorial}, we present three combinatorial results, corresponding respectively to: the lower bounds of Theorems~\ref{th:main} and~\ref{th:2k:asymptotics}; the lower bound of Theorem~\ref{th:finite}; the upper bound of Theorem~\ref{th:main}. 
The first (Theorem~\ref{thm:lower}) may be stated as follows in terms of lamps and switches: converting the configuration with a lone ON lamp at the origin into a configuration with no ON lamp in a ball of radius $n$ requires, at some intermediate step, to have at least $\log_2 n$ ON lamps in that ball; so the energy of the system must rise significantly at some intermediate step. Since this is improbable, the relaxation time is bounded from below.

The second combinatorial result (Theorem~\ref{th:complexity:2}) asserts that there exist configurations in a region of size $n$ which can be reduced to the constant $+1$ configuration, but require the creation of $n^{\Theta(1)}$ additional ON lamps to do so; we call these \emph{entangled} configurations.

The third combinatorial result (Theorem~\ref{th:upper}) considers \emph{cycles}: switch configurations whose associated lamp configuration has all lamps OFF in a given region. We obtain estimates on a generating function counting cycles; namely, we show that, in a suitably large simplex, there are not too many cycles, when weighted by their number of $-1$ switches. It follows that the partition functions associated with two different completions outside a large simplex differ very little.

In Section~\ref{ss:probabilistic}, we provide the probabilistic arguments that make use of the combinatorial results mentioned above. The proof of Theorems~\ref{th:main}, \ref{th:2k:asymptotics} and \ref{th:finite} is completed there. Note however that the upper bounds of Theorems~\ref{th:2k:asymptotics} and~\ref{th:finite} are proved directly (by the bisection technique or canonical paths) without input from Section~\ref{ss:combinatorial}.

Finally, in Sections~\ref{ss:applications}, \ref{sec:Ledrappier} and \ref{sec:coding}, we elaborate on the applications of the combinatorial results to group theory, subshifts and linear error-correcting codes respectively.

\section{Combinatorial results}\label{ss:combinatorial}
In this section we focus on three purely deterministic statements of combinatorial nature which will be the key to the main results.

\subsection{Extremal combinatorics: logarithmic bottleneck for moving a single ON lamp}
\label{subsec:distance}

In this subsection, we only consider the lamp configurations; every lamp may be ON ($1$) or OFF ($0$). We call elements of $\{0,1\}^{\Z^d}$ simply \emph{configurations}, and we identify them with subsets of $\Z^d$ by considering their support. We add configurations pointwise modulo $2$; this corresponds to symmetric difference.

Recall the plaquette $T_d=\{0,\, e_1,\, e_2,\dots,e_d\}$ from \eqref{eq:def:Td}. In one move we may choose a point $x\in\Z^d$ and toggle the lamps at all points of the set $x+T_d$ (i.e.\ add its indicator function modulo $2$).
Any shift $x+T_d$ of this set is called a \emph{plaquette}.

\begin{definition}[{Chain}]\label{def:chain}
A sequence of configurations $(W_i)_{i=0}^n$ is called a \emph{chain} if for every $i$ the symmetric difference
\[
W_i\symdiff W_{i+1}
\]
is a plaquette or $\varnothing$.
\end{definition}

\medskip

We will use the closed $\ell_{\infty}$-ball in $\Z^d$:
\begin{equation}\label{eq:ball}
B(x,R) \coloneqq \left\{y\in\Z^d\mid \|x-y\|_{\infty} \le R\right\}.
\end{equation}

\begin{theorem}[{Combinatorial bottleneck}]\label{thm:lower}
Denote the origin of $\Z^d$ by $O$.
Consider $k\in\N$, and set $B_k = B(O,3d\,2^{k})$.
Let $(W_0,\dots,W_n)$ be a chain such that $W_0=\{O\}$ and $W_n\cap B_k=\varnothing$.
Then for some index $i$ we have
\[
|W_i\cap B_k| > k.
\]
\end{theorem}

We note that this result is sharp up to an additive constant in the last display when $d=2$ and up to a multiplicative one for $d>2$. This was already noted for $d=2$ in \cite{Newman99}, and can be easily seen by induction, using the configuration of Figure~\ref{fig:Sierpinski} and  \eqref{eq:Sierpinski}. Since we will not use this fact, we omit the details.

\subsubsection{Renormalisation}

For $x\in\Z^d$ define the \emph{inverted plaquette}
\begin{equation}
\label{eq:def:inverted}
K_x \coloneqq 2x - T_d = \left\{2x,\, 2x-e_1,\, 2x-e_2, \dots, 2x-e_d\right\}.
\end{equation}

\begin{lemma}[{Types of inverted plaquettes}]
\label{lem:intersection}\leavevmode
\begin{enumerate}
\item\label{item:1} If $x\neq y$ are in $\Z^d$, then $K_x\cap K_y=\varnothing$.
\item\label{item:2} A plaquette intersects inverted plaquettes in exactly one of the following ways:
\begin{enumerate}
\item\label{item:intersection:a} it intersects none;
\item\label{item:intersection:b} it intersects exactly one inverted plaquette and does so in exactly two points;
\item\label{item:intersection:c} it has the form $2x+T_d$ for some $x\in\Z^d$ and intersects exactly the $d+1$ inverted plaquettes
$K_x,K_{x+e_1},K_{x+e_2},\dots, K_{x+e_d}$, each in one point.
\end{enumerate}
\end{enumerate}
\end{lemma}

\begin{proof}
\ref{item:1}. If $2x-t = 2y-s$ for some $t,s\in T_d$, then $2(x-y)=t-s$.
But $t-s$ has coordinates in $\{-1,0,1\}$, hence cannot equal a nonzero even vector (that is, a vector in $2\Z^d$). Therefore $t=s$ and $x=y$.

\ref{item:2}. Modulo $2$ we have $2y\equiv 0$ and $2y-e_i\equiv e_i$, hence
\[
p\in\bigcup_y K_y \iff p \bmod 2 \in T_d.
\]

For a plaquette $B=x+T_d$ the $(d+1)$ residues are $(x\bmod 2)+T_d$, so the number of points of $B$ lying in $\bigcup_y K_y$ is $(d+1)$ iff $x$ is even, $2$ iff $x$ has exactly one or two odd coordinates and otherwise $0$.

Thus we get cases \ref{item:intersection:c}, \ref{item:intersection:b}, \ref{item:intersection:a}. If $x=2u{\in2\Z^d}$, then
\[
2u\in K_u,\qquad 2u+e_i\in K_{u+e_i}\ (i=1,\dots, d),
\]
so $2u+T_d$ meets $K_u,K_{u+e_1},\dots,K_{u+e_d}$. By \ref{item:1}. of the proof, the intersections of $2x+T_d$ with each of these inverted plaquettes are disjoint, so they consist of one point each.

If $x = 2u-e_i$, then $x+T_d$ has two common points $\{2u,2u-e_i\}$ with $K_{u}$. By \ref{item:1}. of the proof, $x+T_d$ has no intersection with any $K_v$ for $v\in\Z^d\setminus\{u\}$.

If $x=2u-e_i-e_j$ for $i\neq j$, then $x+T_d$ has two common points $\{2u-e_i,2u-e_j\}$ with $K_u$. By \ref{item:1}. of the proof, there $x+T_d$ has no intersection with any $K_v$ for $v\in\Z^d\setminus\{u\}$.

In the other cases, $x$ and $x+e_i$ have at least two odd coordinates, so cannot belong to an inverted plaquette. 
\end{proof}

\begin{figure}
    \centering
    \begin{tikzpicture}[x=0.5cm,y=0.5cm,radius=3pt]
\foreach \x in {0,...,8} {
    \draw[gray] (\x, 0)--(0,\x);
    \draw[gray] (\x,0)--(\x,8-\x)--(0,8-\x);}
\foreach \x in {0,...,4} {
  	\foreach \y in {0,...,\the\numexpr4-\x\relax} {
  	\fill[color=red,opacity=0.5] (2*\x,2*\y)--(2*\x,2*\y-1)--(2*\x-1,2*\y)--cycle;}}
  	\fill[color=blue,opacity=0.5] (2,2)--(3,2)--(2,3)--cycle;
  	\fill[color=black,opacity=0.4] (1,4)--(2,4)--(1,5)--cycle;
  	\fill (1,1) circle;
  	\fill (1,2) circle;
  	\fill (2,3) circle;
  	\fill (3,2) circle;
  	\fill (3,1) circle;
  	\fill (0,4) circle;
  	\fill (3,0) circle;
  	\fill (3,4) circle;
  	\fill (5,3) circle;
  	\fill (2,4) circle;
  	\fill (6,1) circle;
  	\fill (6,2) circle;
  	\fill (5,2) circle;
  	\end{tikzpicture}
  	\qquad
  	\begin{tikzpicture}[
  	x=1cm,y=1cm,radius=3pt]
  	\draw[->,thick,>=Stealth] (0,0) -- node[below] {$e_1$} (1,0);
  	\draw[->,thick,>=Stealth] (0,0) -- node[left] {$e_2$} (0,1);
\foreach \x in {0,...,4} {
    \draw[gray] (\x, 0)--(0,\x);
    \draw[gray] (\x,0)--(\x,4-\x)--(0,4-\x);}
    \phantom{\draw (0,0)--(0,-0.5);}
  	\fill[color=blue,opacity=0.5] (1,1)--(2,1)--(1,2)--cycle;
    \fill (0,2) circle;
  	\fill (1,1) circle;
  	\fill (2,0) circle;
  	\fill (2,1) circle;
  	\fill (2,2) circle;
  	\fill (3,1) circle;
  	\end{tikzpicture}
    \caption{Illustration of the renormalisation map $\renorm$ of \eqref{eq:renorm}, with $\Lambda_{2,8}$ on the left and $\Lambda_{2,4}$ on the right, see \eqref{eq:def:Lambda:dn}. The red inverted plaquettes (see \eqref{eq:def:inverted}) on the left are contracted to the vertices on the right. The blue spread plaquette (see Definition~\ref{def:spread}) corresponds to the blue plaquette on the right, while the gray non-spread plaquette has no effect on the renormalised configuration.}
    \label{fig:renorm}
\end{figure}
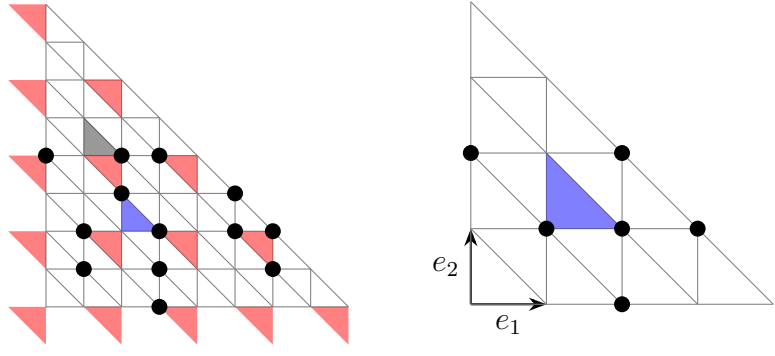

We now define the \emph{renormalisation map} $\renorm$ from configurations to configurations as illustrated in Figure~\ref{fig:renorm}.
Given a configuration $W$, define $\renorm(W)$ at a point $x\in\Z^d$ as the parity of the number of ON lamps of $W$ inside $K_x$:
\begin{equation}
\label{eq:renorm}
\renorm(W)(x) \coloneqq \sum_{p\in K_x} W(p)\pmod 2.
\end{equation}
(this operator already appears, for $d=2$, in the proof of \cite{Arenas-Carmona08}*{Theorem 7.1}). From Lemma~\ref{lem:intersection} we obtain the following key fact.

\begin{lemma}[Renormalisation]\label{lem:renorm}
If $(W_0,\dots,W_n)$ is a chain, then $\bigl(\renorm(W_0),\dots,\renorm(W_n)\bigr)$ is also a chain.\qed
\end{lemma}
Let us note that Lemma~\ref{lem:renorm} applies without change to any plaquette shape instead of $T_d$, although we will not need this fact.

In Lemma~\ref{lem:intersection}, we classified the possible intersection types between plaquettes and inverted plaquettes. We will now look at intersections with simplices, recalling \eqref{eq:def:Lambda:dn}.

\begin{lemma}\label{lem:intersect2}
\leavevmode
\begin{enumerate}
\item\label{item:intersect2:1} A plaquette is either entirely contained in a simplex, or it intersects the simplex in at most one point.
\item\label{item:intersect2:2} Two distinct plaquettes intersect in at most one point.
\end{enumerate}
\end{lemma}

\begin{proof}
\ref{item:intersect2:1}. It suffices to consider a simplex of the form $\Lambda_{d,m}$.
Suppose $T_d+x$ has a nonempty intersection with $\Lambda_{d,m}$ but is not contained in $\Lambda_{d,m}$. There are two cases.

Assume $x\in \Lambda_{d,m}$. Then $x_1+\dots+x_d = m$ (otherwise $x+e_i\in \Lambda_{d,m}$ for all $i$), and therefore
\[
(T_d+x)\cap\Lambda_{d,m} = \{x\}.
\]

Assume $x\notin \Lambda_{d,m}$ and $x+e_i\in \Lambda_{d,m}$ for some $i$. Then $x_i = -1$, and
\[
(T_d+x)\cap \Lambda_{d,m} = \{x+e_i\}.
\]

\ref{item:intersect2:2}. This is immediate from \ref{item:intersect2:1}., since the plaquette $T_d+x$ is precisely $x+\Lambda_{d,1}$.
\end{proof}

\subsubsection{Proof of Theorem \ref{thm:lower}} 

We construct recursively a sequence of nested simplices $U_k$.
We start with $U_0\coloneqq\{O\}$. 
Assuming that the simplex $U_k$ has been constructed, define
\[
U'_{k+1}\coloneqq \bigcup_{x\in U_{k}}K_x.
\]
Let $U_{k+1}$ be the smallest simplex that contains, in full, every plaquette that intersects $U'_{k+1}$.

\begin{prop}[{Inductive statement of the combinatorial bottleneck}]
\label{prop:chains}
Let $(W_0,\dots,W_n)$ be a chain such that $W_0=\{O\}$ and $W_n\cap U_k=\varnothing$.
Then for some index $i$, we have
\[
|W_i\cap U_k| > k.
\]
\end{prop}
\noindent Before proving this result, let us deduce Theorem~\ref{thm:lower} from it.
\begin{proof}[Proof of Theorem~\ref{thm:lower}]
  If $U_k=x+\Lambda_{d,m}$ then $U_{k+1}=x'+\Lambda_{d,m'}$ with $m'=2m+3d$ and $x'=2x-2(e_1+\dots+e_d)$.
  We obtain by induction $U_k=\Lambda_{d,3d(2^k-1)}-2(2^k-1)(e_1+\dots+e_d)\subset B_k$, and we are done by Proposition~\ref{prop:chains}. 
\end{proof}

\begin{definition}[{Good chain}]
Given $k\in\N$, a chain $\mathbb{W} = (W_0,\dots,W_n)$ is \emph{good}, if $W_0=\{O\}$, $W_n\cap U_k=\varnothing$ and 
\[|W_i\cap U_k|\le k\]
for all $i=0,\dots,n$.
\end{definition}

\begin{definition}[{Minimal good chain}]
Given $k\in\N$, a good chain $(W_0,\dots,W_n)$ is \emph{minimal}, if, for any other good chain $(W'_0,\dots,W'_{n'})$, one of the following holds:
\begin{align*}
n&{}< n';\\
n&{}=n',&\sum_{i=0}^n|W_i\cap U_k|&{}<\sum_{i=0}^{n}|W'_i\cap U_k|;\\
n&{}=n',&\sum_{i=0}^n|W_i\cap U_k|&{}=\sum_{i=0}^{n}|W'_i\cap U_k|,&\sum_{i=0}^n|\renorm(W_i)\cap U_{k-1}|&{}\le\sum_{i=0}^n|\renorm(W'_i)\cap U_{k-1}|.
\end{align*}
\end{definition}

\begin{definition}[{Internal inverted plaquette}]
Given $k\in\N$, if $u\in U_{k-1}$, we call the inverted plaquette $K_u$ \emph{internal}.
In what follows, we abbreviate ``internal inverted plaquette'' as \emph{iip}.
\end{definition}
Recall that the union of all iip is the set $U'_k\subset U_k$.

\begin{definition}[{Critical configuration}]
Given $k\in\N$, we call a configuration $W$ \emph{critical} if inside $U'_k$ there are exactly $k$ iip, each containing exactly one ON lamp, and there are no other ON lamps in $U_k$.
\end{definition}

Clearly, a configuration is critical if and only if $|W\cap U_k|=|\renorm(W)\cap U_{k-1}| = k$.

\begin{definition}[{Spread plaquette}]
\label{def:spread}
Given $k\in\N$, we say that a plaquette is \emph{spread} if it is contained in $U'_k$.
\end{definition}
Note that a spread plaquette is necessarily of type \ref{item:intersection:c} in Lemma~\ref{lem:intersection} (i.e.\ it intersects $(d+1)$ inverted plaquettes).

\begin{lemma}\label{lem:critical}
Let $k\in\N$ and let $W$ be a critical configuration, let $\mathcal{T}$ be a plaquette, and assume $W\cap U_k\neq (W\symdiff\mathcal T)\cap U_k$
and $|(W\symdiff\mathcal{T})\cap U_k|\le k$.
Then $\mathcal{T}$ is a spread plaquette.
\end{lemma}

\begin{proof}
  $\mathcal{T}$ intersects $U_k$, so toggling $\mathcal{T}$ must turn off at least one ON lamp of $W$ in $U_k$; otherwise we would only turn lamps on and hence would get $|(W\symdiff\mathcal T)\cap U_k|>k$.
Thus, $\mathcal{T}$ intersects $U'_{k}$, because $W\subset U'_k$.
Then, by the definition of $U_k$, we get $\mathcal{T}\subset U_k$.

Since $d+1\ge 3$, applying $\mathcal{T}$ must switch off at least two ON lamps of $W$.
Each inverted plaquette contains at most one ON lamp, so $\mathcal{T}$ intersects at least two iip.
This forces case \ref{item:intersection:c} of Lemma~\ref{lem:intersection}: applying $\mathcal{T}$ changes lamps in $(d+1)$ inverted plaquettes of the form $K_x,K_{x+e_1},K_{x+e_2},\dots, K_{x+d}$.

Since $U_{k-1}$ is a simplex, Lemma~\ref{lem:intersect2} implies
\[
|(x+T_d)\cap U_{k-1}| \in\{0, 1,d+1\}.
\]
The only case compatible with our assumptions is the third one.
In this case, all $(d+1)$ inverted plaquettes intersecting $\mathcal{T}$ are iip.
\end{proof}

\begin{lemma}[{Minimal good chains around a critical configuration}]
\label{lem:flip}
Let $k\in\N$ and let $\mathbb{W}=(W_0, W_1, \dots, W_n)$ be a minimal good chain.
Let $0 < i < n$ be an index such that $|\renorm(W_i)\cap U_{k-1}| = k$.
Set $\mathcal{T}_1\coloneqq W_i\symdiff W_{i-1}$ and $\mathcal{T}_2\coloneqq W_{i+1}\symdiff W_i$, and set $W'\coloneqq W_{i-1}\symdiff\mathcal{T}_2$.
Then
\begin{enumerate}
\item $\mathcal{T}_1$ and $\mathcal{T}_2$ are disjoint spread plaquettes;
\item $|W_{i-1}\cap U_k| = |W_{i+1}\cap U_k| = k$;
\item the configurations $W_{i\pm 1}$ contain no ON lamps in $U_k\setminus U'_k$;
\item the configuration $W'$ is also critical;
\item the sequence $(W_0, W_1, \dots, W_{i-1}, W', W_{i+1}, \dots, W_n)$ is also a minimal good chain.
\end{enumerate}
\end{lemma}
\begin{proof}
\noindent Clearly, $W_i$ is critical.

$\mathcal{T}_1$ and $\mathcal{T}_2$ both intersect $U_k$; otherwise we could shorten the chain.
By Lemma~\ref{lem:critical}, it follows that $\mathcal{T}_1$ and $\mathcal{T}_2$ are spread plaquettes.
In particular, they are contained in $U'_k$, which proves 3.

If $\mathcal{T}_1 = \mathcal{T}_2$, then $W_{i-1} = W_{i+1}$ and we can again shorten the chain.
Therefore, $\mathcal{T}_1$ and $\mathcal{T}_2$ are disjoint, proving 1. Indeed, distinct spread plaquettes are disjoint by \ref{item:1}. of Lemma~\ref{lem:intersection}, which applies to spread plaquettes by symmetry.

We may apply $\mathcal{T}_1$ and $\mathcal{T}_2$ in the opposite order: consider the chain
\[
\mathbb{W}'\coloneqq(W_0, W_1, \dots, W_{i-1}, W', W_{i+1}, \dots, W_n).
\]
If $|W'\cap U_k| < k$, then $\mathbb W'$ is a good chain and this contradicts the minimality of $\mathbb{W}$. Therefore, $|W'\cap U_k| \ge k$. Since $\mathcal{T}_1$ and $\mathcal{T}_2$ are disjoint,
\[
|W_i\cap U_k| + |W'\cap U_k|
=
|W_{i-1}\cap U_k| + |W_{i+1}\cap U_k|.
\]
The left-hand side is at least $2k$, while the right-hand side is at most $2k$.
Therefore, the only possibility is
\[
|W_{i-1}\cap U_k|=|W_{i+1}\cap U_k|=|W'\cap U_k|=k.
\]
This proves 2.

Next, $|\renorm(W')\cap U_{k-1}| \le |W'\cap U_{k}| = k$, so $\mathbb{W}'$ is also a minimal good chain, proving 5.
If $|\renorm(W')\cap U_{k-1}| < |W'\cap U_{k}|$, we again contradict the minimality of $\mathbb{W}$.
Hence, $|\renorm(W')\cap U_{k-1}| = k$, so $W'$ is critical, proving 4.
\end{proof}

\begin{proof}[Proof of Proposition~\ref{prop:chains}]

We proceed by induction on $k$. The base cases $k=0,1$ are straightforward, using the assumption $d\ge2$. Fix $k>1$. Assume the statement holds for $k-1$, and let us prove it for $k$.

Our goal is to show that no good chain exists.
Suppose, for the sake of contradiction, that there exists a good chain $\mathbb{W} = (W_0,\dots,W_n)$ that we may assume minimal.

Define
\[
V_i \coloneqq \renorm(W_i).
\]
Then $V_0$ consists of the single point $O$.
Moreover, $V_n\cap U_{k-1}=\varnothing$, because if $x\in U_{k-1}$ then
$K_x\subseteq U'_k\subset U_k$, so $W_n$ has no ON lamps in $K_x$.

By Lemma~\ref{lem:renorm}, the sequence $(V_0,\dots,V_n)$ is a chain.
By the induction hypothesis (applied to the chain $(V_0,\dots,V_n)$ and size parameter $k-1$), there exists an index $i$ such that
\[
|V_i\cap U_{k-1}| \ge k.
\]
Then $|V_i\cap U_{k-1}| = k$, as otherwise $|W_i\cap U_k| > k$.

Since $|V_0\cap U_{k-1}|=1$ and $V_n\cap U_{k-1}=\varnothing$, we can choose an index $t$ such that
\[
|V_{t-1}\cap U_{k-1}| < k \quad \text{and} \quad |V_t\cap U_{k-1}| = k.
\]
Since $|W_t\cap U_k| \ge |V_t\cap U_{k-1}|$ and $|W_t\cap U_k| \le k$, the only possibility is $|W_t\cap U_k| = k$.

Set $\mathcal{T}_1\coloneqq W_t\symdiff W_{t-1}$ and $\mathcal{T}_2\coloneqq W_t\symdiff W_{t+1}$.
Then, by Lemma~\ref{lem:flip}, $\mathcal{T}_1$ and $\mathcal{T}_2$ are spread plaquettes, and moreover
\[
|W_{t-1}\cap U_k|=|W_{t+1}\cap U_k|= k.
\]

If the dimension $d$ is even, we immediately obtain a contradiction with $|W_t\cap U_k|=k$, since the plaquette toggles an odd number of lamps.
For odd $d$, we need a more delicate argument.

As in Lemma~\ref{lem:flip}, set
$W'\coloneqq W_{t-1}\symdiff\mathcal{T}_2$.
By Lemma~\ref{lem:flip}, we have $|W'\cap U_k| = k$.
We will contradict the minimality of the original chain if we can show that
$|\renorm(W')\cap U_{k-1}| < |W'\cap U_k|$.

Since $W_t$ is critical, all its ON lamps lie inside iip.
By Lemma~\ref{lem:critical}, $\mathcal{T}_1$ contains no points outside iip.
Therefore, all ON lamps of $W_{t-1} = W_t\symdiff \mathcal{T}_1$ also lie inside iip.
At the same time, we know that $|V_{t-1}\cap U_{k-1}| < |W_{t-1}\cap U_k|$.
Hence, in some iip $K_u$ the configuration $W_{t-1}$ has at least two ON lamps.
It is easy to see that there are exactly two, since applying $\mathcal{T}_1$ changes the number of ON lamps in $K_u$ by exactly $1$, as $\mathcal T_1$ is a spread plaquette touching $K_u$.

For $i=1,2$ set $\mathcal S_i=\{x\in\Z^d:K_x\cap\mathcal T_i\neq\varnothing\}$, and note by Lemma~\ref{lem:intersection} that each $\mathcal S_i$ is a plaquette.
Since $W_t = W_{t-1}\symdiff\mathcal{T}_1$ is critical, we have $u \in \mathcal{S}_1$; otherwise $W_t$ would also contain two ON lamps inside $K_u$.
From Lemma~\ref{lem:flip} we know that $W' = W_{t-1}\symdiff\mathcal{T}_{2}$ is also critical.
Thus, by the same reasoning, $u\in \mathcal{S}_2$.

Therefore, $W_{t-1}$ contains two ON lamps inside $K_u$; denote them by $\ell_1\in\Z^d$ and $\ell_2\in\Z^d$.
Up to reordering, we may assume $\ell_i$ lies in $\mathcal{T}_i$ for $i=1,2$, and there are no other intersections of $\mathcal{T}_i$ with $K_u$.
Then $W_t$ contains exactly one ON lamp inside $K_u$, namely $\ell_2$, while $W'\cap K_u=\{\ell_1\}$, and $W_{t+1}$ contains no ON lamps inside $K_u$.

We now consider two cases.

\medskip

{\itshape Case 1. Assume that the configuration $W_{t+1}$ is not critical.}
We apply to $W_{t+1}$ the same reasoning as above for $W_{t-1}$.
The configuration $W_{t+1}$ has exactly $k$ ON lamps inside $U'_k$, so in some iip $K_v$ two lamps must be lit.
As before, we obtain $v\in \mathcal{S}_1\cap \mathcal{S}_2$.
Moreover, $v\neq u$, because $W_{t+1}$ has no ON lamps inside $K_u$. Therefore $|\mathcal S_1\cap\mathcal S_2|\ge2$, contradicting Lemma~\ref{lem:intersect2}.

\medskip

{\itshape Case 2. Assume that the configuration $W_{t+1}$ is critical.}
Set
\[
\mathcal{T}_3\coloneqq W_{t+2}\symdiff W_{t+1}.
\]
Apply Lemma~\ref{lem:flip} at index $i = t+1$.
We obtain $\mathcal{T}_2\neq \mathcal{T}_3$, and a minimal good chain
\[
\mathbb{W}''\coloneqq(W_0,\dots, W_{t-1}, W_t, W_t\symdiff\mathcal{T}_3, W_{t+2}, \dots, W_n)
\]
Apply Lemma~\ref{lem:flip} again, this time to the sequence $\mathbb{W}''$ at index $t$.
We obtain yet another minimal good chain
\[
\mathbb{W}'''\coloneqq(W_0,\dots, W_{t-1}, W_{t-1}\symdiff\mathcal{T}_3, W_t\symdiff\mathcal{T}_3, W_{t+2}, \dots, W_n).
\]

In addition, $\mathcal{T}_1\neq \mathcal{T}_3$ (otherwise $\mathbb W$ would not have been minimal), and the configuration $W_{t-1}\symdiff\mathcal{T}_3$ is critical.
Recall that $W_{t-1}$ has ON lamps $\ell_i\in K_u\cap\mathcal T_i$, for $i=1,2$.
Neither of these lamps lies in $\mathcal{T}_3$, since distinct spread plaquettes are disjoint.
Therefore, the configuration $W_{t-1}\symdiff\mathcal{T}_3$ cannot be critical, because it still contains at least two ON lamps inside $K_u$.

This contradiction completes the inductive step and hence the proof of the proposition.
\end{proof}

\subsection{Design of entangled configurations}
\label{subsec:complexity}
Recall that a \emph{plaquette} is any translate \(x+T_d\) of the set \(T_d=\{0,e_1,e_2,\dots,e_d\}\).
\begin{definition}[Admissible configuration]
\label{def:admissible}
A set \(W\subset \mathbb{Z}^d\) is called \emph{admissible} if it is the sum modulo \(2\) of finitely many plaquettes.
\end{definition}

The purpose of this subsection is to produce small, admissible configurations that require a large number of additional ON lamps to be turned completely off. We call these \emph{entangled} configurations.

\subsubsection{Small admissible configurations}

Admissible sets have a simple algebraic interpretation. Let
\[
R_d \coloneqq \mathbb{F}_2\left[x_1^{\pm 1},\ldots,x_d^{\pm 1}\right]
\]
be the ring of Laurent polynomials over \(\mathbb{F}_2\) in the variables
\(x_1,\ldots,x_d\). Its elements are finite sums of the form
\[
\sum_{\alpha \in \mathbb{Z}^d} c_\alpha x^\alpha,
\qquad
c_\alpha \in \mathbb{F}_2,
\qquad
x^\alpha = x_1^{\alpha_1}\cdots x_d^{\alpha_d},
\]
where only finitely many coefficients \(c_\alpha\) are nonzero.

\begin{prop}\label{prop:polymom}
A finite set \(W \subset \mathbb{Z}^d\) is admissible if and only if the Laurent polynomial
\[
\sum_{(t_1,\dots,t_d)\in W} \prod_{i=1}^d x_i^{t_i}
\]
is divisible in \(R_d\) by \(1+x_1+\cdots+x_d\).
\end{prop}

\begin{proof}
This follows directly from the fact that a translate \(a+T_d\) corresponds to the Laurent monomial \(x^a\) multiplied by \(1+x_1+\cdots+x_d\), while summation of configurations modulo \(2\) corresponds to addition in \(R_d\).
\end{proof}
Thus, under this correspondence, finite configurations are represented by Laurent polynomials, and admissible configurations form the principal ideal generated by \(1+x_1+\cdots+x_d\). We shall not need any deeper algebraic properties of this representation.

Let \(k\geq 0\) be an integer. By the Frobenius endomorphism,
\[
(1+x_1+\dots+x_d)^{2^k}
=
1+x_1^{2^k}+\dots+x_d^{2^k}.
\]
Therefore the set
\[
2^kT_d = \left\{0,2^ke_1,\dots, 2^ke_d\right\}
\]
is also admissible. We shall call these sets, and all their translates, \emph{large plaquettes}.

For \(d=1\), an admissible set is simply any finite set of even cardinality. In higher dimensions, admissible sets are much more rigid. In particular, admissible sets of small cardinality can be completely described.

\begin{lemma}\label{le:small_admissible}
Let \(d \geq 2\), and let \(X\subset \mathbb{Z}^d\) be an admissible set such that
\[
0 < |X| \leq d+1.
\]
Then \(X\) is a large plaquette.
\end{lemma}

\begin{proof}
We prove the statement by induction on \(d\).

\medskip

\noindent\emph{Base case: \(d=2\).}
This case is proved, for instance, in \cite{Arenas-Carmona08}*{Lemma 5.6}. For the reader's convenience, we recall one possible argument.

The \emph{Newton polytope} of a Laurent polynomial is the convex hull of its support. Under multiplication of Laurent polynomials, Newton polytopes add by Minkowski summation. The Newton polytope of \(1+x_1+x_2\) is the triangle \(\operatorname{conv}(T_2)\). Hence Proposition~\ref{prop:polymom} implies that every admissible set \(X\subset\mathbb{Z}^2\) has
\begin{enumerate}
    \item an edge with outer normal \(-e_2\);
    \item an edge with outer normal \(-e_1\);
    \item an edge with outer normal \(e_1+e_2\).
\end{enumerate}
It follows at once that if \(|X|\leq 3\), then \(X\) is homothetic to \(T_2\) with an integral scaling factor. After translating \(X\), we may assume that
\[
X = \{(0,0), (a,0), (0,a)\}
\]
for some positive integer \(a\). It remains to show that \(a\) is a power of \(2\). Set
\[
P(x_1,x_2)=1+x_1^a+x_2^a.
\]
Since \(P\) is divisible by \(1+x_1+x_2\), the polynomial \(P(x_1,1+x_1)\) must vanish identically. Over \(\mathbb{F}_2\),
\[
0=P(x_1,1+x_1)
=
1+x_1^a+(1+x_1)^a
=
\sum_{i=1}^{a-1}\binom{a}{i}x_1^i.
\]
Thus all intermediate binomial coefficients \(\binom{a}{i}\), \(1\leq i\leq a-1\), are even. By Kummer's theorem, equivalently by Lucas' theorem, this is possible only when \(a\) is a power of \(2\). Hence \(X\) is a large plaquette.

\medskip

\noindent\emph{Induction step.}
Assume that the lemma has already been proved in dimension \(d-1\), we prove it in dimension \(d\).

Let \(X\subset\mathbb{Z}^d\) be admissible. Translating \(X\), we may assume that the minimum first coordinate among points of \(X\) is \(0\). Consider the Laurent polynomial
\[
P(x_1,\dots,x_d)
=
\sum_{(t_1,\dots,t_d)\in X}\prod_{i=1}^d x_i^{t_i}.
\]
By Proposition~\ref{prop:polymom},
\[
P(x_1,\dots,x_d)
=
(1+x_1+\dots+x_d)Q(x_1,\dots,x_d)
\]
for some Laurent polynomial \(Q\).

Let \(b\) be the minimum exponent of \(x_1\) among the nonzero monomials of \(Q\). Since the minimum exponent of \(x_1\) among the monomials of \(P\) is \(0\), we have \(b=0\). Write
\[
Q=Q_1+Q_2,
\]
where all monomials of \(Q_1\) have \(x_1\)-degree \(0\), and all monomials of \(Q_2\) have strictly positive \(x_1\)-degree. Then
\[
P_0\coloneqq Q_1(1+x_2+x_3+\dots+x_d)
\]
is precisely the sum of all monomials of \(P\) with \(x_1\)-degree \(0\). Therefore \(P_0\) has at most \(d+1\) monomials. If \(P_0\) had exactly \(d+1\) monomials, then, since \(|X|\leq d+1\), we would have \(P=P_0\). This is impossible: the Newton polytope of \(P_0\) has dimension at most \(d-1\), whereas every nonzero multiple of \(1+x_1+\dots+x_d\) has \(d\)-dimensional Newton polytope.

Thus \(P_0\) has at most \(d\) monomials. The support of \(P_0\) is an admissible set in \(\mathbb{Z}^{d-1}\), so by the induction hypothesis,
\[
P_0
=
x_2^{t_2}\cdots x_d^{t_d}
\left(1+x_2^{2^k}+\dots+x_d^{2^k}\right)
\]
for some integers \(t_2,\dots,t_d\) and some integer \(k\geq 0\). In particular, \(P_0\) has exactly \(d\) nonzero monomials, and \(P\) has exactly \(d+1\) nonzero monomials. At this point we have also proved that, in dimension \(d\), there are no nonempty admissible sets of cardinality strictly smaller than \(d+1\).

Now consider
\[
P
-
x_2^{t_2}\cdots x_d^{t_d}
\left(1+x_1^{2^k}+x_2^{2^k}+\dots+x_d^{2^k}\right).
\]
This Laurent polynomial is divisible by \(1+x_1+\dots+x_d\). On the other hand, it has at most two nonzero monomials. Since there are no nonempty admissible sets of cardinality smaller than \(d+1\), this polynomial must be zero. Hence
\[
P
=
x_2^{t_2}\cdots x_d^{t_d}
\left(1+x_1^{2^k}+\dots+x_d^{2^k}\right),
\]
so \(X\) is a large plaquette. The induction step is complete.
\end{proof}

\subsubsection[The metric rho and an auxiliary lemma]{The metric \boldmath\(\rho\) and an auxiliary lemma}

Let \(d\in\mathbb{N}\). We call a vector \emph{long} if it has one of the forms
\[
\pm 2^k e_i
\qquad\text{or}\qquad
\pm 2^k(e_i-e_j),
\]
where \(k\geq 0\) is an integer and \(1\leq i<j\leq d\). Equivalently, long vectors are precisely directed edges of large plaquettes.

Define a metric \(\rho\) on \(\mathbb{Z}^d\) as follows: \(\rho(x,y)\) is the minimum number of long vectors whose sum is \(x-y\).

\begin{lemma}[Exponential decay in $\rho$]\label{le:long_sum}
Let
\[
\lambda_d=\frac{d}{d+1}.
\]
Let \(X\subset\mathbb{Z}^d\) be a nonempty admissible set, and let \(\alpha\in X\). Then
\begin{equation}
\label{eq:rho_sum}
\sum_{x\in X\setminus\{\alpha\}}
\lambda_d^{\rho(x,\alpha)-1}
\geq 2.
\end{equation}
\end{lemma}

\begin{proof}
We argue by induction on \(|X|\).

By Lemma~\ref{le:small_admissible}, the smallest possible cardinality of a nonempty admissible set is \(d+1\). In that case \(X\) is a large plaquette. If \(\alpha\in X\), then all other \(d\) points of \(X\) are at \(\rho\)-distance \(1\) from \(\alpha\), and the left-hand side of \eqref{eq:rho_sum} is exactly \(d\). This proves the base case.

Now assume that \(|X|>d+1\). Translating \(X\), we may suppose \(\alpha=0\in X\). We also assume that \(X\) contains a point with at least one odd coordinate. Indeed, if all coordinates of all points of \(X\) are even, then we may divide all coordinates by their largest common power of \(2\). The resulting set remains admissible, and the left-hand side of \eqref{eq:rho_sum} does not increase.

Call the following \(d(d+1)\) points \emph{near}:
\[
\alpha\pm e_i,
\qquad
\alpha\pm(e_i-e_j),
\qquad
1\leq i<j\leq d.
\]
These points have \(\rho\)-distance \(1\) from \(\alpha\).

\medskip

\noindent\emph{Case 1.}
Suppose that at least two near points belong to \(X\). Each of them contributes \(1\) to the sum in \eqref{eq:rho_sum}, so the desired inequality follows immediately.

\medskip

\noindent\emph{Case 2.}
Suppose that \(d=2\) and exactly one near point belongs to \(X\). Denote this point by \(\beta\). Choose a point \(\gamma\) such that
\[
\{\alpha,\beta,\gamma\}
\]
is a plaquette. Then \(\gamma\) is also near, and \(\gamma\notin X\).

Set
\[
X' = X \symdiff \{\alpha,\beta,\gamma\},
\]
where \(\symdiff\) denotes symmetric difference. Then \(X'\) is admissible, \(|X'|<|X|\), and \(\gamma\in X'\). By the induction hypothesis,
\[
\sum_{x\in X\setminus\{\alpha,\beta\}}
\lambda_2^{\rho(x,\gamma)-1}
=
\sum_{y\in X'\setminus\{\gamma\}}
\lambda_2^{\rho(y,\gamma)-1}
\geq 2.
\]
Since \(\rho(\gamma,\alpha)=1\), we have
\[
\rho(x,\alpha)\leq \rho(x,\gamma)+1
\]
for every \(x\). Hence,
\[\sum_{x\in X\setminus\{\alpha\}}
\lambda_2^{\rho(x,\alpha)-1}\geq
1+
\sum_{x\in X\setminus\{\alpha,\beta\}}
\lambda_2^{\rho(x,\alpha)-1} \geq
1+
\lambda_2
\sum_{x\in X\setminus\{\alpha,\beta\}}
\lambda_2^{\rho(x,\gamma)-1} \geq
1+2\lambda_2
>2.
\]

\medskip

\noindent\emph{Case 3.}
Suppose either that \(d=2\) and no near point belongs to \(X\), or that \(d>2\) and at most one near point belongs to \(X\). If such a near point exists, denote it by \(\beta\).

Let $\mathcal I=\{0,1\}^d$ be the set of all binary strings of length \(d\). Partition \(\mathbb{Z}^d\) into \(2^d\) parity classes, denoted by \(\mathbb{Z}_I^d\) for all \(I\in\mathcal I\). Put
\[
X_I \coloneqq (X\setminus\{\alpha=0\})\cap \mathbb{Z}_I^d.
\]
For \(I\in\mathcal I\), define
\[
S_I
\coloneqq
\sum_{x\in X_I}
\lambda_d^{\rho(x,0)-1}.
\]
Define $\mathbf 0=00\dots 0\in\cI$. Among the sums \(S_I\) with \(I\neq \mathbf 0\), choose a maximal one and denote it by \(S_J\). If \(S_J=0\), then all points of \(X\) have all coordinates even, contrary to our reduction above. Therefore \(S_J>0\).

\begin{claim}
\label{claim:gamma}
There exists $\gamma\in T_d$ such that \(\gamma-T_d\) contains no point of \(\mathbb{Z}_J^d\) and \(\gamma-T_d\) contains no point of \(X\) other than possibly \(\alpha=0\).
\end{claim}
\begin{proof}
There are \(d+1\) possible choices for \(\gamma\). For every such \(\gamma\), all points of \(\gamma-T_d\) are either \(0\) or near points. Among all the $d(d+1)$ near points, at most two belong to the parity class \(\mathbb{Z}_J^d\). Moreover, by the assumptions of Case 3, at most one near point belongs to \(X\). Thus there are at most three bad near points when \(d>2\), and at most two when \(d=2\).

It remains to observe that each bad point rules out at most one choice of \(\gamma\in T_d\). Indeed, if the same bad point belonged to both \(\gamma_1-T_d\) and \(\gamma_2-T_d\), then the two distinct translates \(\gamma_1-T_d\) and \(\gamma_2-T_d\) would intersect in both that bad point and the origin. Equivalently, two distinct translates of \(T_d\) would intersect in two points, which is impossible because of Lemma \ref{lem:intersect2}. Hence at least one admissible choice of \(\gamma\) exists.
\end{proof}

Choose $\gamma$ as in Claim~\ref{claim:gamma}. Let \(\mathcal I_\gamma\subset\mathcal I\) be the set of parity classes represented by the points of \(\gamma-T_d\). This set consists of the parity vector of \(\gamma\) and the \(d\) parity vectors obtained from it by flipping exactly one coordinate. In particular, \(|\mathcal I_\gamma|=d+1\), and \(\mathbf 0\in\mathcal I_\gamma\). Also, by construction, \(J\notin\mathcal I_\gamma\).

Recalling the renormalisation map $\renorm$ from \eqref{eq:renorm}, set
\[
Y=\renorm(X-\gamma).
\]
By Lemma~\ref{lem:renorm}, \(Y\) is admissible. Since
\[
|X\cap(\gamma-T_d)|=1,
\]
we have \(0\in Y\). 

For every \(y\in Y\), 
\[
\gamma+2y-T_d
\]
contains an odd number of points of \(X\). Choose one of these points and denote it by \(\tau(y)\). The sets \(\gamma+2y-T_d\) are pairwise disjoint for distinct \(y\) by Lemma~\ref{lem:intersection}, so the points \(\tau(y)\) are pairwise distinct.

For $I\in\cI_\gamma$, define $Y_I=\{y\in Y\setminus\{0\}\mid\tau(y)\in X_I\}$, so we have 
\[
|Y_I|\leq |X_I|
\]
for every \(I\in\mathcal I\). Moreover, \(Y_I\) is empty unless \(I\in\mathcal I_\gamma\). In particular, \(Y_J\) is empty. Since \(S_J>0\), we have \(|X_J|>0\), and therefore \(|Y|<|X|\). Thus the induction hypothesis applies to \(Y\):
\[
\sum_{y\in Y\setminus\{0\}}
\lambda_d^{\rho(y,0)-1}
\geq 2.
\]

We now compare the contributions of \(Y\) and \(X\). If \(\tau(y)\in X_{\mathbf 0}\), then
\[
\rho(\tau(y),0)\leq \rho(y,0).
\]
If \(\tau(y)\notin X_{\mathbf 0}\), then
\[
\rho(\tau(y),0)\leq \rho(y,0)+1.
\]
Consequently,
\[
S_{\mathbf 0}
\geq
\sum_{y\in Y_{\mathbf 0}}
\lambda_d^{\rho(y,0)-1},
\]
and for every \(I\in\mathcal I_\gamma\setminus\{\mathbf 0\}\),
\[
S_I
\geq
\lambda_d
\sum_{y\in Y_I}
\lambda_d^{\rho(y,0)-1}.
\]

Since \(S_J\) is maximal among all \(S_I\) with \(I\neq\mathbf 0\), and since \(J\notin\mathcal I_\gamma\), we get
\[
\begin{aligned}
\sum_{x\in X\setminus\{\alpha\}}
\lambda_d^{\rho(x,\alpha)-1}
&\geq
\sum_{I\in\mathcal I_\gamma}S_I+S_J \\
&\geq
S_{\mathbf 0}
+
\frac{d+1}{d}
\sum_{I\in\mathcal I_\gamma\setminus\{\mathbf 0\}}S_I \\
&\geq
\sum_{y\in Y_{\mathbf 0}}
\lambda_d^{\rho(y,0)-1}
+
\frac{d+1}{d}
\sum_{I\in\mathcal I_\gamma\setminus\{\mathbf 0\}}
\lambda_d
\sum_{y\in Y_I}
\lambda_d^{\rho(y,0)-1}.
\end{aligned}
\]
By our choice \(\lambda_d=d/(d+1)\), the last expression equals
\[
\sum_{I\in\cI_\gamma}\sum_{y\in Y_I}\lambda_d^{\rho(y,0)-1}=\sum_{y\in Y\setminus\{0\}}
\lambda_d^{\rho(y,0)-1}.
\]
Therefore
\[
\sum_{x\in X\setminus\{\alpha\}}
\lambda_d^{\rho(x,\alpha)-1}
\geq
\sum_{y\in Y\setminus\{0\}}
\lambda_d^{\rho(y,0)-1}
\geq 2.
\]
The lemma follows.
\end{proof}

\subsubsection{Configurations with high energy barrier}\label{ss:highbarrier}

\begin{definition}[$r$-separated]
Let $r\in\N$. A subset $X\subset \Z^d$ is called \emph{$r$-separated} if $\rho(x,y)\geq r$ for all distinct points $x,y\in X$.
\end{definition}

\begin{lemma}[Large sets contain admissible ones]\label{le:subset_adm}
Let $d,n\in\N$. If a subset $X\subseteq \Lambda_{d,n}$ contains at least
\[
\binom{n+d-1}{d-1}+1
\]
points, then $X$ contains a nonempty admissible subset.
\end{lemma}

\begin{proof}
Let $S_0$ be the face of the simplex consisting of points whose first coordinate is zero. Then
\[
|S_0|=\binom{n+d-1}{d-1}.
\]

Subsets of $\Lambda_{d,n}$ form a vector space over $\mathbb{F}_2$, and admissible sets form a linear subspace. It is enough to show that the codimension of this subspace is at most $|S_0|$. For this, we show that any lamp configuration can be transformed, by symmetric differences with plaquettes, into a configuration supported inside $S_0$.

For any point $x\in\Lambda_{d,n}\setminus S_0$, the plaquette $x-e_1+T_d$ is contained entirely in $\Lambda_{d,n}$. Hence one can switch off any lamp with positive first coordinate without affecting lamps whose first coordinate is at least that of $x$. Proceeding in decreasing order of the first coordinate transforms the configuration into one supported in $S_0$.
\end{proof}

In the metric $\rho$, balls are infinite. Nevertheless, we can estimate their density, that is, the size of their intersection with a simplex of a given size.

\begin{lemma}[Truncated $\rho$-ball size]
\label{le:size_rho}
Let $d,n\in\N$, and let $0<c<1$. Let $\alpha\in \Lambda_{d,n}$, and let
\[
A=\{x\in\Z^d\mid \rho(\alpha,x)\leq c\ln n\}.
\]
Then
\[
|A\cap \Lambda_{d,n}| \leq n^{3c\ln(4d)-c\ln c}.
\]
\end{lemma}

\begin{proof}
Put $r\coloneqq\lfloor c\ln n\rfloor$. We may assume that $r\geq1$. Choose $k$ such that
\[
2^{k-1}\leq n<2^k.
\]
Consider the quotient map
\[
\pi\colon\Z^d\to \Z^d/2^k\Z^d.
\]
Since $\pi$ does not identify distinct points of $\Lambda_{d,n}$, it is enough to prove that
\[
|\pi(A)|\leq n^{3c\ln(4d)-c\ln c}.
\]
Under $\pi$, all long vectors divisible by $2^k$ map to zero. Consider the set of
\[
m\coloneqq1+kd(d+1)
\]
vectors
\[
V\coloneqq\{0\}\cup\left\{\pm2^a e_i\mid 1\leq i\leq d,\ 0\leq a<k\right\}
\cup\left\{\pm2^a(e_i-e_j)\mid 1\leq i<j\leq d,\ 0\leq a<k\right\}.
\]
Clearly, $|\pi(A)|$ is at most the number of ways to choose $r$ elements from $V$ with repetitions. Therefore
\[
\binom{m+r-1}{r}
\leq \frac{(m+r)^r}{(r/e)^r}
\leq \left(e\left(1+\frac mr\right)\right)^{c\ln n}
= n^{c\left(1+\ln\left(1+\frac mr\right)\right)}.
\]
A routine estimate gives
\[
1+\frac mr \leq \frac{4(d+1)^2}{c}.
\]
After a further harmless simplification, we obtain
\[
|\pi(A)|\leq n^{3c\ln(4d)-c\ln c}.\qedhere
\]
\end{proof}

\begin{corollary}[Admissible separated sets exist]\label{cons:set_exists}
For any $d,n\in\N$ there exists, inside $\Lambda_{d,n}$, a nonempty admissible $(c_d\ln n)$-separated set $X$ of cardinality at most
\[
\binom{n+d-1}{d-1}+1.
\]
Here $c_d>0$ is a constant depending only on $d$.
\end{corollary}

\begin{proof}
Choose $c_d>0$ sufficiently small so that, for all relevant $n$,
\[
n^{3c_d\ln(4d)-c_d\ln c_d}<\frac{n+d}{d}.
\]
Equivalently,
\[
n^{3c_d\ln(4d)-c_d\ln c_d}\binom{n+d-1}{d-1}<|\Lambda_{d,n}|.
\]

Put $r=\lfloor c_d\ln n\rfloor$. We first construct an $r$-separated subset of $\Lambda_{d,n}$ of cardinality
\[
\binom{n+d-1}{d-1}+1.
\]
Add points greedily. If no more than $\binom{n+d-1}{d-1}$ points have already been chosen, Lemma~\ref{le:size_rho} and the inequality above ensure that there is still a point of $\Lambda_{d,n}$ outside all $\rho$-balls of radius $r$ around the chosen points.

By Lemma~\ref{le:subset_adm}, this $r$-separated set contains a nonempty admissible subset. Any subset of an $r$-separated set is again $r$-separated.
\end{proof}

Recall that chains were introduced in Definition~\ref{def:chain}, as sequences of configurations varying by a single switch, and that intermediate configurations need not be contained in $\Lambda_{d,n}$.
\begin{theorem}[Configurations with high energy barrier]\label{th:complexity:2}
For each $d\ge 2$, there exist $\alpha_d>0$ and $n_d\in\N$ such that, for any $n\ge n_d$, there exists an admissible set $X\subseteq \Lambda_{d,n}$ with the following property. Any chain $(X_t)_{t=0}^T$ with $X_0=X$, $X_T=\varnothing$ satisfies
\[|X_t\setminus X|\ge n^{\alpha_d}+3|X\setminus X_t|/2\]
for some $t\in\{0,\dots,T\}$. In particular, $|X_t|\ge|X|+n^{\alpha_d}$.
\end{theorem}
\begin{proof}
By Corollary~\ref{cons:set_exists}, there exists an admissible $(c_d\ln n)$-separated subset $X\subset \Lambda_{d,n}$ with $|X|\leq C_d n^{d-1}$. Put
\[
r=c_d\ln n.
\]
Let $|X|=m$, and write
\[
X=\{x_1,\dots,x_m\}.
\]
For each $i$, define
\[
U_i\coloneqq\{y\in\Z^d\mid \rho(y,x_i)<r/3\}.
\]
The sets $U_i$ are pairwise disjoint.

Let
\[
X=X_0,X_1,\dots,X_T=\varnothing
\]
be any sequence of configurations such that the symmetric difference of any two consecutive configurations is a plaquette. We do not require the intermediate configurations to be contained in $\Lambda_{d,n}$.

Say that a configuration $X_t$ behaves \emph{unusually} inside $U_j$ if
\[
|X_t\cap U_j|\leq 1
\qquad\text{and}\qquad
X_t\cap U_j\neq \{x_j\}.
\]
Initially, the configuration behaves usually inside every $U_j$, whereas the final configuration behaves unusually inside every $U_j$. Hence, there is an index $t$ such that $X_{t-1}$ behaves usually inside all $U_j$, while $X_t$ behaves unusually inside at least one $U_j$. This index $j$ is unique, since a single plaquette cannot intersect two distinct sets $U_j$ and $U_{j'}$ when $r$ is large enough. Without loss of generality, assume that $j=1$.

Put
\[
X' = X_t \symdiff X.
\]
Then $X'$ is admissible, $x_1\in X'$, and
\[
|X'\cap U_1|\leq 2.
\]
By Lemma~\ref{le:long_sum},
\[
\sum_{y\in X'\setminus\{x_1\}} \lambda_d^{\rho(x_1,y)-1}\geq 2,
\qquad
\lambda_d=\frac d{d+1}.
\]
Inside $U_1$, the set $X'\setminus\{x_1\}$ contains at most one point, whose contribution to the sum is at most $1$. Therefore, the points of $X'\setminus U_1$ contribute at least $1$ to the sum. Each such point contributes at most $\lambda_d^{r/3-1}$, and hence
\[
|X'|\geq \lambda_d^{1-r/3}
= C'_d\left(1+\frac1d\right)^{r/3}
\geq n^{\alpha_d}
\]
for a suitable constant $\alpha_d>0$ depending only on $d$.

Call a point of $X'$ \emph{old} if it belongs to $X$, and \emph{new} otherwise. Suppose that $x_j$ is an old point of $X'$ with $j\neq1$. Since $X_t$ behaves usually inside $U_j$, and $x_j\notin X_t$, we must have
\[
|X_t\cap U_j|\geq2.
\]
Thus, each old point other than possibly $x_1$ forces at least two new points. Consequently, if $O$ and $N$ denote the numbers of old and new points in $X'$, then $N\ge 2(O-1)$, so
\[N\geq 5(O-1)/3+N/6=3O/2+|X'|/6-5/3\ge 3O/2+n^{\alpha_d}/6-5/3\ge 3O/2+n^{\alpha_d/2},\]
taking $n$ large enough in the last inequality.
\end{proof}

\subsection{Enumerative combinatorics: generating function of finite-volume groundstates}
\label{subsec:cycles}

Let $A$ and $B$ be two subsets of $\Z^d$, which we think of, now, as sets of switches.
A subset $X\subseteq B$ is called an \emph{$(A,B)$-cycle} if every translate $A+x$ that is entirely contained in $B$ intersects $X$ in an even number of points.
We denote the set of all $(A,B)$-cycles by $\mathcal{C}_{(A,B)}$.
Using symmetric difference $\symdiff$ for addition, $\mathcal{C}_{(A,B)}$ is a vector space over the field $\mathbb{F}_2$. In fact, set $V=\mathbb F_2^B$ and $W=\mathbb F_2^{\{x\mid A+x\subseteq B\}}$, with $f\colon V\to W$ given by $f(v)(x)=\sum_{y\in A+x}v(y)$; then $\mathcal C_{(A,B)}=\ker f$. Here elements of $V$ are switch configurations, whose image in $W$ is the corresponding lamp configuration.

Recall the plaquette $T_d$ from \eqref{eq:def:Td} and the size-$n$ simplex $\Lambda_{d,n}$ from \eqref{eq:def:Lambda:dn}. We are interested in the set of $(T_d,\Lambda_{d,n})$-cycles, which we denote by $\mathcal{C}_{d,n}$. Define the generating polynomial
\begin{equation}\label{eq:def:fdn}
  G_{d,n}(t)=\sum_{X\in \mathcal{C}_{d,n}} t^{|X|}.
\end{equation}

Our third main combinatorial result is the following bound:
\begin{theorem}[Polynomial estimate]\label{th:upper}
  For every $d\in \N$ there exists $a_d\in \N$ such that the following holds. For every $m\in \N$ and every $s\geqslant a_d$, if $t>0$ and $n\in \N$ satisfy
  \[t \le 1-\frac{1}{m} \qquad\text{and}\qquad n \ge m^s,\]
  then
  \[
    G_{d,n}(t) \le 1+n^{-s}.
  \]
\end{theorem}

At a high level, the proof of Theorem~\ref{th:upper} proceeds as follows (see Figure~\ref{fig:cycles}). We will prove a recursive bound (see Proposition~\ref{pr:cycle_rec}) of the form $G_{d,n}(t)\le (G_{d,n/2}(t^{1.5}))^{2^{d-1}}$, provided $G_{d-1,n}(t^{1/4})$ is sufficiently close to $1$. To do so, we consider the restriction of a cycle to the even or odd hyperplanes perpendicular to a well chosen direction among $e_1,\dots,e_d$. We notice that such sections form a collection of $2^{d-1}$ independent cycles blown up by a factor $2$ (see Lemma~\ref{le:rec}). Up to choosing the slicing direction and parity well, we are able to show that (see Lemma~\ref{le:maj}) the total contribution to the generating polynomial $G_{d,n}(t)$ of cycles with identical restrictions is smaller than $t^{1.5|Y|}$, where $Y$ is the restriction. This relies on the bound on $G_{d-1,n}(t^{1/4})$ and the fact that symmetric differences of such cycles boil down to $(d-1)$-dimensional cycles living on a facet of the simplex. Once the recursive bound is established, Theorem~\ref{th:upper} is follows by crude estimation carried out in Section~\ref{subsubsec:poly:estimates}. Morally, we just note that iterating the inequality yields $G_{d,n}(t)\le (1+t^{1.5^{\log_2 n}})^{2^{(d-1)\log_2n}}\to 1$ as $n\to\infty$.

\begin{figure}

    \centering
    \begin{tikzpicture}[x=0.5cm,y=0.5cm,radius=3pt]
\foreach \x in {0,...,8} {
    \draw[gray] (\x, 0)--(0,\x);
    \draw[gray] (\x,0)--(\x,8-\x)--(0,8-\x);}
  	\fill[color=red] (0,8) circle;
  	\fill[color=red] (0,6) circle;
  	\fill[color=blue] (0,5) circle;
  	\fill[color=red] (0,2) circle;
  	\fill[color=brown] (1,7) circle;
  	\fill (1,6) circle;
  	\fill (1,4) circle;
  	\fill (1,2) circle;
  	\fill[color=brown] (1,1) circle;
  	\fill[color=red] (2,0) circle;
  	\fill[color=red] (2,2) circle;
  	\fill[color=blue] (2,3) circle;
  	\fill[color=red] (2,4) circle;
  	\fill[color=blue] (2,5) circle;
  	\fill[color=brown] (3,5) circle;
  	\fill[color=brown] (3,1) circle;
  	\fill (3,0) circle;
  	\fill[color=red] (4,4) circle;
  	\fill[color=blue] (4,1) circle;
  	\fill (5,0) circle;
  	\fill[color=brown] (5,1) circle;
  	\fill[color=brown] (5,3) circle;
  	\fill[color=red] (6,2) circle;
  	\fill[color=blue] (6,1) circle;
  	\fill (7,0) circle;
  	\fill[color=red] (8,0) circle;
  	\end{tikzpicture}
    \caption{A typical cycle in $\cC_{2,8}$. Its four restrictions to $\Lambda_{2,8}^r$ for $r\in\{0,1\}^2$ (see \eqref{eq:def:Lambda:dnr}) are given in different colors and are all (renormalised) cycles (see Lemma~\ref{le:rec}). In this case, a suitable choice of restriction is the one to $\Lambda^{1,1}_{2,8}$ (odd columns, black and brown points; see \eqref{eq:def:Lambda:dnke}), as it is smaller than half the cycle and nonempty ($\Lambda^{2,1}_{2,8}$ also works). Notice that (in two dimensions) there is exactly one other cycle with the same black and brown sites called partial cycle. That other cycle only differs from the one depicted in the leftmost column which contains precisely the sites which are not blue or red in the figure.}
    \label{fig:cycles}
\end{figure}

\subsubsection{Basic properties of cycles}\label{subsubsec:cycles}	
	
\begin{lemma}[{The first layer determines everything}]\label{le:cyc_layer}
  Let $d,n\in \N$, and let $1\leqslant k\leqslant d$.
  Denote by $S_0$ and $S_1$ the sets of points in $\Lambda_{d,n}$ whose $k$-th coordinate is equal to $0$ and $1$ respectively.
  \begin{enumerate}
  \item\label{item:cyc_layer:1} If $X\in \mathcal{C}_{d,n}$ and $X\cap S_0=\emptyset$, then $X=\emptyset$.
  \item\label{item:cyc_layer:2} If $X\in \mathcal{C}_{d,n}$ and $X\cap S_1=\emptyset$, then $X\subseteq S_0$.
    Moreover, $X\in \mathcal{C}_{(T',S_0)}$, where $T'=T_d\setminus\{e_k\}$.
  \end{enumerate}
\end{lemma}
	
\begin{proof}
\ref{item:cyc_layer:1}. Suppose that $|X|>0$.
    Choose a point $x\in X$ with minimal $k$-th coordinate.
    Since $x\notin S_0$, the plaquette $x-e_k+T_d$ is entirely contained in $\Lambda_{d,n}$, and it contains exactly one point of $X$, namely $x$.
    This contradicts the cycle condition.
			
\ref{item:cyc_layer:2}. Let $S_+\coloneqq\Lambda_{d,n}\setminus S_0$, that is, the set of all points whose $k$-th coordinate is positive.
    Clearly, $X\cap S_+$ is a $(T_d,S_+)$-cycle.
    By part~1, already proved above, we have $X\cap S_+=\emptyset$, hence $X\subseteq S_0$.
			
    For any translate $T'+x$ contained entirely in $S_0$, there is a point of $S_1$ which completes it to a translate of $T_d$.
    Since this point does not belong to $X$, it follows that $|(T'+x)\cap X|$ is even.
    Therefore $X\in \mathcal{C}_{(T',S_0)}$.
\end{proof}
	
We now show that every cycle is either empty or has large cardinality.
	
\begin{lemma}[{Cycles are macroscopic}]\label{le:cyc_size}
  Let $X\in \mathcal{C}_{d,n}$ and assume $X\neq \emptyset$.
  Then $|X|\geqslant n+1$.
\end{lemma}
	
\begin{proof}
  We prove this by double induction: first on $n$, and then on $d$.
  The base case is $n=0$ and arbitrary $d$.
		
  Let $X\in \mathcal{C}_{d,n}$.
  Decompose $\Lambda_{d,n}$ into two parts:
  \[
    \Lambda_{d,n}=S_0\sqcup S_+,
  \]
  where $S_0$ consists of the points whose $d$-th coordinate is equal to $0$, and $S_+$ consists of all remaining points.
  
  Observe that $S_+$ is a translate of $\Lambda_{d,n-1}$, while $S_0$ is an embedding of $\Lambda_{d-1,n}$ into $\Z^d$.
  We assume that the induction hypothesis applies to these sets.
		
  \medskip
  \noindent
  \emph{Case 1.} Suppose that $X\cap S_0=\emptyset$.
  This is impossible by part~1 of Lemma~\ref{le:cyc_layer}.
		
  \medskip
  \noindent
  \emph{Case 2.} Suppose that $X\cap S_+=\emptyset$.
  Then by part~2 of Lemma~\ref{le:cyc_layer}, the set $X$ is a $(T',S_0)$-cycle.
  Applying the induction hypothesis for $d-1$ and $n$, we obtain $|X|\geqslant n+1$.
		
  \medskip
  \noindent
  \emph{Case 3.} Suppose that both $X\cap S_0$ and $X\cap S_+$ are nonempty.
  Then $X\cap S_+$ is a $(T_d,S_+)$-cycle, and by the induction hypothesis we have
  \[
    |X\cap S_+|\geqslant n.
  \]
  Hence $|X|\geqslant n+1$.
\end{proof}
	
\subsubsection{A recursive inequality for the generating polynomials of cycles}
		
For $d=1$, the generating polynomial of cycles can be computed explicitly: the set $\Lambda_{1,n}$ is an interval of length $n$.
In this case,
\[
  \mathcal{C}_{1,n}=\{\Lambda_{1,n},\emptyset\},
\]
and therefore
\[
  G_{1,n}(t)=1+t^{n+1}.
\]
	
\begin{prop}[{Recursive relation}]\label{pr:cycle_rec}
  Let $d>1$.
  Assume that for some $t\in (0,1)$ and $n\in \N$ one has
  \[
    G_{d-1,n}\left(t^{1/4}\right)\leqslant 1+\frac{1}{4d}.\]
  Then
  \[
    G_{d,n}(t)\leqslant \left(G_{d,\left\lfloor \frac{n-\epsilon}{2} \right\rfloor}\left(t^{1.5}\right)\right)^{2^{d-1}}
  \]
  for some $0\leqslant \epsilon\leqslant d$.
\end{prop}
	
\begin{proof}
  First, let us derive a simple consequence of the condition imposed on $t$ and $n$.
		
  The simplex $\Lambda_{d-1,n}$ contains the empty cycle.
  It also contains a cycle of size $n+1$, namely the set of all points whose coordinates other than the first are equal to $0$.
  Hence
  \[
    G_{d-1,n}\left(t^{1/4}\right)\geqslant 1+t^{(n+1)/4},
  \]
  and therefore
  \begin{equation}
    t^{(n+1)/4}\leqslant \frac{1}{4d}.
    \label{ineq:t}
  \end{equation}
  
  Now we examine the geometry of the simplex $\Lambda_{d,n}$ more closely.
  Partition all points of $\Lambda_{d,n}$ into $2^d$ subsets according to the parities of their coordinates (see Figure~\ref{fig:cycles}).
  For $r=(r_1,\dots,r_d)\in \{0,1\}^d$, define
  \begin{equation}
  \label{eq:def:Lambda:dnr}
    \Lambda_{d,n}^r\coloneqq\left\{(x_1,\dots,x_d)\in \Lambda_{d,n}\mid x_i\equiv r_i \pmod 2 \text{ for all } 1\leqslant i\leqslant d\right\}.
  \end{equation}
		
  Then, clearly,
  \[
    \Lambda_{d,n}=\bigsqcup_{r\in \{0,1\}^d}\Lambda_{d,n}^r.
  \]
		
  For each $r\in \{0,1\}^d$, the set $\Lambda_{d,n}^r$ looks like a simplex scaled by a factor of $2$, with side length approximately $n/2$.
  More precisely,
  \begin{equation}
    \Lambda_{d,n}^r=r+2\Lambda_{d,n_r},
    \label{eq:rec}
  \end{equation}
  where
  \[
    n_r\coloneqq\left\lfloor \frac{n-(r_1+\dots+r_d)}{2}\right\rfloor.
  \]
  
  Formula~\eqref{eq:rec} defines a bijection
  \[
    \tau_r\colon\Lambda_{d,n}^r\to\Lambda_{d,n_r}.
  \]
  
  If $X\subseteq \Lambda_{d,n}$, define
  \[
    \tau_r(X)\coloneqq\tau_r\left(X\cap\Lambda_{d,n}^r\right).
  \]
		
  Thus, we may view a subset $X\subseteq \Lambda_{d,n}$ as a collection of $2^d$ subsets of the form $\tau_r(X)$.
		
  \begin{lemma}[{Restrictions of cycles are renormalised cycles}]
    \label{le:rec}
    If $X\in \mathcal{C}_{d,n}$, then $\tau_r(X)\in \mathcal{C}_{d,n_r}$ for all $r\in \{0,1\}^d$.
  \end{lemma}
		
  \begin{proof}
    The statement follows immediately from the identity
    \[
      T_d \symdiff (e_1+T_d)\symdiff \dots \symdiff (e_d+T_d)=2T_d.
    \]
			
    This identity is easy to verify directly: all points of the form $e_i+e_j$, where $i\neq j$, cancel out.
    It can also be interpreted as the Frobenius identity
    \[
      (1+x_1+\dots+x_d)^2=1+x_1^2+\dots+x_d^2
    \]
    over a field of characteristic $2$.
    (This observation is not new; it is used, for instance, to construct a renormalisation for the Ledrappier subshift.)
  \end{proof}
		
  Consider $1\leqslant k\leqslant d$ and $\varepsilon\in \{0,1\}$, and define
  \begin{equation}
  \label{eq:def:Lambda:dnke}
    \Lambda_{d,n}^{k,\varepsilon}\coloneqq\left\{(x_1,\dots,x_d)\in \Lambda_{d,n}\mid x_k\equiv \varepsilon \pmod 2\right\}
  \end{equation}
  (see Figure~\ref{fig:cycles}). We call subsets of $\Lambda_{d,n}$ of this form \emph{binary layers}. There are $2d$ binary layers in total.
  Binary layers with the same $k$ and different values of $\varepsilon$ will be called \emph{complementary}; the one with $\varepsilon=0$ will be called \emph{large}, and the one with $\varepsilon=1$ \emph{small}.
  		
  If in addition $X\in \mathcal{C}_{d,n}$, then we call $X\cap\Lambda_{d,n}^{k,\varepsilon}$ a \emph{partial cycle of type $(k,\varepsilon)$}.
  The set of all such partial cycles is denoted by $\mathcal{C}_{d,n}^{k,\varepsilon}$.
		
  \begin{lemma}[{Cycles with fixed even or odd sections}]\label{le:maj}
    Let $1\leqslant k\leqslant d$, $\varepsilon\in \{0,1\}$, and suppose that $t$ and $n$ satisfy the assumptions of Proposition~\ref{pr:cycle_rec}.
    Let $Y$ be a nonempty partial cycle of type $(k,\varepsilon)$, and let $X_1,\dots,X_m$ be distinct cycles from $\mathcal{C}_{d,n}$ such that
    \[
      X_i\cap\Lambda_{d,n}^{k,\varepsilon}=Y
      \quad\text{and}\quad
      |X_i|\geqslant 2|Y|
    \]
    for every $i$.
    Then
    \[
      \sum_{i=1}^{m} t^{|X_i|}\leqslant \frac{t^{1.5|Y|}}{2d}.
    \]
  \end{lemma}		
  \begin{proof}
    Let $S_0$ be the set of points in $\Lambda_{d,n}$ whose $k$-th coordinate is equal to $0$.
    We distinguish two cases depending on the value of $\varepsilon$.
			
    \medskip
    \noindent
    \emph{Case 1: $\varepsilon=0$.}
    We claim that in this case $m\leqslant 1$.
    Indeed, suppose that $X_1\cap\Lambda_{d,n}^{k,0}=X_2\cap\Lambda_{d,n}^{k,0}$.
    Consider the symmetric difference
    \[
      X'\coloneqq X_1\symdiff X_2.
    \]
    Then $X'\cap S_0=\emptyset$, hence $X'=\emptyset$ by part~1 of Lemma~\ref{le:cyc_layer}.
    Therefore $X_1=X_2$, so it remains to check that
    \[
      t^{|X_1|}\leqslant \frac{t^{1.5|Y|}}{2d}.
    \]
    We know that $|X_1|\geqslant n+1$ by Lemma~\ref{le:cyc_size}, and that $|Y|\leqslant |X_1|/2$, hence, by \eqref{ineq:t},
    \[
      t^{|X_1|-1.5|Y|}
      \leqslant
      t^{|X_1|/4}
      \leqslant
      t^{(n+1)/4}
      \leqslant
      \frac{1}{4d}.
    \]
			
    \medskip
    \noindent
    \emph{Case 2: $\varepsilon=1$.}
    Without loss of generality, we may assume that $|X_i|\geqslant |X_1|$ for all $1<i\leqslant m$.
    As in the previous case,
    \[
      t^{|X_1|-1.5|Y|}\leqslant t^{(n+1)/4},
    \]
    so we only need to estimate the sum of the remaining terms.
    
    \noindent Define
    \[
      L_i\coloneqq X_i\symdiff X_1 \qquad (1\leqslant i\leqslant m),
    \]
    so in particular $L_1=\emptyset$.
    Clearly, $L_i\cap \Lambda_{d,n}^{k,1}=\emptyset$ for every $i$, so part~2 of Lemma~\ref{le:cyc_layer} implies that
    \[
      L_i\subseteq S_0
      \quad\text{and}\quad
      L_i\in \mathcal{C}_{(T',S_0)},
    \]
    where $T'=T_d\setminus\{e_k\}$.
    Thus each $L_i$ can be viewed simply as a cycle in a simplex of smaller dimension.
    Therefore, for every $\lambda\in (0,1)$,
    \begin{equation}
      \sum_{i=2}^m \lambda^{|L_i|}
      \leqslant
      \sum_{\substack{X\in \mathcal{C}_{d-1,n}\\ X\neq \emptyset}} \lambda^{|X|}
      =
      G_{d-1,n}(\lambda)-1.
      \label{ineq:Li}
    \end{equation}
    
    \noindent For $1\leqslant i\leqslant m$, let
    \[
      Z_i\coloneqq X_i\cap\Lambda_{d,n}^{k,0},
    \]
    so $Z_i=X_i\setminus Y$. For every $i$ we have $Z_i=Z_1\symdiff L_i$,
    and we also have
    \[
      |Z_i|\geqslant |Z_1| \quad\text{and}\quad |Z_i|\geqslant |Y|.
    \]
			
    We claim that
    \[
      |Z_i|\geqslant |L_i|/2	\qquad\text{for all } i.
    \]
    Indeed, if $|L_i|\geqslant 2|Z_1|$, then $|Z_i|\geqslant |L_i|-|Z_1|\geqslant |L_i|/2$, and if $|L_i|\leqslant 2|Z_1|$ then $|Z_i|\geqslant |Z_1|\geqslant |L_i|/2$.
    Combining this with $|Z_i|\geqslant |Y|$, we have
    \[
      |Z_i|-|Y|/2 \geqslant |Z_i|/2 \geqslant |L_i|/4.
    \]
    
    Using these inequalities, we obtain
    \[
      \sum_{i=2}^{m} t^{|X_i|-1.5|Y|}
      =
      \sum_{i=2}^{m} t^{|Z_i|-|Y|/2}
      \leqslant
      \sum_{i=2}^{m} t^{|L_i|/4}
      \leqslant
      G_{d-1,n}\left(t^{1/4}\right)-1.
    \]
    
    Therefore, by \eqref{ineq:t},
    \[
      \sum_{i=1}^{m} t^{|X_i|-1.5|Y|}
      \leqslant
      t^{(n+1)/4}+G_{d-1,n}\left(t^{1/4}\right)-1
      \leqslant
      \frac{1}{2d}.\qedhere
    \]
  \end{proof}
		
  If $1\leqslant k\leqslant d$ and $\varepsilon\in \{0,1\}$, let $R_{k,\varepsilon}$ denote the set of all binary strings of length $d$ whose $k$-th coordinate is equal to $\varepsilon$.
		
  We will prove the following inequality:
  \begin{equation}
    G_{d,n}(t)\leqslant \frac{1}{2d}\sum_{\substack{1\leqslant k\leqslant d\\ \varepsilon\in \{0,1\}}}
    \prod_{r\in R_{k,\varepsilon}} G_{d,n_r}\left(t^{1.5}\right).
    \label{eq}
  \end{equation}
  
  Observe that this inequality immediately implies Proposition~\ref{pr:cycle_rec}.
  
  \begin{lemma}[{Slicing direction}]\label{le:nonempty}
    Let $X\in \mathcal{C}_{d,n}$ and assume that $|X|>0$.
    Then there exists $k$ such that
    \[
      X\cap\Lambda_{d,n}^{k,0}\neq\emptyset\text{ and }X\cap\Lambda_{d,n}^{k,1}\neq\emptyset.
    \]
  \end{lemma}
  
  \begin{proof}
    Let $x\in X$ and $y\in \Lambda_{d,n-1}$ be such that $x\in y+T_d\subseteq\Lambda_{d,n}$. Since $X\in\cC_{d,n}$, $|(y+T_d)\cap X|\ge 2$, so there exists $z\in(y+T_d)\setminus\{x\}$. Then any $k$ such that $\langle x-z,e_k\rangle\neq 0$ is as desired.
  \end{proof}
		
  Let $X\in \mathcal{C}_{d,n}$ be a nonzero cycle.
  Choose, according to Lemma~\ref{le:nonempty}, values of $k$ and $\varepsilon$ such that
  \[
    0<\left|X\cap\Lambda_{d,n}^{k,\varepsilon}\right|\leqslant |X|/2.
  \]
  Let
  \[
    Y=X\cap\Lambda_{d,n}^{k,\varepsilon}.
  \]
  We say that the cycle $X$ is \emph{of type} $(k,\varepsilon,Y)$.
  
  In this way, we define a type for every nonzero cycle.
  
  We now group the nonconstant terms on the left-hand side of \eqref{eq} by type, and estimate the sum within each type using Lemma~\ref{le:maj}:
  \[
    G_{d,n}(t)
    =
    \sum_{X\in \mathcal{C}_{d,n}} t^{|X|}
    =
    1+\sum_{X \text{ is of type }(k,\varepsilon,Y)} t^{|X|}
    \leqslant
    1+\frac{1}{2d}\sum_{\substack{1\leqslant k\leqslant d\\ \varepsilon\in \{0,1\}}}
    \sum_{\substack{Y\in \mathcal{C}_{d,n}^{k,\varepsilon}\\ Y\neq \emptyset}} t^{1.5|Y|}.
  \]
  Equivalently,
  \[
    G_{d,n}(t)
    \leqslant
    \frac{1}{2d}\sum_{\substack{1\leqslant k\leqslant d\\ \varepsilon\in \{0,1\}}}
    \sum_{Y\in \mathcal{C}_{d,n}^{k,\varepsilon}} t^{1.5|Y|}.
  \]
  
  Now fix $k$ and $\varepsilon$.
  Consider a partial cycle $Y\in \mathcal{C}_{d,n}^{k,\varepsilon}$.
  The support of $Y$ lies inside
  \[
    \Lambda_{d,n}^{k,\varepsilon}=\bigsqcup_{r\in R_{k,\varepsilon}} \Lambda_{d,n}^r,
  \]
  and
  \[
    t^{1.5|Y|}
    =
    \prod_{r\in R_{k,\varepsilon}} \left(t^{1.5}\right)^{|\tau_r(Y)|}.
  \]
  
  By Lemma~\ref{le:rec}, all sets $\tau_r(Y)$ are cycles in the corresponding simplices.
  Therefore the monomial $t^{1.5|Y|}$ appears in the expansion of
  \[
    \prod_{r\in R_{k,\varepsilon}} G_{d,n_r}\left(t^{1.5}\right)
    =
    \prod_{r\in R_{k,\varepsilon}} \sum_{Z\in \mathcal{C}_{d,n_r}} \left(t^{1.5}\right)^{|Z|},
  \]
  namely by choosing in each factor the term corresponding to $\tau_r(Y)$.
  
  Different $Y$ correspond to different such choices.
  Therefore, for positive $t$, we obtain
  \[
    \sum_{Y\in \mathcal{C}_{d,n}^{k,\varepsilon}} t^{1.5|Y|}
    \leqslant
    \prod_{r\in R_{k,\varepsilon}} G_{d,n_r}\left(t^{1.5}\right),
  \]
  which completes the proof of \eqref{eq}, and hence also of Proposition~\ref{pr:cycle_rec}.
\end{proof}
	
\subsubsection{Polynomial estimates}
\label{subsubsec:poly:estimates}
\begin{lemma}[{Crude bound}]
\label{le:G_bounds}
Let $d,n\in \N$.
Then the function $G_{d,n}(t)$ is increasing on the interval $[0,1]$, and on this interval it satisfies
\[
G_{d,n}(t) \leqslant 1 + 2^{(n+1)^{d-1}} t^{n+1}.
\]
\end{lemma}

\begin{proof}
Monotonicity is obvious. By Lemma~\ref{le:cyc_size}, every nontrivial cycle has size at least $n+1$, hence for $0\leqslant t\leqslant 1$,
\[
G_{d,n}(t) \leqslant 1 + |\mathcal{C}_{d,n}|\, t^{n+1}.
\]

It remains to estimate the total number of cycles.
By part 1 of Lemma~\ref{le:cyc_layer}, a cycle is uniquely determined by its intersection with $S_0$, and
\[
|S_0|\leqslant (n+1)^{d-1}.
\]
Therefore
\[
|\mathcal{C}_{d,n}| \leqslant 2^{(n+1)^{d-1}}.\qedhere
\]
\end{proof}

\noindent We are now ready to prove Theorem~\ref{th:upper} by induction.
\begin{proof}[Proof of Theorem~\ref{th:upper}]
\emph{Base case:} $d=1$. We know that
\[
G_{1,n}(t)=1+t^{n+1}.
\]
Take $a_1=6$. We must check that for $s\geqslant 6$ and $n\geqslant m^s$ one has
\[
\left(1-\frac{1}{m}\right)^{n+1} \leqslant n^{-s}.
\]
Using $(1-1/m)^m\leqslant e^{-1}$, it is enough to prove
\[
\frac{n+1}{m}\geqslant s\ln n.
\]
Since $m\leqslant n^{1/s}$, it is enough to show that for $n\geqslant m^s\geqslant 2^s$ we have
\[
n+1\geqslant s n^{1/s}\ln n;
\]
this holds for $n=2^s$ since $s\geqslant 6$, and also for larger $n$ by convexity.

\medskip

\emph{Induction step.}

Fix $a_{d-1}>d$ such that for all $m\geqslant 2$, $s\geqslant a_{d-1}$, and $n>m^s$, one has
\[
G_{d-1,n}\!\left(1-\frac{1}{m}\right) < 1+n^{-s}.
\]

Then, whenever $t<1-\frac{1}{m}$ and $n\geqslant m^{3a_{d-1}}$, we automatically have
\[
n \geqslant m^{3a_{d-1}} \geqslant (4m)^{a_{d-1}} \geqslant 8^d > 4d,
\]
and therefore
\[
G_{d-1,n}\left(\sqrt[4]{t}\right)
\leqslant
G_{d-1,n}\!\left(1-\frac{1}{4m}\right)
<
1+n^{-a_{d-1}}
<
1+\frac{1}{n}
<
1+\frac{1}{4d}.
\]

Hence the assumption of Proposition~\ref{pr:cycle_rec} is satisfied, so
\[
G_{d,n}(t)
\leqslant
\left(
G_{d,\left\lfloor \frac{n-\epsilon}{2}\right\rfloor}\left(t^{1.5}\right)
\right)^{2^{d-1}}
\]
for some $0\leqslant \epsilon \leqslant d$.
Since $n>4d$, we have
\[
\left\lfloor \frac{n-\epsilon}{2}\right\rfloor \geqslant \frac{n}{4}.
\]

Choose $a_d$ large enough so that the following conditions hold:
\begin{gather}
\label{eq:ad1}a_d > 6a_{d-1};\\
\label{eq:ad2}x^{a_d\ln(1.5)/4} \geqslant 3x \cdot \ln 2 \cdot x^{3(d-1)a_{d-1}}\text{ for all }x\geqslant 2;\\
\label{eq:ad3}a_d > \frac{12}{\ln 1.5};\\
\label{eq:ad4}1.5^{x/24} \geqslant x\text{ for all }x\geqslant a_d;\\
\label{eq:ad5}x^{\ln1.5}/12 > 3\ln x\text{ for all }x\geqslant 2^{a_d};\\
\label{eq:ad6}1.5^{\frac{2}{3}x} \geqslant 3\ln 2\cdot dx\text{ for all }x\geqslant \dfrac{a_d}{8}.
\end{gather}

These conditions will be used below. It is easy to check that each of them holds for all sufficiently large $a_d$.

Assume now that
\[
t<1-\frac{1}{m},\qquad s\geqslant a_d,\qquad n\geqslant m^s.
\]

We apply Proposition~\ref{pr:cycle_rec} repeatedly as long as possible.
At each step, the parameter $n$ decreases, and we may continue until $n$ becomes smaller than $m^{3a_{d-1}}$.
The parameter $t$ also decreases, but this does not interfere with the application of the lemma:
\begin{equation}
\label{eq:fdn:recurrence}
G_{d,n}(t)
\leqslant
\left(G_{d,n_1}\left(t^{1.5}\right)\right)^{2^{d-1}}
\leqslant
\left(G_{d,n_2}\left(t^{1.5^2}\right)\right)^{(2^{d-1})^2}
\leqslant
\dots
\leqslant
\left(G_{d,n_r}\left(t^{1.5^r}\right)\right)^{(2^{d-1})^r},
\end{equation}
where $(n_i)$ is a decreasing sequence of positive integers, with
\[
2 \leqslant \frac{n_i}{n_{i+1}} \leqslant 4,
\qquad\text{and}\qquad
n_r < m^{3a_{d-1}}.
\]

This gives a lower bound on $r$:
\[
r \geqslant \frac{\ln n - \ln m^{3a_{d-1}}}{\ln 4}.
\]

By \eqref{eq:ad1}, we have $\ln m^{3a_{d-1}} < 0.5\ln n$, so
\begin{equation}
r \geqslant \frac{\ln n}{4} \geqslant \frac{s\ln m}{4} \geqslant \frac{s}{8}.
\label{neq:r}
\end{equation}

By Lemma~\ref{le:G_bounds}, for any $0<\lambda<1$ we have
\[
G_{d,n_r}(\lambda) \leqslant 1 + 2^{(n_r+1)^{d-1}} \lambda^{n_r+1}\leqslant 1 + 2^{(m^{3a_{d-1}})^{d-1}}{\lambda},
\]
so \eqref{eq:fdn:recurrence} gives
\[
G_{d,n}(t)
\leqslant
\left(
1 + 2^{m^{3(d-1)a_{d-1}}} \, t^{1.5^r}
\right)^{(2^{d-1})^r}.
\]

We want to check that this is at most $1+n^{-s}$.

If $k$ is a positive integer and $0 \leqslant 2kx<1$, then
\[
(1+x)^k \leqslant 1+2kx,
\]
which is easily proved by induction.
So it is enough to verify
\[
2\cdot2^{m^{3(d-1)a_{d-1}}}\cdot t^{1.5^r}\cdot \left(2^{d-1}\right)^r < n^{-s}.
\]

Taking logarithms, this becomes
\[
\ln 2 \cdot m^{3(d-1)a_{d-1}}
+ (\ln t)\cdot 1.5^r
+ \ln 2 \cdot (1+r(d-1))
< -s\ln n.
\]

Since $t<1-\frac{1}{m}$, we have $\ln t < -\frac{1}{m}$. We also have $r\geqslant \frac{s}{8} \geqslant 1$. Thus, it is enough to prove
\begin{equation}
    \label{eq:poly}
\ln 2 \cdot m^{3(d-1)a_{d-1}}
+ s\ln n
+ \ln 2 \cdot dr
<
\frac{1.5^r}{m}.
\end{equation}

To do this, we show that each of the three terms on the left is at most $\frac{1.5^r}{3m}$.

\begin{enumerate}
    \item By \eqref{neq:r} and \eqref{eq:ad2},
    \begin{equation}
    \label{eq:poly:1}
    1.5^r \geqslant m^{s\ln1.5/4} \geqslant m^{a_d\ln1.5/4}
    \geqslant 3m\cdot \ln 2 \cdot m^{3(d-1)a_{d-1}}.
    \end{equation}

    \item 
    By \eqref{neq:r} and \eqref{eq:ad3},
        \begin{equation}
        \label{eq:2:a}
        \frac{r}{3} \geqslant \frac{s\ln m}{12}\ge \frac{a_d\ln m}{12}\ge \frac{\ln m}{\ln 1.5}.\end{equation}
    Again by \eqref{neq:r} and \eqref{eq:ad4}, we have
        \begin{equation}
        \label{eq:2:b}
        1.5^{r/3} \geqslant 1.5^{s/24} \geqslant s.
        \end{equation}
    By \eqref{neq:r} and \eqref{eq:ad5}, we also have
        \begin{equation}
        \label{eq:2:c}
        \frac{r}{3} \geqslant \frac{\ln n}{12}\ge\frac{\ln(3\ln n)}{\ln 1.5}.
        \end{equation}
    Putting \eqref{eq:2:a}, \eqref{eq:2:b} and \eqref{eq:2:c} together, we get
    \begin{equation}
    \label{eq:poly:2}
    3ms\ln n \leqslant 1.5^r.
    \end{equation}
    \item By \eqref{eq:2:a}, \eqref{neq:r} and \eqref{eq:ad6}, we have
    \begin{equation}
    \label{eq:poly:3}
    1.5^r \geqslant m\cdot1.5^{2r/3}\ge 3m\ln 2\cdot dr.
    \end{equation}
\end{enumerate}
The desired \eqref{eq:poly} follows directly from \eqref{eq:poly:1}, \eqref{eq:poly:2} and \eqref{eq:poly:3}.
\end{proof}	

\section{Probabilistic arguments}\label{ss:probabilistic}
\subsection{Lower bounds via the combinatorial bottleneck}
\label{subsec:lower:bottleneck}
In this section we prove the lower bounds of Theorems~\ref{th:main} and~\ref{th:2k:asymptotics}. Given Theorem~\ref{thm:lower}, the route we follow is relatively standard in the context of kinetically constrained models. The two proofs are very similar, but we start with the simpler one.

\begin{proof}[Proof of the lower bound of Theorem~\ref{th:2k:asymptotics}]
The lower bound holds for general even $d\ge2$. For $\varepsilon>0$ small enough (depending on $d$), consider $\beta>1/\varepsilon^2$ and an integer $k$ such that $k\in(1/\varepsilon,\beta\varepsilon)$ (if it exists). We consider the torus $\bbT_n$ with $n=2^k$. Set $\Omega'=\bbF_2^{\bbT_n}$, and define the mapping
\begin{equation}
\label{eq:def:Phi}
\Phi\colon\Omega_{\bbT_n}\to\Omega'\colon\sigma\mapsto\left(\frac{1-[\sigma]_{x+T_d}}2\right)_{x\in\bbT_n}.
\end{equation}
from the spin configuration to the lamp configuration. Then, recalling the definitions of $T_d$ from \eqref{eq:def:Td} and $H$ from \eqref{eq:def:H}, we have \[H_{\bbT_n}(\sigma)=-|\bbT_n|+2\sum_{x\in\bbT_n}\Phi(\sigma)_x,\]
while $\Phi(\sigma^x)=\Phi(\sigma)+\1_{x-T_d}$; here and below addition takes place in $\bbF_2^{\bbT_n}$, and for a set or element $S$ we denote by $\1_S$ the characteristic function of $S$. We then see that the image of the Glauber dynamics on $\bbT_n$ via $\Phi$ is the Markov process with generator
\begin{equation}
\label{eq:def:L:prime}\cL'f(\omega)=\sum_{x\in\bbT_n}\frac{e^{-2\beta\sum_{y\in x-T_d}(1-\omega_y)}}{e^{-2\beta\sum_{y\in x-T_d}\omega_y}+e^{-2\beta\sum_{y\in x-T_d}(1-\omega_y)}}\left(f(\omega+\1_{x-T_d})-f(\omega)\right).
\end{equation}
We denote by $\pi$ the product Bernoulli measure on $\Omega'$ with parameter 
\begin{equation}
\label{eq:def:p}
p=1/(1+e^{2\beta})\in\left[1/(2e^{2\beta}),e^{-2\beta}\right],
\end{equation}
which is clearly invariant for this dynamics.

Consider a central box $B=\{-L,\dots,L\}^d$ with $L=2^{k-2}$. Define the events 
\begin{align*}
\cE_0&{}=\left\{\omega\in\Omega'\mid\forall x\in B,\omega_x=0\right\},&
\cE_1&{}=\left\{\omega\in\Omega'\mid\omega_B=\1_{O}\right\}.
\end{align*}
Recalling Definition~\ref{def:chain} and Theorem~\ref{thm:lower}, we say that $\omega\in\Omega'$ is \emph{good} if there exists a chain $(\omega^{(i)})_{i=0}^m$ such that $\omega^{(0)}=\omega$, $\omega^{(m)}_B=\1_{O}$ and $|\omega^{(i)}|\coloneqq\sum_{x\in B}\omega^{(i)}_x\le k-d-4$ for all $i\in\{0,\dots,m\}$. Let $\cG$ be the set of good configurations. By Theorem~\ref{thm:lower} applied to the chain $(\omega^{(m-i)}+\omega^{(m)}+\1_O)_{i=0}^m$, we obtain $\cG\cap\cE_0=\varnothing$. Indeed, we may apply Theorem~\ref{thm:lower} to $\bbT_m$ rather than to $\bbZ^d$ when $m\ge 4+3d2^{k+1}$: consider a counterexample chain on the torus of minimal length; then all its configurations are supported in $B_k+\{-1,0,1\}^d$, so it lifts to a chain on $\bbZ^d$ that contradicts Theorem~\ref{thm:lower}.
Consequently,
\[\var'(\1_{\cG})=\pi(\cG)(1-\pi(\cG))\ge \pi(\cE_1)\pi(\cE_0)=p(1-p)^{2|B|-1}.\]
Moreover,
\begin{align*}
\cD'\left(\1_{\cG}\right)&{}=-\pi\left(\1_\cG\cL'\1_\cG\right)\le 2\sum_{x\in\bbT_n}\pi\left(\omega\in\cG\not\ni\omega+\1_{x-T_d}\right)\\
&{}\le 2|B+T_d|\pi\left(\sum_{x\in B}\omega_x>k-d-4-|T_d|\right)
\le 2|B+T_d|(p|B|)^{k-2d-4}.
\end{align*}
Putting these inequalities together, we obtain
\begin{align}
\nonumber\frac{\var'(\1_\cG)}{\cD'(\1_\cG)}&{}\ge \frac{p(1-p)^{2|B|-1}}{2|B+T_d|(p|B|)^{k-2d-4}}\ge p^2\left(p2^{dk}\right)^{2d+4-k}\\
&{}\ge e^{(2-\varepsilon d\ln 2)\beta(k-2d-4)-4\beta}/4\ge e^{\beta k(2-\varepsilon d\ln 2)(1-\varepsilon(2d+9))}\ge e^{2\beta k (1-O(\varepsilon))}.\label{eq:trel:lower'}
\end{align}

Recalling \eqref{eq:def:trel}, the proof is concluded by Lemma~\ref{lem:Phi:bij} below.
\end{proof}

\begin{lemma}
\label{lem:Phi:bij}
Whenever $n$ is a power of $2$ and $d$ is even, the map $\Phi$ from \eqref{eq:def:Phi} is bijective.
\end{lemma}
\begin{proof}
Note that $|\Omega_{\bbT_n}|=|\Omega'|$, that $\Phi$ is a group homomorphism, and that $\Omega'_{\bbT_n}$ is a vector space with basis $(\1_x)_{x\in\bbT_n}$. Therefore, it is enough to show that, for any $x\in\bbT_n$, there exists $\sigma\in\Omega_{\bbT_n}$ such that $\Phi(\sigma)=\1_{x}$. By symmetry we only show this for $x=O$. Indeed, the Sierpinski gasket configuration (see Figure~\ref{fig:Sierpinski})
\begin{equation}
\label{eq:Sierpinski}
\sigma_y=(-1)^{\binom{-\sum_{i=1}^dy_i}{-y_1,\dots,-y_d}}
\end{equation}
for $-y\in\{0,\dots,n-1\}^d$ verifies this property for $d$ even and $n$ power of $2$.
\end{proof}

\begin{figure}
    \centering    \begin{tikzpicture}[x=1cm,y=1cm,scale=0.5,every circle/.style={scale=2,radius=3pt}]
\draw[gray,step=1] (-7.5,-7.5) grid (0.5,0.5);
\foreach \x in {0,...,7} {
    \draw[gray] (-\x-0.5, 0.5)--(0.5,-\x-0.5);}
\foreach \x in {1,...,7} {
    \draw[gray] (-\x+0.5, -7.5)--(-7.5,-\x+0.5);}
    \fill (0,0) circle;
  	\fill (-1,0) circle;
  	\fill (-2,0) circle;
  	\fill (-3,0) circle;
  	\fill (-4,0) circle;
  	\fill (-5,0) circle;
  	\fill (-6,0) circle;
  	\fill (-7,0) circle;
  	\fill (0,-1) circle;
  	\fill (0,-2) circle;
  	\fill (0,-3) circle;
  	\fill (0,-4) circle;
  	\fill (0,-5) circle;
  	\fill (0,-6) circle;
  	\fill (0,-7) circle;
  	\fill (-2,-1) circle;
  	\fill (-1,-2) circle;
  	\fill (-4,-1) circle;
  	\fill (-1,-4) circle;
  	\fill (-6,-1) circle;
  	\fill (-1,-6) circle;
  	\fill (-4,-2) circle;
  	\fill (-2,-4) circle;
  	\fill (-5,-2) circle;
  	\fill (-2,-5) circle;
  	\fill (-4,-3) circle;
  	\fill (-3,-4) circle;
  	\fill[red,opacity=0.4] (0,0) rectangle (0.5,0.5);
  	\fill[red,opacity=0.4] (-7.5,0.5)--(-7,0)--(-7.5,0)--cycle;
  	\fill[red,opacity=0.4] (0.5,-7.5)--(0,-7)--(0,-7.5)--cycle;
  	\draw[dashed] (-7.5,-7.5) rectangle (0.5,0.5);
  	\end{tikzpicture}
    \caption{A Sierpi\'nski gasket spin configuration for $d=2$. Dots represent $-1$ spins. Viewed on the torus $\bbT_8$, this configuration has only one ON lamp, corresponding to the red triangle having three spins in state $-1$. Viewed on $\bbZ^2$, it has three ON lamps, corresponding to the three red triangles, parts of which are drawn.}
    \label{fig:Sierpinski}
\end{figure}
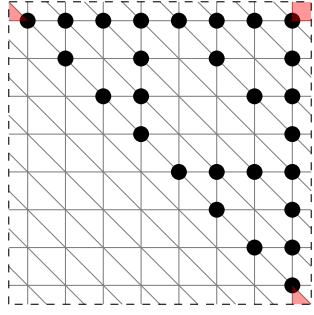

To prepare for the infinite volume case, we need the following substitute for Lemma~\ref{lem:Phi:bij} which fails in infinite volume.
\begin{lemma}
\label{lem:Phi:mu:pi}
Consider $\Omega'=\bbF_2^{\bbZ^d}$ and
\[\Phi\colon\Omega\to\Omega'\colon\sigma\mapsto\left(\frac{1-[\sigma]_{x+T_d}}2\right)_{x\in\bbZ^d}.\]
Let $\Phi(\mu)$ denote the push forward of $\mu$ under $\Phi$ and $\pi$ denote the Bernoulli product measure on $\bbZ^d$ with parameter $p=1/(1+e^{2\beta})$. Then $\Phi(\mu)=\pi$.
\end{lemma}
\begin{proof}
Fix $\varepsilon>0$ and finite $\Lambda\subset\bbZ^d$. We will prove (independently) in Section~\ref{subsec:proba:upper} that the Gibbs measure $\mu$ is unique. Consequently, we can find $\Lambda'\supset\Lambda+\{-1,0,1\}^d=:\Lambda_+$ large enough so that 
\[\sup_{\tau\in\Omega_{\bbZ^d\setminus\Lambda'}}\dtv\left(\mu^\tau_{\Lambda_+},\mu_{\Lambda_+}\right)<\varepsilon\]
($\dtv$ is the total variation distance, that is, half of the $\ell^1$-distance), where $\mu^\tau_{\Lambda_+}$ is the marginal of $\mu^\tau_{\Lambda'}$ on $\Lambda_+$ and similarly for $\mu$ instead of $\mu^\tau_{\Lambda'}$. Note that $\Omega\to\Omega'_\Lambda:\sigma\mapsto\Phi(\sigma)_{\Lambda}$ is local with support contained in $\Lambda_+$ and surjective. Indeed, we may change the value at a single vertex in $\Lambda$ by adding to $\sigma$ a large enough Sierpi\'nski gasket translated at that vertex (see \eqref{eq:Sierpinski} and Figure~\ref{fig:Sierpinski}). 

Consider the image under $\Phi$ of the chain with generator $\cL_{\Lambda'}^\tau$. As above, it is a Markov chain, with invariant measure $\pi^\tau_{\Lambda'}=\pi_{\Lambda'-T_d}(\cdot\mid\Omega'_{\Lambda',\tau})$, where $\Omega'_{\Lambda',\tau}\subseteq\Omega'_{\Lambda'}\coloneqq\bbF_2^{\Lambda'-T_d}$ is the irreducible component of $\cL$ induced by the boundary condition $\tau$. In particular, $\Omega'_{\Lambda',\tau}$ are either equal or disjoint for different $\tau\in\Omega_{\bbZ^d\setminus\Lambda'}$. By the surjectivity above, $\bigcup_{\tau\in\Omega_{\bbZ^d\setminus\Lambda'}}\Omega'_{\Lambda',\tau}=\Omega'_{\Lambda'}$, so we can find a probability measure $\mathrm d\theta(\tau)$ on $\Omega_{\bbZ^d\setminus\Lambda'}$ such that $\theta(\pi_{\Lambda'}^\tau)=\pi_{\Lambda'}$. Since, for any $\tau,\tau'\in\Omega_{\bbZ^d\setminus\Lambda'}$,
\[\dtv(\pi^\tau_{\Lambda},\pi^{\tau'}_{\Lambda})\le \dtv(\mu_{\Lambda+}^\tau,\mu_{\Lambda_+}^{\tau'})\le 2\varepsilon,\]
we deduce 
\[\dtv\left(\pi_{\Lambda},\Phi(\mu)_\Lambda\right)\le 2\varepsilon+\sup_{\tau\in\Omega_{\bbZ^d\setminus\Lambda'}}\dtv\left(\pi_\Lambda^\tau,\Phi(\mu)_\Lambda\right)\le2\varepsilon+\sup_{\tau\in\Omega_{\bbZ^d\setminus\Lambda'}}\dtv\left(\mu^\tau_{\Lambda_+},\mu_{\Lambda_+}\right)\le 3\varepsilon.\]
Therefore, $\pi=\Phi(\mu)$.
\end{proof}

\begin{proof}[Proof of the lower bound of Theorem~\ref{th:main}]
The rest of the proof proceeds exactly as the lower bound of Theorem~\ref{th:2k:asymptotics}, so we only indicate the non-trivial changes.

Fix $d\ge 2$ and $\beta>0$. Let $k=\lceil\beta/(d\ln 2)\rceil$, $L=3d2^k$, $B=\{-L,\dots,L\}^d$, so that 
\[e^{\beta}\le L^d\le |B|\le (3L)^d\le (18d)^de^{\beta}.\]
Define $\cG\subset\Omega'$ as in the proof of Theorem~\ref{th:2k:asymptotics} with chains featuring at most $k$ lamps ON (instead of $k-d-4$). Then, we similarly get
\begin{equation}
\label{eq:trel:lower}
 \frac{\var'(\1_\cG)}{\cD'(\1_\cG)}\ge \frac{p(1-p)^{2|B|-1}}{2|B+T_d|(p|B|)^{k-d}}\ge c e^{c\beta^2}\end{equation}
for a small enough $c=c(d)>0$. By Lemma~\ref{lem:Phi:mu:pi} and \eqref{eq:trel:lower}, plugging the local function $\1_{\Phi^{-1}(\cG)}$ into \eqref{eq:def:trel} gives the desired lower bound on $\trel$.
\end{proof}

\subsection{Lower bound via entangled configurations}
In this section, we prove the lower bound of Theorem~\ref{th:finite}, using Theorem~\ref{th:complexity:2}.

\begin{proof}[Proof of the lower bound of Theorem~\ref{th:finite}]
Fix $d\ge 2$, $\alpha_d\le1$ and $n_d$ as in Theorem~\ref{th:complexity:2}. Let $\beta>0$ and $n\ge n_d$ satisfy $n\le e^{2\beta/(5d)}$. Let $\Lambda=-\Lambda_{d,n}$. Define the mapping 
\[\Phi\colon\Omega_{\Lambda}\to\bbF_2^{\bbZ^d}\colon\sigma\mapsto\left(\frac{1-[\sigma\cdot\bone_{\bbZ^d\setminus\Lambda}]_{x+T_d}}2\right)_{x\in\bbZ^d}.\]

\begin{lemma}\label{lem:admissible:preimage}
Let $\Omega'=\{X\subset\bbZ^d\mid -X\subseteq\Lambda-T_d,X\text{ is admissible}\}$. Then the function $\Phi$ is bijective onto its image 
\[\Phi(\Omega_\Lambda)=\left\{\1_{-X}\mid X\in\Omega'\right\},\]
where we recall that $\1_{-X}$ denotes the characteristic function of $-X$.
\end{lemma}
\begin{proof}
We first claim that, for every $\sigma\in\Omega_\Lambda$, $\Phi(\sigma)=\1_{-X}$ for some $X\in\Omega'$. Indeed, \[\Phi(\sigma)_x=\frac{1-[\sigma\cdot\bone_{\bbZ^d\setminus\Lambda}]_{x+T_d}}{2}=\frac{1-\prod_{y\in x+T_d}1}{2}=0\] for all $x\in\bbZ^d\setminus\Lambda-T_d$. To see that $\Phi(\sigma)$ is admissible, it suffices to change one spin at a time starting with $\bone_{\Lambda}$ and take the image under $\Phi$. Hence, $\Phi(\Omega_\Lambda)\subseteq\{\1_{-X}\mid X\in\Omega'\}$.

Fix $X\in\Omega'$. By Definition~\ref{def:admissible}, $\1_{-X}=\sum_{y\in Y}\1_{y-T_d}$ for a finite set $Y\subset\bbZ^d$. Assume for a contradiction that $Y\not\subseteq\Lambda$. Assume that $z\in Y$ is such that, for any $z'\in Y$, if $z_i'\ge z_i$ for all $i\in\{1,\dots,d\}$, then $z'=z$ and further assume that $z_1> 0$ (we proceed analogously for other coordinates). Then $\1_{-X}(z)=\sum_{y\in Y}\1_{y-T_d}(z)=\1_{z-T_d}(z)=1$, contradicting the assumption $-X\subseteq\Lambda-T_d$. Assume that $z\in Y$ is such that $\langle z,(1,\dots,1)\rangle=\min_{y\in Y}\langle y,(1,\dots,1)\rangle<-n$ and $z_1\le z_1'$ for any $z'\in Y$ with $\langle z',(1,\dots,1)\rangle=\langle z,(1,\dots,1)\rangle$. Then $\1_{-X}(z-e_1)=\sum_{y\in Y}\1_{y-T_d}(z-e_1)=\1_{z-T_d}(z-e_1)=1$, again contradicting $-X\subseteq\Lambda-T_d$. This proves that $Y\subseteq\Lambda$ (recall \eqref{eq:def:Lambda:dn}). Hence, $\Phi((-\bone_Y)\cdot\bone_{\Lambda\setminus Y})=\1_{-X}$, proving surjectivity.

If there are two preimages of $\1_{-X}$, taking their point-wise product, we obtain $\sigma\in\Omega_\Lambda$ with $\Phi(\sigma)=\bzero$. Proceeding as above, we get that the support of $\sigma$ is contained in $\Lambda$, but since $\varnothing$ is admissible and contained in any translate of $\Lambda$, this implies that $\sigma=\bone$, yielding the desired injectivity.
\end{proof}
By Theorem~\ref{th:complexity:2}, let $X\in\Omega'$ be such that any chain from $X$ to $\varnothing$ goes through 
\[\cX\coloneqq\left\{X'\in\Omega'\middle|\;|X'\setminus X|\ge n^{\alpha_d}+3|X\setminus X'|/2\right\}.\]
By Lemma~\ref{lem:admissible:preimage}, let $\sigma=\Phi^{-1}(\1_{-X})$. Let $\cG\subset\Omega_\Lambda$ be the set of configurations which can be reached from $\bone_\Lambda$ by switching the state of one spin in $\Lambda$ at a time without going through $\partial\cG\coloneqq\Phi^{-1}(\{\1_{-X'}\mid X'\in\cX\})$. In particular, $\sigma\not\in\cG\ni\bone_\Lambda$.

Thus, 
\begin{equation}
\label{eq:finite:lower:var}
\var_{\Lambda}^\bone(\1_\cG)\ge \frac12\min\left(\mu_\Lambda^\bone(\bone_\Lambda),\mu_{\Lambda}^\bone(\sigma)\right)=\frac{\mu_\Lambda^\bone(\sigma)}2.
\end{equation}
Moreover, by \eqref{eq:def:H}, for $X'\in\Omega'$, \[H_\Lambda^\bone(\bone_\Lambda)=-|\Lambda-T_d|=H_\Lambda^\bone\left(\Phi^{-1}\left(\1_{-X'}\right)\right)-2|X'|,\]
so
\begin{equation}
\label{eq:finite:lower:mu:sigma}
\mu^\bone_\Lambda(\sigma)=\frac{e^{-2\beta|X|}}{\sum_{X'\in\Omega'}e^{-2\beta|X'|}}\ge\frac{e^{-2\beta|X|}}{(1+e^{-2\beta})^{|\Lambda-T_d|}}\ge \exp\left(-2\beta|X|-e^{-2\beta}|\Lambda-T_d|\right).\end{equation}

Similarly, recalling \eqref{eq:def:D}, we get
\begin{align}
\nonumber\cD_\Lambda^\bone(\1_\cG)&{}\le 2 \sum_{x\in\Lambda}\sum_{\sigma'\in\Omega_\Lambda}\mu_\Lambda^\bone(\sigma')\1_{\sigma'\not\in\cG\ni(\sigma')^x}\le 2|\Lambda|\mu^\bone_\Lambda(\partial\cG)=2|\Lambda|\sum_{X'\in\cX}\mu_\Lambda^\bone\left(\Phi^{-1}\left(\1_{-X'}\right)\right)\\
\nonumber&{}\le 2|\Lambda|\sum_{X'\in\cX}e^{-2\beta|X'|}\le 2|\Lambda|e^{-2\beta|X|}\sum_{X'\in\cX}e^{-2\beta(4n^{\alpha_d}/5+(|X'\setminus X|+|X\setminus X'|)/5)}\\
\nonumber&{}\le 2|\Lambda|e^{-2\beta(|X|+4n^{\alpha_d}/5)}\sum_{Y\subseteq\Lambda-T_d}e^{-2\beta|Y|/5}=2|\Lambda|e^{-2\beta(|X|+4n^{\alpha_d}/5)}\left(1+e^{-2\beta/5}\right)^{|\Lambda-T_d|}\\
&{}\le 2|\Lambda|\exp\left(-2\beta(|X|+4n^{\alpha_d}/5)+|\Lambda-T_d|e^{-2\beta/5}\right).
\label{eq:finite:lower:D}
\end{align}

Finally, recalling \eqref{eq:def:trel} and combining \eqref{eq:finite:lower:var}, \eqref{eq:finite:lower:mu:sigma} and \eqref{eq:finite:lower:D}, we obtain
\begin{align*}\trel^{\Lambda,\bone}&{}\ge \frac{\exp(8\beta n^{\alpha_d}/5-|\Lambda-T_d|e^{-2\beta/5}-|\Lambda-T_d|e^{-2\beta})}{4|\Lambda|}\\&{}\ge \exp\left(8\beta n^{\alpha_d}/5-2^{d+1}-(d+2)\ln(2d)-2\beta/5\right)\ge\frac{e^{\beta n^{\alpha_d}}}{C},\end{align*}
since $|\Lambda|\le|\Lambda-T_d|\le (2n)^d\le 2^de^{2\beta/5}$, taking $C=(2d)^{d+2}e^{2^{d+1}}$.
\end{proof}

\subsection{Upper bound via bisection}
In this section we prove the upper bound of Theorem~\ref{th:2k:asymptotics} via the bisection and canonical paths techniques.
\begin{proof}[Proof of the upper bound of Theorem~\ref{th:2k:asymptotics}]
Fix $d=2$ and $\beta>0$. The result follows immediately, once we prove by induction that, for $k\ge 0$
\[\trel^{\bbT_{2^k}}\le \left(41472\cdot e^{2\beta}\right)^{k}.\]
The base case $k=0$ is clear and uses that $d$ is even (the dynamics is not irreducible for $d$ odd).

Assume $k\ge 1$ and let $n=2^k$. Recall $\Phi$ and $p$ from \eqref{eq:def:Phi} and \eqref{eq:def:p}. As in Section~\ref{subsec:lower:bottleneck}, let $\pi$ denote the product Bernoulli measure with parameter $p$ on $\Omega'_{\bbT_{n}}=\bbF_2^{\bbT_n}$ and $\var'$ be its variance. By Lemma~\ref{lem:Phi:bij}, we may reason in terms of the lamp Markov chain with generator as in \eqref{eq:def:L:prime}. For $r=(r_1,\dots,r_d)\in\{0,1\}^d$, let 
\[\bbT^r_n=\left\{(x_1,\dots,x_d)\in\bbT_n\mid\forall i\in\{1,\dots,d\},x_i=r_i\pmod 2\right\}=r+2\bbT_n,\]
taking into account that $n$ is even. For $A\subseteq\bbT_n$ and $\omega\in\Omega'_{\bbT_n}$, we write $\omega^A=\omega+\1_A$. Let
\begin{align}
\cL^rf(\omega)&{}=\sum_{x\in\bbT_n^r}\frac{\pi(\omega^{x-2T_d})}{\pi(\omega)+\pi(\omega^{x-2T_d})}\left(f\left(\omega^{x-2T_d}\right)-f(\omega)\right),\\
\label{eq:def:Dr}\cD^r(f)&{}=\sum_{\omega\in\Omega'_{\bbT_n}}\sum_{x\in\bbT_n^r}\frac{\pi(\omega)\pi(\omega^{x-2T_d})}{\pi(\omega)+\pi(\omega^{x-2T_d})}\left(f\left(\omega^{x-2T_d}\right)-f(\omega)\right)^2=-\pi\left(f\cL^rf\right),\\
\label{eq:def:trelr}\trel^r&{}=\left(\inf\left\{\frac{\cD^r(\pi_{\bbT_n\setminus\bbT_n^r}(f))}{\var'\left(\pi_{\bbT_n\setminus\bbT_n^r}(f)\right)}\middle| f\colon\Omega'_{\bbT_n}\to\bbR,\var'(\pi_{\bbT_n\setminus\bbT_n^r}(f))\neq0\right\}\right)^{-1}.
\end{align}

Fix $f\colon\Omega'_{\bbT_n}\to\bbR$. 
By the Poincar\'e inequality for independent Glauber dynamics (the Efron--Stein inequality) and \eqref{eq:def:trelr}, we have
\begin{equation}
\label{eq:var'}\var'(f)\le \sum_{r\in\{0,1\}^d}\pi\left(\var'_{\bbT^r_{n}}(f)\right)\le \sup_{r\in\{0,1\}^d}\trel^{r}\sum_{r\in\{0,1\}^d}\pi\left(\cD^r(f)\right)=\trel^{\bbT_{n/2}}\sum_{r\in\{0,1\}^d}\pi\left(\cD^r(f)\right).\end{equation}

We finally seek to relate each term in $\pi(\cD^r(f))$ to appropriate terms featuring in $\cD(f)$. Consider $r\in\{0,1\}^d$, $x\in\bbT^r_{n}$ and $\omega\in\Omega'_{\bbT_n}$. Set $\Omega''=\{\omega\in\Omega'_{\bbT_n}\mid\sum_{y\in x-2T_d}\omega_y\in\{2,3\}\}$. In view of the symmetry between $\omega$ and $\omega^{x-2T_d}$ in \eqref{eq:def:Dr}, we may assume $\omega\in\Omega''$. Define the configurations $\omega=\omega^{(0)},\omega^{(1)},\omega^{(2)},\omega^{(3)}=\omega^{x-2T_d}$ as indicated in Figure~\ref{fig:paths}. The Cauchy--Schwarz inequality gives
\[\left(f\left(\omega^{x-2T_d}\right)-f(\omega)\right)^2\le 3\sum_{i=1}^3\left(f\left(\omega^{(i)}\right)-f\left(\omega^{(i-1)}\right)\right)^2.\]
Summing over $\omega\in\Omega''$, we get
\newlength{\lhswidth}
\settowidth{\lhswidth}{${}= 6$} 
\begin{align}
\label{eq:paths}
&2\sum_{\omega\in\Omega''}\frac{\pi(\omega)\pi(\omega^{x-2T_d})}{\pi(\omega)+\pi(\omega^{x-2T_d})}\left(f\left(\omega^{x-2T_d}\right)-f(\omega)\right)^2\le 2\sum_{\omega\in\Omega''}\pi(\omega)\left(f\left(\omega^{x-2T_d}\right)-f(\omega)\right)^2\\
\nonumber&{}\le 6\sum_{i=1}^3\sum_{\omega\in\Omega''}\pi(\omega)\left(f\left(\omega^{(i)}\right)-f\left(\omega^{(i-1)}\right)\right)^2 \underbrace{\sum_{y\in x-T_d}\1_{\omega^{(i-1)}=(\omega^{(i)})^{y-T_d}}}_{=1}\underbrace{
\sum_{\omega'\in \Omega'_{\bbT_n}}\1_{\omega'=\omega^{(i)}}\frac{\pi(\omega')\pi((\omega')^{y-T_d})}{\pi(\omega')\pi((\omega')^{y-T_d})}}_{=1}\\
&{}= 6
\begin{multlined}[t][\dimexpr\linewidth - \lhswidth - 0.5em]
\sum_{\omega'\in\Omega'_{\bbT_n}}\sum_{y\in x-T_d}\left(f\left(\omega'\right)-f\left(\left(\omega'\right)^{y-T_d}\right)\right)^2\frac{\pi(\omega')\pi((\omega')^{y-T_d})}{\pi(\omega')+\pi((\omega')^{y-T_d})}\\
\times\underbrace{\sum_{i=1}^3\sum_{\omega\in\Omega''}\1_{\omega^{(i)}=\omega',\omega^{(i-1)}=(\omega')^{y-T_d}}\frac{\pi(\omega)(\pi(\omega')+\pi((\omega')^{y-T_d}))}{\pi(\omega')\pi((\omega')^{y-T_d})}}_{=:K(\omega,\omega',i,y)}.\end{multlined}
\nonumber
\end{align}

A direct inspection of Figure~\ref{fig:paths} shows that, for any $\omega\in\Omega''$ and any $j\in\{0,\dots,3\}$ we have $\|\omega^{(j)}\|_1-\|\omega\|_1\le 1$ (this property is specific to $d=2$). Consequently,
\begin{align*}\sup_{\omega'\in\Omega'_{\bbT_n},y\in x-T_d}K\left(\omega,\omega',i,y\right)&{}\le \frac2p\sup_{\omega'\in\Omega'_{\bbT_n}}\sum_{i=1}^3\sum_{\omega\in\Omega''}\1_{\omega^{(i)}=\omega'}\\
&{}\le \frac{6}{p}\left|\left\{\omega\in\Omega''\middle|\omega_{\bbT_n\setminus (x-T_d-T_d)}=\omega'_{\bbT_n\setminus(x-T_d-T_d)}\right\}\right|=\frac{6\cdot 2^6}{2p}.\end{align*}

Combining this with \eqref{eq:def:trel}, \eqref{eq:def:D}, \eqref{eq:var'}, \eqref{eq:paths}, we obtain
\[\trel^{\bbT_n}\le \trel^{\bbT_{n/2}}\cdot \frac{1152}p\cdot6\cdot\sup_{y\in\bbT_n}\left|\left\{x\in\bbT_n\middle| y\in x-T_d\right\}\right|=\trel^{\bbT_{n/2}}\cdot\frac{20736}{p},\]
completing the induction step in view of \eqref{eq:def:p}.
\end{proof}
\begin{figure}
    \centering
    \begin{tikzpicture}[x=0.5cm,y=0.5cm]
    \clip (-2.5,-2.3) rectangle (12.5,0.3);
\foreach \x in {0,...,2} {
    \draw[gray] (-\x, 0)--(0,-\x);
    \draw[gray] (-\x,0)--(-\x,-2+\x)--(0,-2+\x);}
  	\fill (0,0) circle(3pt);
  	\fill (-2,0) circle(3pt);
  	\fill (0,-2) circle(3pt);
  	\draw (1,-1) node {$\leftrightarrow$};
  	\begin{scope}[shift={(3.5,0)}]
\foreach \x in {0,...,2} {
    \draw[gray] (-\x, 0)--(0,-\x);
    \draw[gray] (-\x,0)--(-\x,-2+\x)--(0,-2+\x);}
  	\fill (0,0) circle(3pt);
  	\fill (-2,0) circle(3pt);
  	\fill (-1,-1) circle(3pt);
  	\fill (0,-1) circle(3pt);
  	\draw (1,-1) node {$\leftrightarrow$};
  	\begin{scope}[shift={(3.5,0)}]
\foreach \x in {0,...,2} {
    \draw[gray] (-\x, 0)--(0,-\x);
    \draw[gray] (-\x,0)--(-\x,-2+\x)--(0,-2+\x);}
  	\fill (-1,0) circle(3pt);
  	\fill (-2,0) circle(3pt);
  	\fill (-1,-1) circle(3pt);
  	\draw (1,-1) node {$\leftrightarrow$};
  	\begin{scope}[shift={(3.5,0)}]
\foreach \x in {0,...,2} {
    \draw[gray] (-\x, 0)--(0,-\x);
    \draw[gray] (-\x,0)--(-\x,-2+\x)--(0,-2+\x);}
  	\end{scope}
  	\end{scope}
  	\end{scope}
  	\end{tikzpicture}\qquad
    \begin{tikzpicture}[x=0.5cm,y=0.5cm]
    \clip (-2.5,-2.3) rectangle (12.5,0.3);
\foreach \x in {0,...,2} {
    \draw[gray] (-\x, 0)--(0,-\x);
    \draw[gray] (-\x,0)--(-\x,-2+\x)--(0,-2+\x);}
  	\fill (-2,0) circle(3pt);
  	\fill (0,-2) circle(3pt);
  	\draw (1,-1) node {$\leftrightarrow$};
  	\begin{scope}[shift={(3.5,0)}]
\foreach \x in {0,...,2} {
    \draw[gray] (-\x, 0)--(0,-\x);
    \draw[gray] (-\x,0)--(-\x,-2+\x)--(0,-2+\x);}
  	\fill (-2,0) circle(3pt);
  	\fill (-1,-1) circle(3pt);
  	\fill (0,-1) circle(3pt);
  	\draw (1,-1) node {$\leftrightarrow$};
  	\begin{scope}[shift={(3.5,0)}]
\foreach \x in {0,...,2} {
    \draw[gray] (-\x, 0)--(0,-\x);
    \draw[gray] (-\x,0)--(-\x,-2+\x)--(0,-2+\x);}
  	\fill (-1,0) circle(3pt);
  	\fill (0,-1) circle(3pt);
  	\draw (1,-1) node {$\leftrightarrow$};
  	\begin{scope}[shift={(3.5,0)}]
\foreach \x in {0,...,2} {
    \draw[gray] (-\x, 0)--(0,-\x);
    \draw[gray] (-\x,0)--(-\x,-2+\x)--(0,-2+\x);}
    \fill (0,0) circle(3pt);
  	\end{scope}
  	\end{scope}
  	\end{scope}
  	\end{tikzpicture}\\[4mm]
    \begin{tikzpicture}[x=0.5cm,y=0.5cm]
    \clip (-2.5,-2.3) rectangle (12.5,0.3);
\foreach \x in {0,...,2} {
    \draw[gray] (-\x, 0)--(0,-\x);
    \draw[gray] (-\x,0)--(-\x,-2+\x)--(0,-2+\x);}
  	\fill (0,0) circle(3pt);
  	\fill (0,-2) circle(3pt);
  	\draw (1,-1) node {$\leftrightarrow$};
  	\begin{scope}[shift={(3.5,0)}]
\foreach \x in {0,...,2} {
    \draw[gray] (-\x, 0)--(0,-\x);
    \draw[gray] (-\x,0)--(-\x,-2+\x)--(0,-2+\x);}
  	\fill (0,-2) circle(3pt);
  	\fill (0,-1) circle(3pt);
  	\fill (-1,0) circle(3pt);
  	\draw (1,-1) node {$\leftrightarrow$};
  	\begin{scope}[shift={(3.5,0)}]
\foreach \x in {0,...,2} {
    \draw[gray] (-\x, 0)--(0,-\x);
    \draw[gray] (-\x,0)--(-\x,-2+\x)--(0,-2+\x);}
  	\fill (-1,0) circle(3pt);
  	\fill (-1,-1) circle(3pt);
  	\draw (1,-1) node {$\leftrightarrow$};
  	\begin{scope}[shift={(3.5,0)}]
\foreach \x in {0,...,2} {
    \draw[gray] (-\x, 0)--(0,-\x);
    \draw[gray] (-\x,0)--(-\x,-2+\x)--(0,-2+\x);}
    \fill (-2,0) circle(3pt);
  	\end{scope}
  	\end{scope}
  	\end{scope}
  	\end{tikzpicture}\qquad
    \begin{tikzpicture}[x=0.5cm,y=0.5cm]
    \clip (-2.5,-2.3) rectangle (12.5,0.3);
\foreach \x in {0,...,2} {
    \draw[gray] (-\x, 0)--(0,-\x);
    \draw[gray] (-\x,0)--(-\x,-2+\x)--(0,-2+\x);}
  	\fill (0,0) circle(3pt);
  	\fill (-2,0) circle(3pt);
  	\draw (1,-1) node {$\leftrightarrow$};
  	\begin{scope}[shift={(3.5,0)}]
\foreach \x in {0,...,2} {
    \draw[gray] (-\x, 0)--(0,-\x);
    \draw[gray] (-\x,0)--(-\x,-2+\x)--(0,-2+\x);}
  	\fill (-2,0) circle(3pt);
  	\fill (-1,0) circle(3pt);
  	\fill (0,-1) circle(3pt);
  	\draw (1,-1) node {$\leftrightarrow$};
  	\begin{scope}[shift={(3.5,0)}]
\foreach \x in {0,...,2} {
    \draw[gray] (-\x, 0)--(0,-\x);
    \draw[gray] (-\x,0)--(-\x,-2+\x)--(0,-2+\x);}
  	\fill (0,-1) circle(3pt);
  	\fill (-1,-1) circle(3pt);
  	\draw (1,-1) node {$\leftrightarrow$};
  	\begin{scope}[shift={(3.5,0)}]
\foreach \x in {0,...,2} {
    \draw[gray] (-\x, 0)--(0,-\x);
    \draw[gray] (-\x,0)--(-\x,-2+\x)--(0,-2+\x);}
    \fill (0,-2) circle(3pt);
  	\end{scope}
  	\end{scope}
  	\end{scope}
  	\end{tikzpicture}
    \caption{The canonical paths used in the proof of the upper bound of Theorem~\ref{th:2k:asymptotics}. The other $28$ paths are obtained by taking the symmetric difference of these four with a configuration with some of the three internal lamps ON. This increases the number of ON lamps in the intermediate configurations no more than in the initial and final configurations of each path.}
    \label{fig:paths}
    
\end{figure}
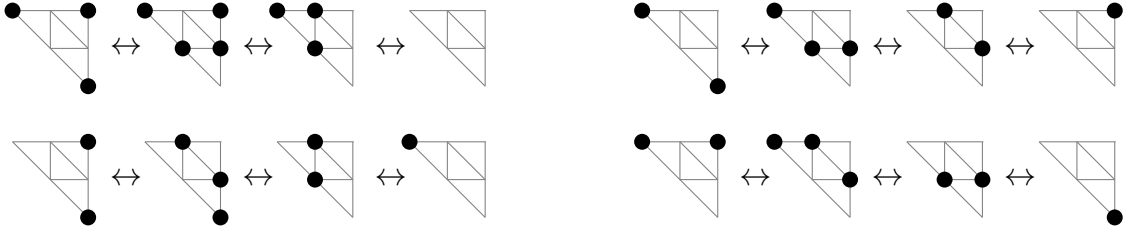

\subsection{Upper bound via the cycle generating function}
\label{subsec:proba:upper}
In this section we prove the upper bound of Theorem~\ref{th:main}. The main ingredient is Theorem~\ref{th:upper}. As a warm-up, we prove the upper bound of Theorem~\ref{th:finite}, which is an immediate corollary of the following trivial bound (see \cite{Martinelli99}*{Theorem~3.8} for a slightly better one which gives $n^{d-1}$ instead of $n^d$ in Theorem~\ref{th:finite}). We include its proof for completeness.
\begin{lemma}[Finite volume upper bound]
\label{lem:trivial}For any $d\ge 2$, $n\ge 1$, $\beta>0$ and $\tau\in\Omega_{\bbZ^d\setminus(-\Lambda_{d,n})}$, we have
\[\trel^{-\Lambda_{d,n},\tau}\le e^{2(\beta+1)|\Lambda_{d,n+3}|}.\]
\end{lemma}
\begin{proof}
    Recalling \eqref{eq:def:trel}, we consider $f\colon\Omega_{-\Lambda_{d,n}}\to\bbR$ and may assume that there exist $\sigma,\sigma'\in\Omega_{-\Lambda_{d,n}}$ with $f(\sigma)=0\le f(\sigma'')\le f(\sigma')=1$ for all $\sigma''\in\Omega_{-\Lambda_{d,n}}$. In particular, $\var_{-\Lambda_{d,n}}^\tau(f)\le 1$. On the other hand, we may fix a sequence $\sigma^{(0)},\dots,\sigma^{(m)}$ such that $\sigma^{(0)}=\sigma$, $\sigma^{(m)}=\sigma'$, $\|\sigma^{(i)}-\sigma^{(i-1)}\|_1=1$ for $i\in\{1,\dots,m\}$ (we change one spin at a time) and $m\le |\Lambda_{d,n}|$. Then there exists $k\in\{1,\dots,m\}$ such that $f(\sigma^{(k)})-f(\sigma^{(k-1)})\ge 1/m$. Write $\sigma^{(k)}-\sigma^{(k-1)}=2\1_x$, with $\1_x$ the characteristic function of $x\in-\Lambda_{d,n}$. Then
    \begin{align*}
    \cD^{\tau}_{-\Lambda_{d,n}}(f)&{}\overset{\eqref{eq:def:D}}\ge \mu_{-\Lambda_{d,n}}^\tau\left(\sigma^{(k)}\right)\var_x(f)\left(\sigma^{(k)}\right)\ge \mu_{-\Lambda_{d,n}}^\tau\left(\sigma^{(k)}\right)\cdot\frac{\min(\mu^\tau_{-\Lambda_{d,n}}(\sigma^{(k)}),\mu^\tau_{-\Lambda_{d,n}}(\sigma^{(k-1)}))}{2m^2(\mu^\tau_{-\Lambda_{d,n}}(\sigma^{(k)})+\mu^\tau_{-\Lambda_{d,n}}(\sigma^{(k-1)}))}
    \\
    &{}\overset{\eqref{eq:def:H}}\ge\frac{\mu^\tau_{-\Lambda_{d,n}}(\sigma^{(k)})}{2m^2(1+e^{2(d+1)\beta})}\ge \frac{\min\{\mu^\tau_{-\Lambda_{d,n}}(\bar\sigma)\mid\bar \sigma\in\Omega_{-\Lambda_{d,n}}\}}{4|\Lambda_{d,n}|^2e^{2(d+1)\beta}}
    \\
    &{}
    \overset{\eqref{eq:def:H}}\ge\frac{\min_{\bar\sigma}e^{-\beta H^\tau_{-\Lambda_{d,n}}(\bar\sigma)}}{|\Omega_{-\Lambda_{d,n}}|\max_{\bar\sigma} e^{-\beta H^\tau_{-\Lambda_{d,n}}(\bar\sigma)}}\cdot\frac{1}{4|\Lambda_{d,n}|^2e^{2(d+1)\beta}}\overset{\eqref{eq:def:H}}\ge\frac{e^{-2\beta|-\Lambda_{d,n}-T_d|}}{|\Omega_{-\Lambda_{d,n}}|\cdot |\Lambda_{d,n}|^2\cdot4e^{2(d+1)\beta}}\\
    &{}\ge \frac{e^{-(2\beta+2)|\Lambda_{d,n+1}|}}{4e^{2(d+1)\beta}}
    \end{align*}
    using $n\ge 1,d\ge 2$ in the last inequality. Plugging this into \eqref{eq:def:trel} and using $n\ge 1,d\ge 2$ again, we obtain the desired inequality.
\end{proof}

The upper bound of Theorem~\ref{th:main} will use Lemma~\ref{lem:trivial} at a suitably chosen scale of order $e^{C\beta}$ with large $C$. Above this scale we will be able to apply Theorem~\ref{th:upper} to deduce that boundary conditions have little influence on the partition function and so also on the Boltzmann measure. This will allow us to use the path coupling method in finite volume and then transfer fast convergence to equilibrium to infinite volume.

\begin{proof}[Proof of the upper bound of Theorem~\ref{th:main}]
The starting point is the same high temperature expansion as in \cite{Chleboun17}*{Section~4.0.1}, which we recall for completeness. Fix a positive integer $d$, $\beta>0$ and set $\theta=\tanh(\beta)$. For any $u\in\{-1,1\}$, we have
\[e^{\beta u}=\cosh(\beta)(1+u\theta).\]
Consequently, for any finite $\Lambda\subset\bbZ^d$ and $\tau\in\Omega_{\bbZ^d\setminus\Lambda}$, we have
\begin{align*}
Z^\tau_\Lambda&{}=\sum_{\sigma\in\Omega_\Lambda}\prod_{x\in\Lambda-T_d}e^{\beta[\sigma\cdot\tau]_{x+T_d}}=(\cosh(\beta))^{|\Lambda-T_d|}\sum_{X\subset\Lambda-T_d}\theta^{|X|}\sum_{\sigma\in\Omega_\Lambda}\prod_{x\in X}[\sigma\cdot\tau]_{x+T_d}\\
&{}=(\cosh(\beta))^{|\Lambda-T_d|}\sum_{X\in\cC_{(-T_d,\Lambda-T_d)}}\theta^{|X|}\sum_{\sigma\in\Omega_\Lambda}\prod_{x\in X}[\sigma\cdot\tau]_{x+T_d},
\end{align*}

recalling the set of $(A,B)$-cycles $\cC_{A,B}$ from Section~\ref{subsubsec:cycles}. Let $\cZ_\Lambda^\tau=Z_\Lambda^\tau2^{-|\Lambda|}(\cosh(\beta))^{-|\Lambda-T_d|}$. Then 
\[\max_{\tau\in\Omega_{\bbZ^d\setminus\Lambda}}\left|\cZ_\Lambda^\tau-1\right|\le \sum_{X\in\cC_{(-T_d,\Lambda-T_d)}\setminus\{\varnothing\}}\theta^{|X|}.\]
Recalling \eqref{eq:def:Lambda:dn} and \eqref{eq:def:fdn}, we get
\[\max_{\tau\in\Omega_{\bbZ^d\setminus (-\Lambda_{d,n})}}\big|\cZ^\tau_{-\Lambda_{d,n}}-1\big|\le G_{d,n+1}(\theta)-1\]
for any positive integer $n$. 

The rest of the proof diverges from \cite{Chleboun17}. Set $m=\lceil1/(1-\theta)\rceil\ge 2$ and fix a positive integer $a_d\ge d^2$ such that Theorem~\ref{th:upper} holds and
\begin{equation}
\label{eq:nNconditions}
    2^{a_d-d^2}\ge 4(6d)^{d(d+1)}
\end{equation}
(the latter is clearly true for $a_d$ large enough). Then Theorem~\ref{th:upper} gives
\begin{equation}
\label{eq:Z:close:to:1}
n\ge m^{a_d}\Longrightarrow \max_{\tau\in\Omega_{\bbZ^d\setminus(-\Lambda_{d,n})}}\left|\cZ^\tau_{-\Lambda_{d,n}}-1\right|\le n^{-a_d}.\end{equation}
This result will combine with the following observation.

\begin{lemma}[Boundary condition influence]
\label{lem:measures}
Fix finite $\Lambda'\subset\Lambda\subset\bbZ^d$ and $x\in\bbZ^d\setminus\Lambda$ such that $\min_{y\in\Lambda\setminus\Lambda'}\|x-y\|_1\ge 3$. For $\tau\in\Omega_{\bbZ^d\setminus\Lambda}$, define the marginal of $\mu^\tau_\Lambda$ on $\Lambda\setminus\Lambda'$ by
\[\mu^\tau_{\Lambda\setminus\Lambda'}\colon\Omega_{\Lambda\setminus\Lambda'}\to\bbR\colon\sigma\mapsto \sum_{\eta\in\Omega_{\Lambda'}}\mu^\tau_\Lambda(\eta\cdot\sigma).\]
Then, for any $\tau\in\Omega_{\Lambda\setminus\Lambda'}$, the marginal's Radon--Nikod\'ym derivative is uniformly bounded by
\[\frac{\mathrm d\mu^\tau_{\Lambda\setminus\Lambda'}}{\mathrm d\mu^{\tau^x}_{\Lambda\setminus\Lambda'}}\ge \frac{\cZ^{\tau^x}_\Lambda}{\cZ^\tau_\Lambda}\min_{\sigma\in\Omega_{\bbZ^d\setminus\Lambda'}}\frac{\cZ^{\sigma^x}_{\Lambda'}}{\cZ^{\sigma}_{\Lambda'}}.\]
\end{lemma}
\begin{proof}
Fix $\tau\in\Omega_{\bbZ^d\setminus\Lambda}$. For $\sigma\in\Omega_{\Lambda\setminus\Lambda'}$ and $\eta'\in\Omega_{\Lambda'}$, we have
\begin{align*}
\mu^\tau_{\Lambda\setminus\Lambda'}(\sigma)&{}=\sum_{\eta\in\Omega_{\Lambda'}}\mu^\tau_\Lambda(\eta\cdot\sigma)=\sum_{\eta\in\Omega_{\Lambda'}}\frac{e^{-\beta H_\Lambda^\tau(\eta\cdot\sigma)}}{Z^\tau_\Lambda}= \frac{1}{Z^{\tau}_\Lambda}\sum_{\eta\in\Omega_{\Lambda'}}e^{-\beta H_{\Lambda'}^{\tau\cdot\sigma}(\eta)}e^{\beta\sum_{y\in(\Lambda-T_d)\setminus(\Lambda'-T_d)}[\tau\cdot\sigma\cdot\eta]_{y+T_d}}\\
&{}=\frac{Z_{\Lambda'}^{\tau\cdot\sigma}}{Z^{\tau}_\Lambda}e^{\beta\sum_{y\in(\Lambda-T_d)\setminus(\Lambda'-T_d)}[\tau^x\cdot\sigma\cdot \eta']_{y+T_d}}=\frac{Z^{\tau\cdot\sigma}_{\Lambda'}Z^{\tau^x}_\Lambda}{Z^\tau_\Lambda Z^{\tau^x\cdot\sigma}_{\Lambda'}}\mu^{\tau^x}_{\Lambda\setminus\Lambda'}(\sigma)=\frac{\cZ^{\tau\cdot\sigma}_{\Lambda'}\cZ^{\tau^x}_\Lambda}{\cZ^{\tau^x\cdot\sigma}_{\Lambda'}\cZ^\tau_\Lambda}\mu^{\tau^x}_{\Lambda\setminus\Lambda'}(\sigma),\end{align*}
where we used the fact that $\sum_{y\in(\Lambda-T_d)\setminus(\Lambda'-T_d)}[\tau\cdot\sigma\cdot\eta]$ does not depend on $\tau_x$ nor on $\eta$ because of the range of $y$ and the assumption on $x$.
\end{proof}

Fix \begin{align}
\label{eq:def:N}
n&{}=m^{a_d},&N&{}=(3d)^{d+1}(n+1)^d,
\end{align}
so that 
\begin{align}
\label{eq:nNconditions2}
    N^d/(3d^d)&{}\ge 3^dN^{d-1}d(n+1)^d\ge (n+1)^d(d+1)(N+2)^{d-1}\\
\nonumber&{}\ge (N+2)^{2d-1}(d+1)N^{-d}\ge (N+2)^{2d-1}(d+1)(6d)^{-(d+1)d}n^{-d^2}\\
\label{eq:nNconditions3}&{}\ge (N+1)^d(N+2)^{d-1}(d+1)4n^{-a_d},
\end{align}
thanks to \eqref{eq:nNconditions} in the last inequality. Let us consider the following block dynamics on a $d$-dimensional torus $\bbT=\bbT_L$ of side length $L>N+1$ which we first describe via its graphical representation. Endow each $x\in\bbT$ with a Poisson clock of unit rate which rings at times $P_x\subset(0,\infty)$. When the clock at $x$ rings, change the current configuration $\sigma$ to $\tau\cdot\sigma'$ with $\tau=\sigma_{\bbT\setminus (x-\Lambda_{d,N})}$ and $\sigma'$ chosen at random according to $\mu_{x-\Lambda_{d,N}}^{\tau}$; that is, we fully resample the configuration in a simplex region conditionally on the configuration outside of it. In other words, this is the Markov process with generator and Dirichlet form given by
\begin{align*}\cL_\bbT f&{}=\sum_{x\in\bbT}\left(\mu_{x-\Lambda_{d,N}}(f)-f\right),&\cD_\bbT(f)&{}=\sum_{x\in\bbT}\mu_\bbT\left(\var_{x-\Lambda_{d,N}}(f)\right).\end{align*}

\begin{figure}
    \centering
    \begin{tikzpicture}[x=-0.5cm,y=-0.5cm]
\foreach \x in {0,...,8} {
    \draw[gray] (\x, 0)--(0,\x);
    \draw[gray] (\x,0)--(\x,8-\x)--(0,8-\x);}
\foreach \x in {0,...,9} {
  	\fill (\x,-1) circle(3pt);
  	\fill (-1,\x) circle(3pt);
  	\fill (\x,9-\x) circle(3pt);}
  	\draw (0,0) node[yshift=0.15cm]{$y$};
  	\draw (6,-1) node[above]{$x$};
  	\draw[very thick,red] (5,0)node[black,yshift=0.15cm]{$z$}--(8,0)--(5,3)--cycle;
  	\end{tikzpicture}
  	\caption{Illustration of the contracting coupling in the proof of the upper bound of Theorem~\ref{th:main}. The simplex depicted is $y-\Lambda_{d,N}$ and the elements of $\partial\Lambda$ are drawn as dots. The simplex $z-\Lambda_{d,n}$ is thickened in red for one possible choice of $z$, given $x$ and $y$. In the figure, we take $d=2$, $N=8$, $n=3$, but bear in mind that  by \eqref{eq:def:N} in reality $n$ is much smaller than $N$.}
    \label{fig:block:dynamics}
\end{figure}
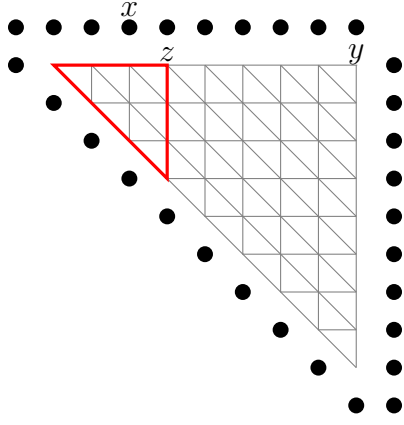

We next resort to the path coupling method (see \cite{Levin09}*{Chapter~14}). It is used to show that a Markov chain (the block dynamics we just defined) converges to its invariant measure exponentially fast. To be able to apply it, we need to couple two copies of the chain with different initial conditions in such a way that their $\ell^1$-distance (number of discrepancies) is contracted in expectation. By a well-known reduction, it suffices to treat initial conditions that only differ at a single site.

Consider a configuration $\sigma\in\Omega_\bbT$ and $x\in\bbT$. Recall that $\sigma^x$ is the configuration that only differs from $\sigma$ at $x$. We construct a coupling between the Markov process $(\sigma(t))_{t\ge 0}$ with initial state $\sigma$ and the one with initial state $\sigma^x$ denoted by $(\sigma'(t))_{t\ge 0}$ as follows. Let $(\xi_{y,t})_{y\in\bbT,t\in P_y}$ be a sequence of i.i.d.\ Bernoulli random variables with parameter $1-4n^{-a_d}$. We use the same Poisson clocks $(P_y)_{y\in\bbT}$ for both processes. Both processes remain constant outside the clock ring times $\bigcup_{y\in \bbT}P_y$. Assume the coupling is constructed until the clock ring time $t\in P_y$ for some $y\in\bbT$. Set $\partial\Lambda=y+((-\Lambda_{d,N}+(-T_d)+T_d)\setminus(-\Lambda_{d,N}))$, $\tau=\sigma_{\bbT\setminus(y-\Lambda_{d,N})}(t-)$, $\tau'=\sigma_{\bbT\setminus(y-\Lambda_{d,N})}'(t-)$. If $x\in\partial\Lambda$, fix $z\in\bbT$ so that $z-\Lambda_{d,n}\subset y-\Lambda_{d,N}$ and $x$ is at $\ell^1$-distance at least 3 from $(y-\Lambda_{d,N})\setminus(z-\Lambda_{d,n})$,
see Figure~\ref{fig:block:dynamics}. We construct $\sigma(t)$ and $\sigma'(t)$ by the following rules.
\begin{enumerate}
    \item If $x\in \bbT\setminus(y-\Lambda_{d,N}+(-T_d)+T_d)$ and $\tau^x=\tau'$, we define $\sigma(t)=\tau\cdot\eta$ and $\sigma'(t)=\tau'\cdot\eta$ with $\eta$ distributed according to $\mu^\tau_{y-\Lambda_{d,N}}=\mu^{\tau'}_{y-\Lambda_{d,N}}$. That is, if $x$ is far from the update block, it does not influence the update.
    \item If $x\in y-\Lambda_{d,N}$ and $\tau=\tau'$, we define $\sigma(t)=\sigma'(t)=\tau\cdot\eta$ with $\eta$ distributed according to $\mu^\tau_{y-\Lambda_{d,N}}=\mu^{\tau'}_{y-\Lambda_{d,N}}$. That is, if $x$ is inside the block and the boundary conditions coincide, the two processes coalesce.
    \item If $\tau'=\tau^x$, $x\in\partial\Lambda$ and $\xi_{y,t}=1$, we define $\sigma(t)=\tau\cdot\kappa\cdot\eta$ and $\sigma'(t)=\tau'\cdot\kappa\cdot\eta'$ with $\kappa$ distributed according to $\mu^\tau_{(y-\Lambda_{d,N})\setminus(z-\Lambda_{d,n})}$ and, conditionally on $\kappa$, $(\eta,\eta')$ distributed according to $\mu^{\tau\cdot\kappa}_{z-\Lambda_{d,n}}\otimes\mu^{\tau'\cdot\kappa}_{z-\Lambda_{d,n}}$. That is, if $x$ is on the boundary of the block, the boundary conditions coincide except at $x$, and we are lucky, we ensure that the two processes coincide outside the small simplex close to $x$.
    \item If $\tau'=\tau^x$, $x\in\partial\Lambda$ and $\xi_{y,t}=0$, we define $\sigma(t)=\tau\cdot\eta$ and $\sigma'(t)=\tau'\cdot\eta'$ with $\eta$ and $\eta'$ independent with distributions such that the total marginals with the previous case are $\mu^\tau_{y-\Lambda_{d,N}}$ and $\mu^{\tau'}_{y-\Lambda_{d,N}}$. That is, if $x$ is on the boundary of the block, the boundary conditions coincide except at $x$, and we are unlucky, we perform the update somehow.
    \item Otherwise, we define $\sigma(t)=\tau\cdot\eta$ and $\sigma'(t)=\tau'\cdot\eta'$ with $(\eta,\eta')$ with law $\mu^\tau_{y-\Lambda_{d,N}}\otimes\mu^{\tau'}_{y-\Lambda_{d,N}}$. That is, in all other cases, the two processes perform the block update independently.
\end{enumerate}
We note that this construction is possible, since Lemma~\ref{lem:measures} and \eqref{eq:Z:close:to:1} give
\begin{align*}\frac{\mathrm d\mu_{(y-\Lambda_{d,N})\setminus(z-\Lambda_{d,n})}^{\tau^x}}{\mathrm d\mu_{(y-\Lambda_{d,N})\setminus(z-\Lambda_{d,n})}^{\tau}}&{}\ge \frac{1-N^{-a_d}}{1+N^{-a_d}}\cdot\frac{1-n^{-a_d}}{1+n^{-a_d}}\overset{N\ge n}\ge\left(\frac{1-n^{-a_d}}{1+n^{-a_d}}\right)^2\\
&{}\ge(1-n^{-a_d})^4\ge  1-4n^{-a_d}{=\mathbb P(\xi_{y,t}=1)}.\end{align*}
We next claim that this coupling is contracting for the natural graph metric on $\Omega_\bbT$.
\begin{lemma}[Contraction]
\label{lem:contraction}
Under the above coupling, setting $\phi\colon[0,\infty]\colon t\mapsto \bbE\|\sigma(t)-\sigma'(t)\|_1/2$, we have
\[\phi'(0)\le -\frac{N^d}{3d^d}\phi(0).\]
\end{lemma}
\begin{proof}
By construction, $\phi(0)=1$, since $(\sigma(0))^x=\sigma'(0)$. We consider each of the first four cases in the coupling separately (the fifth one requires two updates to occur, so it does not feature in the derivative at 0). Notice that, in these cases, $\phi$ increases by $0$; $-1$; at most $|\Lambda_{d,n}|$; at most $|\Lambda_{d,N}|$, respectively. Hence, taking into account the rate at which each updates of each type occur, we get
\begin{align*}
&\phi'(0)\\
&{}\le 0\cdot (|\bbT|-|\Lambda_{d,N}+(-T_d)+T_d|) - 1 \cdot|\Lambda_{d,N}|+|\Lambda_{d,n}|\cdot|\partial\Lambda|\cdot(1-4n^{-a_d})+|\Lambda_{d,N}|\cdot|\partial\Lambda|\cdot 4n^{-a_d}\\
&{}\le-(N/d)^d+(n+1)^d\cdot (d+1)(N+2)^{d-1}+(N+1)^d\cdot (d+1)(N+2)^{d-1}\cdot 4n^{-a_d}\\
&{}\le -N^d/(3d^d),
\end{align*}
recalling \eqref{eq:nNconditions2} and \eqref{eq:nNconditions3} in the last inequality.
\end{proof}

By the standard technique of path coupling (see \cite{Levin09}*{Theorem~14.6}, which applies similarly to the continuous time setting), Lemma~\ref{lem:contraction} implies that, for any probability measure $\nu$ on $\Omega_\bbT$ and $t\ge 0$
\begin{equation}
\label{eq:dTV}
\dtv\left(\nu P^\bbT_t,\mu_\bbT\right)\le W_1\left(\nu P_t^\bbT,\mu_\bbT\right)\le e^{-tN^d/(3d^d)}|\bbT|,\end{equation}
where $P_t^\bbT$ denotes the semi-group of $\cL_\bbT$ and $2W_1$ is the $\ell_1$-Wasserstein distance and $\dtv$ is the total variation distance.

We next seek to compare the invariant measure $\mu_\bbT$ on the torus to the semi-group $P^{(N)}_t$ of the block dynamics on $\bbZ^d$, whose Dirichlet form is given by
\begin{equation}
\label{eq:def:DN}
\cD^{(N)}(f)=\sum_{x\in\bbZ^d}\mu\left(\var_{x-\Lambda_{d,N}}(f)\right)\end{equation}
and whose relaxation time we denote by $\trel^{(N)}$ (recall \eqref{eq:def:trel}). This follows from the classical fact that, in a finite-range model, information cannot travel too fast, as we recall next.
\begin{lemma}[Exponential ergodicity]
\label{lem:decay}
Let $a$ be a positive integer, $f$ be a local function with $\supp f \subset[-a,a)^d$ and $\sup f-\inf f\le 1$. Then, for any $t\ge a$ large enough depending on $d$, if the size $L$ of $\bbT=\bbT_L$ satisfies $L/t\in[(5N)^{5d},(6N)^{5d}]$, we have
\[\left\|P_t^{(N)}f-\mu_\bbT(f)\right\|_\infty\le e^{-tN^d/(4d^d)}.\]
\end{lemma}
\begin{proof}
We define $f_\bbT\colon \Omega_\bbT\to \bbR:\sigma^{\bbT}\mapsto f(\sigma)$, where $\sigma$ is any configuration in $\Omega$ such that $\sigma_x=\sigma^\bbT_x$ for $x\in[-L/2,L/2)^d\cap\bbZ^d\equiv\bbT$(we assume $L$ divisible by 2 for simplicity).

Consider a coupling of the block dynamics $(\sigma^\bbT(t))_{t\ge 0}$ on $\bbT$ and the one on $\bbZ^d$ denoted $(\sigma(t))_{t\ge 0}$ with consistent initial conditions $\sigma^\bbT(0)=\sigma(0)_{[-L/2,L/2)^d\cap\bbZ^d}$, constructed as follows. Let $(P_x)_{x\in\bbZ^d}$ be Poisson processes on $[0,\infty)$ with unit intensity. At times $t\in P_x$ for $x\in\bbT\equiv[-L/2,L/2)^d\cap\bbZ^d$, we update both $\sigma^\bbT(t-)$ and $\sigma(t-)$ to the same state in the block $x-\Lambda_{d,N}$, if the two processes have the same boundary condition on the boundary of this simplex (in the sense of Figure~\ref{fig:block:dynamics}). If the boundary conditions do not match, the updates are performed independently. The updates at sites in $\bbZ^d\setminus [-L/2,L/2)^d$ are performed only in $\sigma(t)$, governed by independent Poisson processes.

Clearly, if $f(\sigma(t))\neq f_\bbT(\sigma^\bbT(t))$ for some $t\ge 0$, then there exists a sequence $(x_i)_{i=0}^m$ of points in $[-L/2,L/2)^d\cap\bbZ^d$ and decreasing sequence of times $(t_i)_{i=0}^m$ such that: $t_0\le t$; $\supp f\cap(x_0-\Lambda_{d,N})\neq \varnothing$; $x_m-\Lambda_{d,N}+\Lambda_{d,N}-T_d+T_d\not \subset [-L/2,L/2)^d$; for all $i\in\{1,\dots,m\}$ we have $(x_{i-1}-\Lambda_{d,N}-T_d+T_d)\cap(x_i-\Lambda_{d,N})\neq\varnothing$; for all $i\in\{0,\dots,m\}$ we have $t_i\in P_{x_i}$. Note that any such sequence satisfies $m+2\ge (L/2-a)/(2N+1)\ge L/(5N)+1$. 

The probability that clocks ring in order along a given sequence $(x_i)_{i=0}^k$ vertices before time $t$ is exactly $\bbP(N_t\ge k+1)$, where $N_t$ denotes a Poisson random variable of mean $t$. In order to bound it, we recall the Bennett inequality: for 
$\alpha\ge 1$,
\[\bbP(N_t\ge t\alpha)\le e^{-t(\alpha\ln(\alpha)-\alpha+1)},\]
which follows from Markov's inequality applied to the moment generating function $\bbE[e^{xN_t}]=e^{(e^x-1)t}$. Hence,
\begin{align}
\nonumber\left\|P^\bbT_tf_\bbT-P^{(N)}_tf\right\|_\infty&{}\le(2a)^d|\Lambda_{d,N}-\Lambda_{d,N}+T_d-T_d|^{L/(5N)}\bbP(N_t\ge L/(5N))\\
\nonumber&{}\le (2a)^d(N+2)^{2dL/(5N)}e^{-L/(5N)(\ln (L/(5Nt))-1)}\\
\label{eq:PtTZd}&{}\le (2a)^de^{-2L/(5N)}\le e^{-L/(5N)},
\end{align}
since $L\ge (5N)^{5d}t\ge 5e^3tN(N+2)^{2d}$ and $e^{L/(5N)}\ge e^{2dt}\ge e^{2da}\ge (2a)^d$. Then,
\begin{align*}
\left\|P^{(N)}_tf-\mu_\bbT(f)\right\|_\infty&{}\overset{\eqref{eq:PtTZd}}\le e^{-L/(5N)}+\left\|P_t^\bbT f-\mu_\bbT(f)\right\|_\infty\overset{\eqref{eq:dTV}}\le e^{-t(5N)^{5d-1}}+e^{-tN^d/(3d^d)}L^d\\
&{}\le e^{-tN^d}+e^{-t N^d/(3d^d)}(6N)^{5d^2}t^d\le e^{-tN^d/(4d^d)}.\qedhere
\end{align*}
\end{proof}

Thanks to the range of possible values of $L$ in Lemma~\ref{lem:decay}, we conclude that $\mu_\bbT(f)$ has a limit $\mu(f)$ (as the size $L$ of $\bbT$ diverges) and $P_t^{(N)} f$ converges exponentially fast to it in $\ell^\infty$ at rate $N^d/(4d^d)$. In particular, the infinite volume process is ergodic, there is a unique invariant measure and therefore a unique Gibbs measure (see \cite{Liggett05}*{Theorem~IV.2.15}). Moreover, this exponential ergodicity implies (see the remark at the end of \cite{Martinelli99}*{Section 3.5})
\begin{equation}
\label{eq:trelN:upper}
\trel^{(N)}\le 4d^d/N^d.\end{equation}

Furthermore, for any local function $f\colon\Omega\to\bbR$ we have
\begin{align}
\label{eq:trel:small:scales}\var (f)&{}\overset{\eqref{eq:def:trel}}\le \trel^{(N)}\cD^{(N)}(f)\overset{\eqref{eq:def:DN}}=\trel^{(N)}\sum_{x\in\bbZ^d}\mu\left(\var_{x-\Lambda_{d,N}}(f)\right)\\
\nonumber&{}\overset{\eqref{eq:def:trel}}\le \trel^{(N)}\sum_{x\in\bbZ^d}\sup_{\tau\in\Omega_{\bbZ^d\setminus(-\Lambda_{d,N})}}\trel^{-\Lambda_{d,N},\tau}\mu\left(\cD_{x-\Lambda_{d,N}}(f)\right)\\
\nonumber&{}\overset{\eqref{eq:def:D}}=\trel^{(N)}\sup_{\tau}\trel^{-\Lambda_{d,N},\tau}\sum_{x\in\bbZ^d}\mu\left(\sum_{y\in x+\Lambda_{d,N}}\mu_{x-\Lambda_{d,N}}(\var_y(f))\right)\\
&{}\overset{\eqref{eq:def:D}}=\trel^{(N)}\sup_{\tau}\trel^{-\Lambda_{d,N},\tau}|\Lambda_{d,N}|\cD(f).
\end{align}

Using \eqref{eq:def:trel}, \eqref{eq:trel:small:scales}, \eqref{eq:trelN:upper} and Lemma~\ref{lem:trivial}, we get
\[\trel\le \frac{4d^d}{N^d}(N+1)^de^{2(\beta+1)(N+4)^d}\le e^{(\beta+1)m^{C'}}\le Ce^{e^{C\beta}},\]
recalling \eqref{eq:def:N} and $m=\lceil 1/(1-\tanh(\beta))\rceil\ge 2$ and taking suitably large $C',C>0$ depending only on $d$ (and $a_d$).
\end{proof}

\section{Applications to geometric group theory}
\label{ss:applications}

We saw in Section~\ref{sec:intro} how modifying switches to affect lamps can be understood as a game, whose goal is to turn off all lamps by an appropriate sequence of switches. \emph{Combinatorial group theory}, insofar as it is concerned with computations in group presentations, generalises broadly this idea; see in particular the beautiful application by Conway and Lagarias to tiling problems~\cite{Conway90}. 
We start by reviewing the basics.

Consider a group presentation $G=\langle S\mid R\rangle$. This is a compact description of a group $G$ with the property that every element may be represented as a word over $S\cup S^{-1}$, and such that two such representations define the same group element if one word may be obtained from the other by repeated insertion of an word from $R\cup R^{-1}$, and free cancellation $(s\cdot s^{-1}=s^{-1}\cdot s=1)$. For example, $G=\langle x,y\mid x^{-1} y^{-1} x y\rangle$ defines the group $\bbZ^2$; every element may be uniquely represented as a word of the form $x^m y^n$ with $m,n\in\bbZ$.

Consider a word $w$ over $S\cup S^{-1}$ that represents the identity in $G$. The process of reducing $w$ to $1$ by repeated insertions of relators may be visualised in a \emph{van Kampen diagram}: this is a planar graph all of whose edges are oriented and labeled by $S$, all of whose faces read an element of $R\cup R^{-1}$ on their boundary, if edges in reversed orientation are interpreted as $S^{-1}$; and whose external boundary reads $w$. The \emph{weight} of $w$ is the minimal number of faces of a van Kampen diagram for $w$, or equivalently the minimal $n$ such that $w$ may be written in the free group $F_S$ in the form
\begin{equation}\label{eq:Dehn}
w=\prod_{i=1}^n u_i r_i u_i^{-1},\quad r_i\in R\cup R^{-1}.
\end{equation}
The \emph{Dehn function} of $G$ is the function $\Delta(n)$ counting the maximal weight of all words of length at most $n$ that represent the identity. For example, the Dehn function of $\Z^2$, with its presentation given above, is $\Delta(n)=\lfloor (n/4)^2\rfloor$; the weight of $x^{-n}y^{-n}x^n y^n$ is $n^2$ and maximises the weight among words of length at most $4n$.

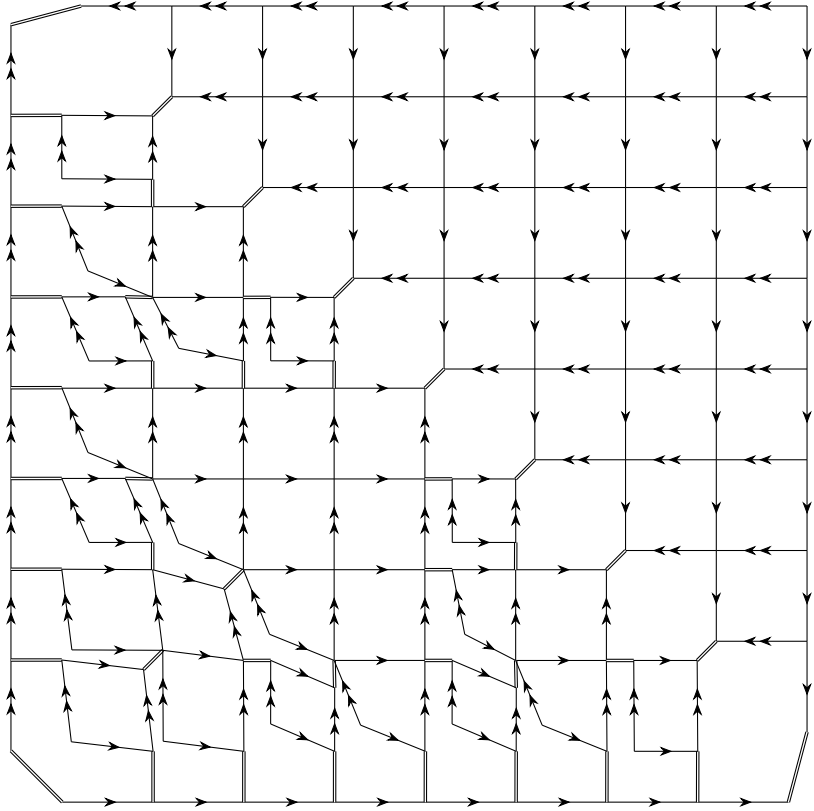
\begin{figure}
    \centering
    \begin{tikzpicture}[scale=1.2]
  \foreach\i in {1,...,8} {
    \pgfmathparse{\i-1}\foreach\j in {0,...,\pgfmathresult} {
      \draw [yarrow] (\i,8-\j) -- (\i-1,8-\j);
      \draw [xarrow] (8-\j,\i) -- (8-\j,\i-1);
    }
  }
    
  \foreach\i in {0,...,8} {
    \coordinate (a0\i) at ($(-165:0.8) + (0,\i)$);
    \coordinate (b\i0) at ($(-105:0.8) + (\i,0)$);
  }
  \foreach\i/\j in {0/1,1/2,2/3,3/4,4/5,5/6,6/7,7/8} {
    \draw [yarrow] (a0\i) -- (a0\j);
    \draw [xarrow] (b\i0) -- (b\j0);
  }
  \foreach\i in {1,...,7} {
    \coordinate (b0\i) at ($(a0\i) + (0.56,0)$); \draw [aarrow] (a0\i) -- (b0\i);
    \coordinate (c\i0) at ($(b\i0) + (0,0.56)$); \draw [aarrow] (b\i0) -- (c\i0);
    \coordinate (d\i) at ($(\i,8-\i) + (-135:0.3)$); \draw[aarrow] (\i,8-\i) -- (d\i);
  }
  \foreach\i in {1,...,4} {
    \pgfmathparse{\i-1}\foreach\j in {0,...,\pgfmathresult} {
      \draw [yarrow] ($(d4) + (\i-4,-\j-1)$) -- +(0,1);
      \draw [xarrow] ($(d4) + (\i-5,-\j)$) -- +(1,0);
    }
  }
  \draw [xarrow] (b06) -- ($(d2) + (-1,0)$);
  \draw [xarrow] ($(d2) + (-1,0)$) -- +(1,0);
  \draw [yarrow] ($(d2) + (0,-1)$) -- +(0,1);
  \draw [yarrow] ($(d2) + (-1,-1)$) -- +(0,1);
  \draw [xarrow] ($(d2) + (-1,-1)$) -- +(1,0);

  \draw [xarrow] ($(d6) + (-1,0)$) -- +(1,0);
  \draw [xarrow] ($(d6) + (-1,-1)$) -- +(1,0);
  \draw [yarrow] ($(d6) + (-1,-1)$) -- +(0,1);
  \draw [yarrow] (c60) -- ($(d6) + (0,-1)$);
  \draw [yarrow] ($(d6) + (0,-1)$) -- (d6);
  
  \coordinate (y11) at ($(0.9,0.9)$);
  \coordinate (z11) at ($(y11) + (-135:0.3)$);
  \draw [aarrow] (a00) -- (b00) (a08) -- (0,8) (b80) -- (8,0) (y11) -- (z11);
  \draw [xarrow] (b07) -- (d1);
  \draw [xarrow] (b01) -- (z11);
  \draw [yarrow] (c70) -- (d7);
  \draw [yarrow] (c10) -- (z11);
  
  \foreach\i in {1,2,3} {
    \coordinate (e\i) at ($(d4) + (\i-4,0.3)$);
    \draw [aarrow] ($(d4) + (\i-4,0)$) -- (e\i);
    \coordinate (f\i) at ($(d4) + (0.3,\i-4)$);
    \draw [aarrow] ($(d4) + (0,\i-4)$) -- (f\i);
  }
  \draw [yarrow] (e2) -- ($(d2) + (0,-1)$);
  \coordinate (b17) at ($(d2) + (0,-1) + (0.3,0)$);
  \draw [aarrow] ($(d2) + (0,-1)$) -- (b17);
  \draw [xarrow] (b17) -- (d3);
  \draw [yarrow] (e3) -- (d3);
  \draw [aarrow] ($(d2) + (-1,0)$) -- +(0,0.3);
  \coordinate (b16) at ($(d2) + (-1,0.3)$);
  \draw [yarrow] (b16) -- (d1);
  \coordinate (b15) at ($(b05) + (0.7,0)$);
  \draw [xarrow] (b05) -- (b15);
  \draw [aarrow] (b15) -- ($(d2) + (-1,-1)$);
  \draw [yarrow] (e1) -- (b15);
  \coordinate (b13) at ($(b03) + (0.7,0)$);
  \draw [xarrow] (b03) -- (b13);
  \draw [aarrow] (b13) -- ($(d4) + (-3,-1)$);
  \coordinate (b12) at ($(d4) + (-3,-2)$);
  \draw [xarrow] (b02) -- (b12);
  \draw [yarrow] (y11) -- (b12);
  \coordinate (b12e) at ($(b12) + (0,0.3)$);
  \draw [aarrow] (b12e) -- (b12);
  \draw [yarrow] (b12e) -- (b13);
  \coordinate (x2) at ($(d4) + (-2,-2) + (-135:0.3)$);
  \draw [xarrow] ($(d4) + (-3,-2)$) -- (x2);
  \draw [aarrow] (x2) -- ($(d4) + (-2,-2)$);

  \draw [xarrow] (f2) -- ($(d6) + (-1,0)$);
  \coordinate (b71) at ($(d6) + (-1,0.3)$);
  \draw [aarrow] ($(d6) + (-1,0)$) -- (b71);
  \draw [yarrow] (b71) -- (d5);
  \draw [xarrow] (f3) -- (d5);
  \coordinate (b61) at ($(d6) + (0,-1) + (0.3,0)$);
  \draw [aarrow] ($(d6) + (0,-1)$) -- (b61);
  \draw [xarrow] ($(b61)$) -- (d7);
  \coordinate (b51) at ($(c50) + (0,0.7)$);
  \draw [yarrow] (c50) -- (b51);
  \draw [aarrow] (b51) -- ($(d6) + (-1,-1)$);
  \draw [xarrow] (f1) -- (b51);
  \coordinate (b31) at ($(c30) + (0,0.7)$);
  \draw [yarrow] (c30) -- (b31);
  \draw [aarrow] (b31) -- ($(d4) + (-1,-3)$);
  \coordinate (b21) at ($(d4) + (-2,-3)$);
  \draw [yarrow] (c20) -- (b21);
  \draw [xarrow] (y11) -- (b21);
  \coordinate (b21e) at ($(b21) + (0.3,0)$);
  \draw [aarrow] (b21e) -- (b21);
  \draw [xarrow] (b21e) -- (b31);
  \draw [yarrow] ($(d4) + (-2,-3)$) -- (x2);

  \closerel{(c10)}{(z11)}{(b01)}
  \closerel{(y11)}{(b12)}{(b02)}
  \closerel{(b12e)}{(b13)}{(b03)}
  \closerel{(e1)}{(b15)}{(b05)}
  \closerel{(b16)}{(d1)}{(b07)}
  \closerel{(e3)}{(d3)}{(b17)}

  \closerel{(c30)}{(b31)}{(b21e)}
  \closerel{(c50)}{(b51)}{(f1)}
  \closerel{(c70)}{(d7)}{(b61)}
  \closerel{(b71)}{(d5)}{(f3)}
  \closerel{(c20)}{(b21)}{(y11)}
  
  \straightrel{($(d2) + (-1,-1)$)}{(b06)}
  \straightrel{(c60)}{($(d6) + (-1,-1)$)}
  \straightrel{($(d4) + (-3,-1)$)}{(b04)}
  \straightrel{(e2)}{($(d2) + (-1,-1)$)}
  \straightrel{(c40)}{($(d4) + (-1,-3)$)}
  \straightrel{($(d6) + (-1,-1)$)}{(f2)}
  \straightrel{($(d4) + (-1,-3)$)}{($(d4) + (-2,-2)$)}
  \straightrel{($(d4) + (-2,-2)$)}{($(d4) + (-3,-1)$)}
\end{tikzpicture}
    \caption{A van Kampen diagram for the element $a x^8 
    a x^{-8} y^{8} a y^{-8}$ in the group $\Gamma$. Generators $x,y,a$ are respectively shown with single arrow, double arrow and double line.}
    \label{fig:vk}
\end{figure}

The Dehn function of a finitely presented group depends only mildly on the presentation, and is a crude measure of the complexity of the word problem (namely, the problem of testing whether two words define equal elements of the group). For example, some of the most computationally tractable groups, Gromov's \emph{word hyperbolic groups}, may be characterised by the fact that their Dehn function is bounded by a linear function. There exist finitely presented groups with unsolvable word problem, and therefore with uncomputable Dehn function.

There is nothing sacred about the Dehn function, and in fact many other invariants have been considered; see~\cite{Gromov93}*{\S5} for such a collection. An interesting invariant defined there is the \emph{filling length}, also considered by Gersten in~\cite{Gersten95}. The \emph{filling length} of a van Kampen diagram $\mathcal K$ with basepoint $*$ is the minimal $\lambda$ such that there exists a sequence $\gamma_0,\dots,\gamma_n$ of loops at $*$ in $\mathcal K$, such that
\begin{enumerate}
    \item all $\gamma_i$ have length at most $\lambda$;
    \item $\gamma_0$ is the trivial loop at $*$, and $\gamma_n$ is the boundary loop of $\mathcal K$;
    \item $\gamma_i$ and $\gamma_{i+1}$ differ by exactly one cell, which is outside $\gamma_i$ and inside $\gamma_{i+1}$ with winding number $1$
\end{enumerate}
(such a sequence of loops is called a \emph{combinatorial null-homotopy}). As before, one defines the \emph{filling length} of a word $w$ representing the identity as the minimal filling length of van Kampen diagrams for $w$, and the \emph{filling function} $\lambda(n)$ as the maximum of $\lambda(w)$ over all words $w$ of length at most $n$ that represent the identity.

We may refine it as follows: firstly, for a generator $s$, we let $|w|_s$ denote the number of $s^{\pm1}$ in a freely reduced word $w$; then, for a van Kampen diagram $\mathcal K$, we let $\lambda_s(\mathcal K)$ be the minimal number $\lambda$ such that there exists a combinatorial null-homotopy $(\gamma_0,\dots,\gamma_n)$ with $|\gamma_i|_s\le\lambda$ for all $i$, and as before $\lambda_s(w)$ as the minimum of $\lambda_s(\mathcal K)$ over van Kampen diagrams for $w$, and $\lambda_s(n)$ as the maximum of $\lambda_s(w)$ over all words $w$ of length at most $n$ that represent the identity.

Secondly, we denote by $\gamma_i^\circ$ the part of $\gamma_i$ that is not on the boundary of the van Kampen diagram (a ``chord''). For example, van Kampen diagrams in groups admitting a free subgroup of finite index are thin, so they admit null homotopies of bounded length. We let $\lambda^\circ_s(\mathcal K)$ be the minimum, over combinatorial null-homotopies, of $\max_i|\gamma_i^\circ|_s$, and define $\lambda^\circ_s(w)$ and $\lambda^\circ_s(n)$ as above.

We return to the finitely presented group $\Gamma$ of \eqref{eq:Baumslag}. Kassabov and Riley prove in~\cite{Kassabov12} that its Dehn function satisfies $\Delta(n)\precsim n^4$, and note that the actual value $\Delta(n)\approx n^2$ follows from~\cites{Bartholdi08,Drutu04}.

To better understand van Kampen diagrams for $\Gamma$, we note that every word $w$ over $\{a,x^{\pm1},y^{\pm1}\}$ that represents the identity has an associated lamp configuration: if $w=w_1\dots w_n$, start at the origin and read the letters $w_1,\dots,w_n$ in sequence, moving left/right/up/down when the letter is $x^{-1},x,y,y^{-1}$ respectively, and toggling the lamp at the current position when the letter is $a^{\pm1}$. After the last letter has been read, we have returned to the origin. The \emph{configuration diameter} of $w$ is the diameter of the associated lamp configuration, namely, the maximal distance between two ON lamps, in the $\ell^\infty$ metric from~\eqref{eq:ball}.

For example, the word $a x^{-2^n} a x^{2^n}y^{-2^n} a y^{2^n}$ represents the identity in $\Gamma$, and has configuration diameter $2^n$ with three ON lamps. Figure~\ref{fig:vk} shows a van Kampen diagram for this word with $n=3$.

We obtain the following consequence of Theorem~\ref{thm:lower}:
\begin{corollary}
  In the group $\Gamma$, the $a$-filling length function $\lambda^\circ_a(n)$ is bounded from below by $\log_2(n)$.
\end{corollary}
\begin{proof}
  Consider the word $w=a x^{-2^n} a x^{2^n}y^{-2^n} a y^{2^n}$ mentioned above, which has size $\mathcal O(2^n)$ and represents the identity in $\Gamma$. Let $\mathcal K$ be a van Kampen diagram for $w_n$, and let $(\gamma_0,\dots,\gamma_\Delta)$ be a combinatorial null-homotopy: a sequence of loops in $\mathcal K$, based at the start of $w$, such that $\gamma_0$ is the empty loop, $\gamma_\Delta$ in the peripheral loop reading $w$, and each $\gamma_i$ differs from $\gamma_{i-1}$ by enclosing a single extra cell $C_i$. If $u_i$ is the word read along $\gamma_i$ from the basepoint to the beginning of the relator $r_i$ around $C_i$, we obtain a corresponding expression of $w_n$ as a product of conjugates of relators $r=a x a x^{-1} y a y^{-1}$ and $s=[x,y]$,
  \[w_n=\prod_{i=1}^\Delta u_i r_i u_i^{-1},\quad r_i\in\{r^{\pm1},s^{\pm1}\}.\]
  Consider now a loop $\gamma_i$; the word read along it also represents the identity in $\Gamma$. Consider the associated lamp configuration: follow the $x^{\pm1},y^{\pm1}$ letters of $\gamma_i$, and flip the lamp at the current position when an $a$ is read. This lamp configuration $W_i$ is admissible, and the configurations $(W_0,\dots,W_\Delta)$ form a chain. The claim now follows from Theorem~\ref{thm:lower}.
\end{proof}

In an entirely similar manner, we deduce from Theorem~\ref{th:complexity:2} the following result.
\begin{corollary}
  There exist $\alpha_2>0$ and $n_2\in\N$ such that, for any $n\ge n_2$, there exists $w\in\{a,x^{\pm1},y^{\pm1}\}$ with the following properties:
  \begin{itemize}
      \item it represents the identity in $\Gamma$;
      \item it has configuration diameter at most $n$;
      \item it contains $L$ letters $a$ or $a^{-1}$;
      \item its $a$-filling length function $\lambda_a(w)$ is at least $L+n^{\alpha_2}$.\qed
   \end{itemize}
\end{corollary}

We note in passing that the homotopy in the van Kampen diagram contains strictly more information than the chain from Definition~\ref{def:chain}, since it specifies not only which lamps remain ON at an intermediate step of a turning-off, but also in which order and through which lamp switches they are visited. A tight reformulation of Theorem~\ref{thm:lower} in terms of group theory would be homological, rather than homotopical. In particular, a null-homotopy determines a chain, but not conversely.

\section{Applications to ergodic theory and dynamical systems}
\label{sec:Ledrappier}
We start with a brief introduction to the Ledrappier subshift from an ergodic theory perspective. Consider a measure-preserving action of a discrete group $G$ on a probability space $(X,\mu)$, and recall that it is called \emph{$r$-mixing} if for any measurable $A_1,\dots,A_r$ one has
\[\mu\left(\bigcap_{i=1}^r g_i\cdot A_i\right)\to\prod_{i=1}^r\mu(A_i)\text{ as }d(g_i,g_j)\to\infty\;\forall i\ne j;\]
and ``$2$-mixing'' is simply called ``mixing''. The question whether, for $\bbZ$-actions, mixing necessarily implies $r$-mixing for all $r$ is an old open problem in ergodic theory. Ledrappier gave in~\cite{Ledrappier78} an example of mixing $\bbZ^2$-dynamical system which is not $3$-mixing --- this is precisely the system of lamp configurations that are cycles, in our terminology; namely,
\[X=\left\{x\colon\bbZ^2\to\bbF_2\middle| x(m,n)+x(m+1,n)+x(m,n+1)=0\;\forall m,n\right\},\]
with $\mu$ its Haar measure. (To see that it is not $3$-mixing, consider $A_1=A_2=A_3=[x(0,0)=1]$, and $g_1=(0,0),g_2=(2^n,0),g_3=(0,2^n)$. Each of the events $g_i\cdot A_i$ has measure $1/2$, but their intersection is empty.)

In~\cite{Arenas-Carmona08}, Arenas-Carmona, Berend and Bergelson show that the Ledrappier system is ``almost $r$-mixing'' for all $r\ge2$. To explain that notion, let $d_{\mathrm{H}}(C,D)$ denote the Hausdorff distance between two elements or subsets $C,D$ of $\bbZ^2$, and let $d_{\mathrm{H}}(C,\mathscr D)$ denote the minimal distance between $C\subseteq\bbZ^2$ and any element of a collection $\mathscr D$ of subsets of $\bbZ^2$. For $r\ge2$ and $\mathscr M\subseteq G^r$, they define $X$ to be \emph{$r$-mixing modulo $\mathscr M$} if for any measurable $A_1,\dots,A_r$ one has
\[\mu\left(\bigcap_{i=1}^r g_i\cdot A_i\right)\to\prod_{i=1}^r\mu(A_i)\text{ as }d(g_i,g_j)\to\infty\;\forall i\ne j\text{ and }d_{\mathrm{H}}(\{g_1,\dots,g_r\},\mathscr M)\to\infty.\]
Recall from Definition~\ref{def:admissible} that \emph{admissible} configurations are finite sums of translates of the plaquette. In~\cite{Arenas-Carmona08}'s terminology, an admissible configuration with $r$ lamps is called a \emph{special $r$-gon}. Their main result is that, for all $r\ge3$, the Ledrappier subshift is $r$-mixing modulo special $r$-gons. In particular, the Ledrappier shift is $3$-mixing modulo large plaquettes.

This led to a study of the structure of special $r$-gons, in~\cite{Arenas-Carmona08}*{\S7}. The results in Section~\ref{ss:highbarrier} partly answer their questions, as we now show. We first recall:
\begin{theorem}[\cite{Arenas-Carmona08}*{Theorem~7.1}]
  If $X\subset \Z^{2}$ is an admissible set of cardinality $n$, then $X$ is a sum modulo $2$ of at most $n^3$ large plaquettes.
\end{theorem}

Let us concentrate on $d=2$. Let $h(n)$ denote the minimal number of large plaquettes sufficient to represent any admissible set of cardinality at most $n$; so $h(n)\le n^3$. In \cite{Arenas-Carmona08}*{Remark 7.2}, the authors ask for lower bounds on $h(n)$ and construct examples which, conjecturally, give $h(n)\ge 3n-8$. Our construction proves a stronger lower bound:
\begin{theorem}[The function $h(n)$ grows superlinearly] For some $\eta>0$ and any $n$ large enough, we have
  \begin{equation}\label{eq:h}
  h(n)\geq n^{1+\eta}.
  \end{equation}
\end{theorem}
\begin{proof}
By Corollary~\ref{cons:set_exists}, there exists a nonempty admissible configuration
\[
X=\{x_1,x_2,\dots,x_m\}
\]
of at most $n$ points such that any two distinct points have $\rho$-distance at least $r$, where $r\geq c'\ln(n-2)\ge 15+c\ln n$ for some positive constants $c,c'$, as long as $n$ is large enough. Up to repeatedly replacing $X$ by $X\sqcup(X+v)$ for some $v\in\Z^2$ at $\rho$-distance at least $r$ from $X-X$, we may assume $m\ge n/2$.

Since $X$ is admissible, write it as the sum of at most $h(n)$ large plaquettes. For every $i=1,\dots,m$, let $M_i$ be the set of those large plaquettes in the sum that have at least one vertex at $\rho$-distance $<r/3$ from $x_i$. If $r\geq3$, then the sets $M_i$ are pairwise disjoint.

Let $X_i$ be the sum modulo $2$ of all large plaquettes from $M_i$. All other points of $X_i$ are at $\rho$-distance at least $r/3$ from $x_i$.  Applying Lemma~\ref{le:long_sum} to $X_i$ and the point $x_i$, we obtain
\[
|X_i|\geq \frac34\left(\frac32\right)^{r/3}
\geq 4 n^{\eta},
\qquad
\eta\coloneqq\frac{c}{3}\ln\left(\frac32\right)>0.
\]
This implies
\[
|M_i|\geq 2 n^{\eta}.
\]
Summing over $i=1,\dots,m$ and using the disjointness of the sets $M_i$ gives $h(n)\ge 2m\,n^\eta\ge n^{1+\eta}$, as claimed.
\end{proof}

\section{Application to linear error-correcting codes}
\label{sec:coding}

Let us briefly recall the basic terminology from coding theory.
A {\it binary linear code} of length $n$
and dimension $k$ is a $k$-dimensional vector subspace
$C\subseteq\mathbb F_2^n$. 
Its elements are called {\it codewords}.

The {\it Hamming weight} $|v|$ of a vector $v$ is its number of nonzero
coordinates, and the \emph{Hamming distance} between two vectors is
the number of coordinates in which they differ.
The minimum nonzero distance of a linear code is
\[
\delta(C)
=\min_{\substack{u,v\in C\\u\ne v}}|u-v|
=\min_{u\in C\setminus\{0\}}|u|.
\]
The triple $[n,k,\delta(C)]$ is called the {\it parameters} of the code $C$.

A choice of basis for $C$ identifies binary strings of length $k$ with
codewords of length $n$. The resulting reduction in the information
transmission rate by a factor of $n/k$ increases reliability: if fewer
than $\delta(C)/2$ errors occur during transmission, the original
codeword can be recovered uniquely.

Identifying subsets with their indicator functions, the cycle space
$\mathcal C_{d,n}$ from Section~\ref{subsec:cycles} is a binary linear code.
Its length is the number of points in the simplex $\Lambda_{d,n}$, namely $\binom{n+d}{d}$. Its defining relations are linear
parity-check equations, one for each plaquette contained in the
simplex.
Since every configuration on one of the faces of  $\Lambda_{d,n}$ extends uniquely to a
codeword, the dimension of the code is the number of points on that
face, namely $\binom{n+d-1}{d-1}$.
By Lemma~\ref{le:cyc_size}, every nonzero codeword has weight at least $n+1$. It is also easy to find a codeword of weight $n+1$: the characteristic
function of any one-dimensional edge, so the parameters are
$$
\left[\binom{n+d}{d},\binom{n+d-1}{d-1},n+1\right].
$$
For $d=2$, these codewords are Steinhaus triangles, also called
Boolean Pascal triangles, with $n+1$ entries on each side; see
\cites{Brunat18,Malyshev06}.

In dimension two, this code has the same local parity-check
	rule as the {\it Newman--Moore code}: $c_x+c_{x+e_1}+c_{x+e_2}=0.$
The code we study differs from the Newman--Moore code in the shape of the domain
(a triangle rather than a square). Moreover, we impose no boundary
conditions on the codewords.
Note that the dimension of the Newman--Moore code with periodic boundary
conditions depends nontrivially on $n$; see the comment after Theorem~\ref{th:2k:asymptotics}.

The \emph{weight enumerator polynomial} of a binary linear code $C$ of length $n$
is the polynomial

\begin{equation}
	W_{C}(t)=\sum_{x\in C} t^{|x|}.
\end{equation}

In particular, $W_C(0)=1$. 
For the codes considered here,
the weight enumerator is exactly the generating polynomial from~\ref{eq:def:fdn}:
\begin{equation}
	W_{\mathcal C_{d,n}}(t)=G_{d,n}(t).
	\label{eq:code-weight-enumerator}
\end{equation}

Suppose that a codeword $x\in C$ is transmitted through a
\emph{binary symmetric channel} with crossover probability
$0\le p<1/2$: the received word is $y=c+E$, where the coordinates
of the error $E$ are independent Bernoulli random variables with parameter $p$.

To recover $x$ from $y$, we use {\it nearest-neighbour decoding}:
we choose a codeword $D(y) \in C$ that minimizes the Hamming distance to $y$.

Let
$$
P_{\mathrm{err}}(C,p)
=\max_{x\in C}\mathbb P\bigl(D(x+E)\ne x\bigr)
$$
denote the probability
of failing to recover the initial codeword, maximised over the
transmitted codeword.

We use the standard \emph{Bhattacharyya bound}
(see, e.g., \cite{Hall10}*{[Theorem~9.2.6}): for a binary linear code
transmitted through a binary symmetric channel with crossover probability $p$ and decoded by nearest-neighbour decoding,
\begin{equation}
	P_{\mathrm{err}}(C,p)
	\le
	W_C\bigl(2\sqrt{p(1-p)}\bigr)-1.
	\label{eq:bhattacharyya-bound}
\end{equation}

Combined with Theorem~\ref{th:upper}, this estimate yields the following corollary:

\begin{corollary}[Superpolynomial error bound]
	\label{cor:reliable-decoding}
	For every fixed $d\ge1$ and $0\le p<1/2$, nearest-neighbour decoding
	of $\mathcal C_{d,n}$ over the binary symmetric channel with crossover
	probability $p$ has \[P_{\mathrm{err}}(\cC_{d,n},p)\le \exp\left(-\frac{\ln^2n}{\ln \lceil 1/(1-2\sqrt{p(1-p)})\rceil}\right)\] for $n$ large enough.
\end{corollary}

Note that a naive way to compute $D(y)$ involves searching through
an exponentially large set of codewords. It is therefore of interest
to find a decoding algorithm that runs in time polynomial in $n$ and
has as low an error probability as possible.

\section*{Acknowledgments}
L.\ B.\ and I.\ M.\ gratefully acknowledge support from the SNSF Advanced Grant TMAG-2\_216487/1.

We thank Paul Chleboun, Fabio Martinelli and Marius Tiba for helpful discussions, as well as Guillaume Aubrun for sparking this collaboration. We also thank Juan Garrahan for useful bibliographic pointers, particularly \cite{Heringa89}. For open access purposes, the authors have applied a CC BY public copyright license to any author-accepted manuscript version arising from this submission.

\paragraph{Declaration on the use of AI.} No artificial intelligence tools were used in the development of the mathematical ideas, arguments, or proofs presented in this paper, nor were they used for drafting or writing the manuscript.

\bibliographystyle{plain}
\bibliography{Bib}
\end{document}